\documentclass[11pt]{article}
\usepackage{amssymb} \usepackage{amsmath} \usepackage{amsthm}
\usepackage{latexsym} \usepackage{exscale}

\newcommand{\bap}{\begin{app}}
\newcommand{\eap}{\end{app}}
\newcommand{\begs}{\begin{exams}}
\newcommand{\eegs}{\end{exams}}
\newcommand{\beg}{\begin{example}}
\newcommand{\eeg}{\end{exaplem}}
\newcommand{\bpr}{\begin{proposition}}
\newcommand{\epr}{\end{proposition}}
\newcommand{\bt}{\begin{theorem}}
\newcommand{\et}{\end{theorem}}
\newcommand{\bc}{\begin{corollary}}
\newcommand{\ec}{\end{corollary}}
\newcommand{\bl}{\begin{lemma}}
\newcommand{\el}{\end{lemma}}
\newcommand{\bd}{\begin{definition}}
\newcommand{\ed}{\end{definition}}
\newcommand{\brs}{\begin{remarks}}
\newcommand{\ers}{\end{remarks}}

\newcommand{\const}{\text{\rm constant}}

\newcommand{\Span}{{\rm Span }}

\newcommand{\rank}{\text{\rm{rank}}}
\newcommand\kernel{\hbox{\rm Ker}}

\newcommand\br{\begin{remark}}
\newcommand\er{\end{remark}}
\newcommand\bp{\begin{pmatrix}}
\newcommand\ep{\end{pmatrix}}
\newcommand{\be}{\begin{equation}}
\newcommand{\ee}{\end{equation}}
\newcommand\ba{\begin{equation}\begin{aligned}}
\newcommand\ea{\end{aligned}\end{equation}}

\newcommand{\na}{{\nabla}}
\renewcommand{\div}{{\rm div}}
\newtheorem{theorem}{Theorem}[section]
\newtheorem{lemma}[theorem]{Lemma}
\newtheorem{proposition}[theorem]{Proposition}

\newtheorem{remark}[theorem]{Remark}

\newtheorem{definition}{Definition}
\newtheorem{corollary}[theorem]{Corollary}
\newtheorem{example}[theorem]{Example}

\newtheorem{thm}{Theorem}[section]
\newtheorem{prop}[thm]{Proposition}
\newtheorem{cor}[thm]{Corollary}
\newtheorem{lem}[thm]{Lemma}
\newtheorem{defn}[thm]{Definition}
\newtheorem{ass}[thm]{Assumption}

\newtheorem{rem}[thm]{Remark}

\newtheorem{exams}[thm]{Examples}
\newtheorem{notation}[thm]{Notation}
\numberwithin{equation}{section}

\def\({\left(\begin{array}{cccccc}}
\def\){\end{array}\right)}

\newcommand{\beq}{\begin{equation}}
\newcommand{\eeq}{\end{equation}}

\newcommand{\eps}{\varepsilon}
\newcommand{\RR}{\mathbb{R}}

\newcommand{\pf}{\begin{proof}}
\newcommand{\foorp}{\end{proof}}

\newcommand{\NN}{{\mathbb N}}

\newcommand{\CC}{{\mathbb C}}

\newcommand{\EE }{{\mathbb E}}

\newcommand{\cA}{{\mathcal A}}
\newcommand{\cB}{{\mathcal B}}
\newcommand{\cC}{{\mathcal C}}
\newcommand{\cD}{{\mathcal D}}
\newcommand{\cE}{{\mathcal E}}
\newcommand{\cF}{{\mathcal F}}
\newcommand{\cG}{{\mathcal G}}
\newcommand{\cH}{{\mathcal H}}
\newcommand{\cI}{{\mathcal I}}

\newcommand{\cK}{{\mathcal K}}

\newcommand{\cN}{{\mathcal N}}
\newcommand{\cO}{{\mathcal O}}
\newcommand{\cP}{{\mathcal P}}

\newcommand{\cU}{{\mathcal U}}

\newcommand{\uu}{{\underline u}}

\newcommand{\ux}{{\underline x}}

\newcommand{\Id}{{\rm Id }}

\newcommand{\ls}{{\ \lesssim \ }}
\newcommand{\D}{{\partial }}

\newcommand{\bR}{\mathbb{R}}
\newcommand{\bG}{\mathbb{G}}

\newcommand{\bE}{\mathbb{E}}

\newcommand{\bN}{\mathbb{N}}

\newcommand{\bC}{\mathbb{C}}

\newcommand{\mrp}{\mathrm{p}}
\newcommand{\mrh}{\mathrm{h}}

\newcommand{\whgamma}{\widehat{\gamma'}}

\newcommand{\hzeta}{\hat{\zeta}}

\newcommand{\umu}{\underline{\mu}}

\title{\textbf{Viscous boundary layers in hyperbolic-parabolic systems with Neumann boundary conditions}}

\author{\sc \small
Olivier Gues\thanks{LATP, Universit\'e d'Aix-Marseille;
olivier.gues@univ-amu.fr },
Guy M\'etivier\thanks{MAB, Universit\'e de Bordeaux I;
metivier@math.u-bordeaux.fr.
Research of G.M. was partially supported by European
network HYKE,  HPRN-CT-2002-00282. },
Mark Williams\thanks{
University of North Carolina;
williams@email.unc.edu.
Research of M.W. was partially supported by
NSF grants number DMS-0070684 and DMS-0401252.},
Kevin Zumbrun\thanks{Indiana University;
kzumbrun@indiana.edu: K.Z. thanks the Universities of
Bordeaux I and Provence for their hospitality during visits
in which this work was partially carried out.
Research of K.Z. was partially supported by
NSF grants number DMS-0070765 and DMS-0300487. } }

\begin{document}
\date{Revised: \today}
\maketitle
\begin{abstract}\noindent
\emph{\quad}We initiate the study of noncharacteristic
boundary layers in hyperbolic-parabolic problems with Neumann boundary conditions.  More generally, we study boundary layers with mixed Dirichlet--Neumann boundary
conditions where the number of Dirichlet conditions is fewer than the number of hyperbolic
characteristic modes entering the domain,
that is, the number of boundary conditions
needed to specify an outer hyperbolic solution.
We have shown previously that this situation prevents
the usual WKB approximation involving an outer solution with pure Dirichlet
conditions.  It also rules out the usual maximal estimates for the linearization of the hyperbolic-parabolic problem
about the boundary layer.

Here we show that for linear, constant-coefficient, hyperbolic-parabolic problems
 one obtains a reduced hyperbolic problem  satisfying Neumann
or mixed Dirichlet--Neumann rather than Dirichlet boundary conditions.
When this hyperbolic problem can be solved, a unique formal boundary-layer expansion
can be constructed.  In the extreme case of pure Neumann conditions and totally
incoming characteristics, we carry out a full analysis of the quasilinear case,
obtaining a boundary-layer approximation to all orders with a rigorous error analysis. As a  corollary we characterize the small viscosity limit for this problem. The analysis shows that although the associated linearized hyperbolic
and hyperbolic--parabolic problems do not satisfy the usual
maximal estimates for Dirichlet conditions, they do satisfy analogous versions
with losses.
\end{abstract}

\tableofcontents

%
%

\section{Introduction}

\emph{\quad}
In the study of noncharacteristic boundary layers of hyperbolic-parabolic
systems, physical applications motivate the inclusion of Neumann boundary
conditions along with the usual Dirichlet boundary conditions
that have traditionally been considered for such problems
(see, e.g., \cite{GS,R2,R3} and rererences therein).
In particular, as discussed in \cite{NZ1,NZ2,GMWZ5,R},
suction-induced drag reduction along an
airfoil\footnote{ See \cite{S,Br}, or NASA site
http://www.dfrc.nasa.gov/Gallery/photo/F-16XL2/index.html}
 is typically modeled by the compressible Navier--Stokes equations
\begin{equation}
\label{NSeq}
\left\{ \begin{aligned}
 & \D_t \rho +  \div (\rho u) = 0
 \\
 &\D_t(\rho  u) + \div(\rho u^tu)+ \na p =
\eps \mu \Delta u + \eps(\mu+\eta) \nabla \div u
 \\
 &
 \D_t(\rho E) + \div\big( (\rho E  +p)u\big)=
\kappa \Delta T +
\eps \mu \div\big( (u\cdot \nabla) u\big) \\
&
\qquad \qquad \qquad \qquad
\qquad \qquad
+ \eps(\mu+\eta) \nabla(u\cdot \div u)
 \end{aligned}\right.
\end{equation}
on an exterior domain $\Omega$,
with no-slip {\it suction-type} boundary conditions on the velocity,
$
 u_T|_{\partial\Omega}=0$, $u_\nu|_{\partial \Omega}= V(x)< 0,
$
and either prescribed or insulative boundary conditions on
the temperature,
$ T|_{\partial \Omega}= T_{wall}(x)$
or $\D_\nu T|_{\partial \Omega}= 0.
$

The study of such mixed-type boundary layer problems was initiated in
\cite{GMWZ5,GMWZ6} for certain combinations of Dirichlet and Neumann boundary conditions in the viscous problem.  However, the ansatz used there, which assumes that the residual hyperbolic problem should have only Dirichlet boundary conditions,  breaks down when there are too many Neumann conditions in the viscous problem - more precisely,
when there are too few Dirichlet conditions, in the sense that
the number of scalar Dirichlet conditions in the viscous problem is strictly less than the ``correct" number of residual boundary conditions for the hyperbolic problem.
In such cases, the construction in \cite{GMWZ5} of
``$\cC$-manifolds'' of reachable states determining Dirichlet boundary
conditions for the outer, hyperbolic solution fails, due to a lack
of transversality, as a consequence of which (together with the low-frequency
decomposition of \cite{R2}) the maximal linearized estimates used
in \cite{GMWZ5,GMWZ6} to establish rigorous convergence may be
shown to fail as well.
As noted in \cite{R}, the case of \eqref{NSeq} with
incoming supersonic velocity falls into this category, so is not accessible
by the techniques developed up to now.

Clearly, in such cases, a new analysis is required.
Several questions arise, including:

(1) Does the hyperbolic-parabolic problem have a solution on a fixed time interval independent of $\eps$?

(2) Is there a residual hyperbolic problem whose solution gives the small viscosity limit of solutions to the hyperbolic-parabolic problem? In particular, what are the correct residual hyperbolic boundary conditions?
And, are these uniquely determined?

(3) What are the maximal linearized estimates that we may
expect in this context, both for the
residual hyperbolic and full hyperbolic--parabolic problem?

\medskip
In this paper, we answer these questions completely in the extreme
case of pure Neumann boundary conditions and totally incoming hyperbolic
characteristic modes, showing that there is a reduced hyperbolic
problem with Neumann instead of Dirichlet conditions, and that in place of
the standard Dirichlet-type linearized estimates for the reduced hyperbolic
and full hyperbolic--parabolic systems, there hold modified versions with
losses, sufficient to close a rigorous convergence argument. As a corollary we characterize the small viscosity limit for the quasilinear problem.

In the general, linear constant-coefficient case, we present two approaches to constructing a
formal boundary-layer expansion to all orders of the solution to the hyperbolic-parabolic problem.  In general the reduced
hyperbolic (outer) problem features mixed Dirichlet--Neumann boundary conditions.   In the pure Neumann case we prove that the exact and approximate solutions to the hyperbolic-parabolic problem are close when $\eps$ is small.

{\it Our results
motivate the further study of first-order hyperbolic initial-boundary-value
problems with  Neumann or mixed Neumann--Dirichlet boundary
conditions}.   This is at first sight a counterintuitive problem,
since the normal derivative on the boundary is not controlled by
the usual hyperbolic solution theory,
and it does not seem to have received much attention before now.
We regard this as one of the most interesting aspects of the analysis.

\subsection{Linear systems with Neumann boundary conditions}\label{s:linear}

  \emph{\quad}  First we examine a linear problem for which the above
questions have a positive, and rather simple, answer.
Let us consider the parabolic boundary value problem on $\overline{\bR}^{d+1}_+:=\{x=(x',x_d)=(x_0,x'',x_d)\in\bR^{d+1}:x_d\geq 0\}$:
\begin{eqnarray}
  Lu &=& f+ \eps \Delta_x u \ \mathrm{in} \ \{x_d >0\}, \label{0,1}\\
  \D_du_{|x_d=0 } &=& 0 ,\label{0,2} \\
  u_{|t<0}&=& 0 \label{0,3},
\end{eqnarray}
where $L$ is a symmetric hyperbolic operator with constant
coefficients
$$
L = \D_t + \sum_{j=1}^d A_j \D_j,\;  t=x_0
$$
and $f\in H^\infty(\overline{\RR}_+^{1+d})$ with $f_{|t<0}=0$. The $N\times N$ matrices $A_j$
are constant (for now), and the boundary is noncharacteristic:
$$
\det A_d \ne 0.
$$

We look for an approximate solution of the form
$$
u^\eps (x)= u_0(x) + \eps u_1(x,\frac{x_d}{\eps}) + \eps^2
u_2(x,\frac{x_d}{\eps})+ \dots
$$
with the usual profiles
$$
u_j(x,z)= \uu_j(x) + u_j^*(x',z), \quad j \geq 1,
$$
where $\uu_j$ is an ``outer'' solution, and $u^*_j$ is a boundary
layer profile which goes to $0$ as $z \rightarrow \infty$.

\begin{rem}\label{no0}
\textup{
One could postulate a more general profile
$u_0(x,z)= \uu_0(x) + u_0^*(x',z)$
at level $j=0$; however, the resulting $\eps^{-1}$ order
profile equations
$A_d \D_z u_0^* - \D_z^2 u_0^* =0,$
with boundary condition
$ \D_z (u_0^*)_{|z=0} = 0$ would give then $ \D_z u_0^*\equiv 0$,
recovering the assumption $u_0=u_0(x)$.
}
\end{rem}

The profile equation obtained at the order $\eps^0$ is
$$
Lu_0 + A_d \D_z u_1 - \D_z^2 u_1 = f.
$$
which leads to the two equations for $u_0$ and $u^*_1$:
\begin{equation}\label{1a}
    Lu_0 =f
\end{equation}
and
\begin{equation}\label{2}
A_d \D_z u_1^* - \D_z^2 u_1^* =0.
\end{equation}
The boundary condition (\ref{0,2}) gives at the order $\eps^0$:
$$
(\D_d u_0)_{|x_d=0} + (\D_z u_1^*)_{|z=0} = 0.
$$
Hence the solution to the boundary layer equation (\ref{2}) is
\begin{equation}\label{3a}
    u_1^*(x',z)= - e^{z A_d} \, A_d^{-1} \D_du_0(x',0).
\end{equation}
It follows  that $ u_1^*$ is decreasing at $+\infty$ if and only if
$\D_du_0|_{x_d=0}$ lies in  $\EE_-(A_d)$, the negative eigenspace of $A_d$:
\begin{equation}\label{5}
    \D_d {u_0}_{|x_d=0} \in \EE_-(A_d).
\end{equation}
But $u_0$ satisfies $Lu_0=f$; thus
$$
\D_d u_0 = -A_d^{-1}\sum^{d-1}_0 A_j \D_j u_0 + A_d^{-1} f
$$
and the condition (\ref{5}) is equivalent to
\begin{equation}\label{6a}
    H {u_0}_{|x_d=0} \in A_d^{-1} f|_{x_d=0} +\EE_-(A_d),
\end{equation}
where $H$ is the tangential operator $H:= A_d^{-1}\sum_0^{d-1} A_j
\D_j$.   So we are led to solve the mixed problem
\begin{eqnarray}
  Lu_0 &=& f\ \mathrm{in} \ \{x_d >0\},\label{15}\\
  H {u_0}_{|x_d=0 } &\in & A_d^{-1}f_{|x_d=0} + \EE_-(A_d) ,\label{16}\\
  {u_0}_{|t<0}&=& 0.\label{17}
\end{eqnarray}
(The boundary conditions may be rephrased via projections
as described in Remark \ref{unusual}.)

To solve this problem introduce the unknown $v:= Hu_0$, which is the solution of the
symmetric hyperbolic problem with dissipative boundary conditions
\begin{eqnarray}
  Hv + \D_d v &=& H (A_d^{-1}f)\ \mathrm{in} \ \{x_d >0\},\\
  v_{|x_d=0 } &\in & A_d^{-1}f_{|x_d=0} + \EE_-(A_d) ,\\
  {v}_{|t<0}&=& 0.
\end{eqnarray}
Hence $v$ is completely determined; thus $u_0$ is also uniquely
determined as the unique solution of
$$ Hu_0 = v, \quad {u_0}_{|t<0}
=0 $$
(here considered as an initial-value problem defined on slices $x_d\equiv \const$).
 Then $u^*_1$ is uniquely determined by formula $(\ref{3a})$,
and decays to zero at $+\infty$.

The construction follows the same pattern for the
next terms. For example, setting $L'=\D_t + \sum_1^{d-1} A_j \D_j$ we obtain
at the order $\eps^1$ the profile equation
$$
L\underline{u}_1 + L'u^*_1+A_d \D_z u_2 - \D_z^2 u_2 = \Delta u_0.
$$
which leads to the two equations for $\uu_1$ and $u^*_2$:
\begin{equation}
    L\uu_1 = \Delta u_0
\end{equation}
and
\begin{equation}
A_d \D_z u_2^* - \D_z^2 u_2^*  = - L'u_1^*.
\end{equation}
The boundary condition (\ref{0,2}) gives at the order $\eps^1$:
$$
(\D_d \uu_1)_{|x_d=0} + (\D_z u_2^*)_{|z=0} = 0.
$$
One can solve as before these equations which gives a unique solution
for $\uu_1$ and $u_2^*$.

\begin{thm}
$
u^\eps(x) = u_0(x) + \eps u_1(x, x_d/\eps) + \cdots +
 \eps^k \, u_k(x,x_d/\eps) + O(\eps^k)
$
in $L^2((-\infty, T]\times \RR^d_+)$ for all given $T>0$ and all
$k\in\NN$ as $\eps \rightarrow 0$.
\end{thm}

\begin{proof}
Since we can construct an approximate solution to any order, it is
sufficient to prove an estimate of $|u|_{L^2(\Omega_T)}$,
where $\Omega= (-\infty,T] \times \RR^d_+$, for the solution $u$ to
the problem (\ref{0,1})(\ref{0,2})(\ref{0,3}). First we estimate the
normal derivative. Applying $\D_d$ to the equation $(\ref{0,1})$ and
using condition $(\ref{0,2})$ leads to a
hyperbolic--parabolic problem with a
homogenous Dirichlet boundary condition for $\D_d u$. A simple integration by
parts yields (with $|u|_{\gamma} = |e^{-\gamma t}
u|_{L^2(\Omega_T)}$):
$$
\eps | \nabla_x \D_d u|_{\gamma}^2 + \gamma | \D_du|^2_\gamma
\ls \gamma^{-1} | \D_df|_\gamma^2.
$$
Going back to the system (\ref{0,1}), taking the product on the left
by $u$, and integrating by parts leads to
$$
\eps | \nabla u |^2_\gamma + \gamma |u|^2_\gamma \ls
\gamma^{-1} |f|_\gamma^2 + |u|_\gamma |\D_du|_\gamma.
$$
Hence using the previous estimate one gets
$$
\eps |\nabla u|_\gamma^2 + \gamma |u|^2_\gamma \ls
\gamma^{-1} |f|_\gamma^2 + \gamma ^{-3} |\D_df|_\gamma^2,
$$
and finally
\begin{equation}\label{18}
    |u|_{\gamma} \ls \gamma^{-1} |f|_\gamma + \gamma^{-2}
    |\D_df|_\gamma.
\end{equation}
Applying the estimate (\ref{18}) to the error $w= u^\eps -
u^\eps_{approx}$, with the function $f$ replaced by $O(\eps^r)$ and
$\D_d f$ replaced by $O(\eps ^{r-1})$
for $r$  chosen large enough
(i.e., $r\ge 2$),
proves the theorem.
\end{proof}

An analogous result with convergence in $L^2$ replaced by convergence in $L^\infty$ can easily be obtained after getting higher derivative estimates.

\begin{rem}\label{serrermk}
\textup{
The approach followed here is similar to the idea of ``filtering'' introduced by
Serre \cite{Se1} in the somewhat different context of second-order
hyperbolic problems with variational structure,\footnote{
Also featuring Neumann, or ``free,'' boundary conditions.}
in which a degenerate
problem is decomposed into the composition of problems of standard
type, each inducing its own losses/gains.
}
\end{rem}

\subsection{Quasilinear systems with Neumann boundary conditions}

\emph{\quad} Next we derive a candidate for the residual hyperbolic problem in the quasilinear case.
Consider the nonlinear
parabolic problem
\begin{eqnarray}\label{orig}
  L_u(u)  &=& f+ \eps \Delta_x u \ \mathrm{in} \ \{x_d >0\},\\
  \D_du_{|x_d=0 } &=& 0 , \\
  u_{|t<0}&=& 0.
\end{eqnarray}
where $L_u$ is a symmetric hyperbolic operator
$$
L_u = \D_t + \sum_1^d A_j(u) \D_j,
$$
and $f\in H^\infty(\overline{\RR}_+^{1+d})$ with $f_{|t<0}=0$. The matrices $A_j$
are smooth and symmetric, and the boundary is noncharacteristic:
$$
\det A_d (u)\ne 0, \quad \forall u\in \RR^N.
$$
Again we expect an expansion of the form
$$
u^\eps (x)= u_0(x) + \eps u_1(x,\frac{x_d}{\eps}) + \eps^2
u_2(x,\frac{x_d}{\eps})+\dots,
$$
that is, a ``weak'' layer of order $\eps$ in amplitude.
This may be deduced exactly as in the linear constant-coefficient
case, by examination of the order $\eps^{-1}$ profile equations
as described in Remark \ref{no0}.

The equations for the terms of order $\eps ^0$ give
$$
L_{u_0} u_0 + A_d(u_0)\D_z u_1^* - \D_z^2 u_1^* = f.
$$
This equation splits into two parts
\begin{equation}\label{7}
A_d({u_0}_{|x_d=0})\D_z u_1^* - \D_z^2 u_1^* = 0
\end{equation}
and
\begin{equation}\label{8}
L_{u_0} u_0 = f,
\end{equation}
and the boundary condition at the order $\eps^0$ is still
\begin{equation}\label{10}
    (\D_du_0)_{|x_d=0} + (\D_z u_1^*)_{|z=0}.
\end{equation}
 The solution to the boundary layer
equation (\ref{7}) is
\begin{equation}\label{9}
u_1^*(x',z) = - e^{z A_d(u_0(x',0))}A_d^{-1}\big(u_0(x',0)\big)
\D_d u_0(x',0).
\end{equation}
This solution goes to $0$ at $+\infty$ if and only if
\begin{equation}\label{11}
    \D_d u_0(x',0) \in \EE_-\big(A_d\big(u_0(x',0)\big)\big).
\end{equation}
Using the equation (\ref{8}) we rewrite this condition:
\begin{equation*}
    H_{u_0} \big(u_0\big) \in A_d^{-1} \big(u_0(x',0)\big) f_{|x_d=0} +
    \EE_-\big(A_d\big(u_0(x',0)\big)\big).
\end{equation*}
with $H_u := A_d(u)^{-1} L_u - \D_d$. Writing instead
$$
L'_u := \D_t + \sum_1^{d-1} A_j(u) \D_j.
$$
we obtain the  following hyperbolic boundary problem obtained for $u_0$:
\begin{eqnarray}\label{13}
  L_u (u)  &=& f \ \quad \mathrm{in}\; (-\infty, T] \times \RR^d_+ \\
  L'_u(u) &\in & f_{|x_d=0} +
    \EE_-\big(A_d(u)\big) \  \quad  \mathrm{on} \ \{x_d=0\} ,\label{14}\\
  u_{|t<0} &=& 0.
\end{eqnarray}

\begin{rem}\label{unusual}
\textup{We do not know if this  problem is  well-posed in general. The boundary conditions
(\ref{14}) are unusual; they can be as rephrased as
\begin{align}\label{13a}
\begin{split}
&\pi_+(A_d(u))\left(L'_u(u)- f_{|x_d=0}\right)=0\;\mathrm{on } \; \{x_d=0\} \text{ or }\\
&\pi_+(A_d(u))\partial_du=0\;\mathrm{on } \; \{x_d=0\}
\end{split}
\end{align}
(equivalently, $\pi_+\partial_d u=0$),
where $\pi_+(A(u))$ is the projection onto $E_+(A_d(u))$ along $E_-(A_d(u))$.  Yet,  in the constant coefficient linear case the
corresponding problem (\ref{15}), (\ref{16}), (\ref{17}) turns out to have a unique natural solution.}

\textup{In the totally incoming case where $A_d(u)>0$ and thus $E_-(A_d(u))=0$, one can solve \eqref{13a} by first solving a hyperbolic system on the
boundary, as we describe further below.
A high-order approximate solution to the hyperbolic-parabolic problem \eqref{orig} can be constructed, and the small viscosity limit can be completely analyzed.}
\end{rem}

\subsection{Assumptions and main result.}\label{s:result}

\emph{\quad} Our main result treats a quasilinear hyperbolic-parabolic problem where the
questions posed at the beginning  can be answered completely, the case where all characteristics for the hyperbolic problem are incoming: $A_d(u)>0$.  We study the forward problem on $\overline{\bR}^{d+1}_+:=\{x=(x',x_d)=(x_0,x'',x_d)\in\bR^{d+1}:x_d\geq 0\}$:
\begin{align}\label{aa}
\begin{split}
&\cE(u_\eps):=\sum_{j=0}^dA_j(u)\partial_{x_j}u-\eps\Delta u=f\\
&\partial_{x_d}u|_{x_d=0}=0\\
&u=0\text{ in }x_0 <0
\end{split}
\end{align}
where the $A_j$ are $N\times N$ matrices (not necessarily symmetric), $A_d(u)>0$, and $A_0=I$.

The approximate solution, which is  constructed in section \ref{approx}, has the form
\begin{align}\label{ab}
u^a_\eps(x)=u^0(x)+\eps u^1(x)+\dots +\eps^M u^M(x)
\end{align}
and satisfies
\begin{align}\label{ac}
\begin{split}
&\cE(u^a):=\sum_{j=0}^d A_j(u^a)\partial_{x_j}u^a-\eps\Delta u^a=f+\eps^M R_{\eps}\\
&\partial_{x_d}u^a|_{x_d=0}=0\\
&u^a=0\text{ in }x_0 <0.
\end{split}
\end{align}

As a consequence of the totally incoming assumption, there is no fast transition layer in $u^a$. Nevertheless, the nonlinear stability of $u^a$ and the analysis of the small viscosity limit turn out to be delicate questions, because the Evans function for this problem vanishes at zero frequency.  Thus, $u^a$ can be expected to be at best ``weakly stable".

The low frequency Evans function is computed explicitly in section \ref{Evans} and its degeneracy near $0$  is precisely estimated.\footnote{Outside a neighborhood of zero frequency, the Evans function is nonvanishing by
\eqref{18}; recall that the layer in the totally incoming case is constant,
so the analysis of Section \ref{s:linear} applies.}
This estimate allows us to construct  degenerate Kreiss symmetrizers at the symbol level  in section \ref{symm}, and these symmetrizers are used there to prove resolvent estimates for the frozen coefficient linearized problem.\footnote{Degenerate symmetrizers were used also in \cite{GMWZ2}, but there the degeneracy occurred in the elliptic bloc ($S_P$ in \eqref{b1}), rather than the hyperbolic block.}

The resolvent estimates are quantized in section \ref{basic} using the pseudodifferential calculi outlined in the Appendix.   This section provides the main variable coefficient $L^2$
 estimate, Theorem \ref{c0c},  for the problem obtained by linearizing the original system \eqref{aa} around the approximate solution $u^a$.
 Fortunately, the $L^2$ estimate exhibits no loss of derivatives, but there is a  loss of a factor of $\sqrt{\eps}$ when the boundary datum $g=0$.
 This loss in the main estimate, which reflects the degeneracy in the Evans function,  is the source of most of the technical difficulties in the paper, because it prevents us from absorbing  terms that would otherwise be  absorbed easily as ``error terms" in the estimates.

Higher derivative estimates are proved in section \ref{higher} using an appropriate enlarged system, and these estimates are then used in section \ref{nonlinear} to solve the nonlinear error equation satisfied by $u_\eps-u^a$ by Picard iteration.

We let $\Omega_T:=\{x=(x',x_d)=(x_0,x'',x_d)\in\bR^{d+1}:x_d\geq 0, \; x_0\leq T\}$ and sometimes write $t=x_0$.

  \begin{ass}\label{assumptions}
  I.) The $N\times N$ matrices $A_j(u)$ in the system \eqref{aa} are $C^\infty$ and symmetric, $A_0=I$, and $A_d(u)>0$.  Thus, in particular the boundary is
noncharacteristic.

  II.) Let $f\in H^s(\overline{\bR}^{d+1}_+)$ for $s$ large (as in Theorem \ref{e11.8}), $f=0$ in $t<0$, and let $u_0(x)\in\Omega_{T_0}$ denote the solution to the residual hyperbolic problem:
    \begin{align}\label{res}
\begin{split}
&\partial_tu_0+\sum^d_{j=1}A_j(u_0)\partial_ju_0=f\text{ in }x_d>0\\
&\partial_{d}u_0|_{x_d=0}=0\\
&u_0=0 \text{ in }t<0.
\end{split}
\end{align}
    Assume that for $x\in \Omega_{T_0}$ the function $u_0$ takes values in a neighborhood of $0$, $\cU$, such that for $u\in \cU$, the hyperbolic operator $\partial_t+\sum^d_{j=1}A_j(u)\partial_j$ has semisimple characteristics of constant multiplicity.

  \end{ass}

\begin{rem}\label{r1}\textup{The positivity of $A_d$ implies that the boundary condition in \eqref{res} agrees with \eqref{13a}. Assumption II is a familiar condition implying that  the hyperbolic system satisfies the ``block structure" condition first formulated by Kreiss \cite{K} for constructing symmetrizers.   We could replace Assumption II by other weaker assumptions that imply block structure.  We could also require that such an assumption holds only for $x$ near $x_d=0$ with only minor changes in the proofs.}
\end{rem}

\begin{thm}\label{ev1.1}
Under Assumption \ref{assumptions}  there exists an $\epsilon_0$ such that for $0<\epsilon\leq \epsilon_0$ the parabolic problem \eqref{aa} has an exact solution $u_\eps$ on $\Omega_{T_0}$ of the form \begin{align}\label{ev1.2}
u^\epsilon(x)=u^a_\eps+\epsilon^L v_\eps,
\end{align}
where $u^a_\eps$ has the expansion \eqref{ab} in which the leading term is the solution $u_0$ to the residual hyperbolic problem \eqref{res}.
The exponent $L$ can be chosen as large as desired provided the approximate solution is constructed with sufficiently many terms ($M(L)$) and in that case we have: \begin{align}\label{ev1.3}
|\partial^\alpha(v_\eps,\epsilon\partial_d v_\eps)|_{L^\infty}\leq 1 \end{align}
for $|\alpha|\leq L,\; 0<\epsilon\leq\epsilon_0$.  Here $\partial=(\partial_0,\dots,\partial_{d-1})$.

\end{thm}
This Theorem is an immediate corollary of the more precisely stated Theorem \ref{e11.8}, which is phrased in terms of $U=(v,\eps\partial_dv)$.

\begin{cor}[Small viscosity limits]\label{svlim}
Let $u_\eps$ be the solution to the hyperbolic-parabolic system \eqref{aa}, $u^a_\eps$ the approximate solution \eqref{ab} to that system, and  $u_0$ the solution to the residual hyperbolic problem \eqref{res}. Then
\begin{align}\label{lim}
\begin{split}
&|u_\eps-u^a_\eps|_{L^\infty(\Omega_{T_0})}\leq C\eps^L\\
&|u_\eps-u_0|_{L^\infty(\Omega_{T_0})}\leq C\eps.
\end{split}
\end{align}
\end{cor}

\subsection{Mixed boundary conditions: toward a general theory}\label{s:mixed}
\emph{\quad} We conclude with a discussion of the case of mixed Dirichlet--Neumann
boundary conditions in the linear constant-coefficient case,
making contact with the previous work of \cite{GMWZ5}.
Consider again a linear constant-coefficient boundary value problem
$$
Lu  = f+ \eps \Delta_x u \ \mathrm{in} \ \{x_d >0\},
$$
for $L$ as in section \ref{s:linear},\footnote{
Evidently, we can extend as in Sec. \ref{s:result}
to the nonsymmetric case, at the expense of
further assumptions.}
with mixed
boundary conditions
\ba\label{mbc}
\Gamma_1 u|_{x_d=0}&=g_1,\\
\Gamma_2 \partial_d u|_{x_d=0}&=g_2
\ea
satisfying
\be\label{ra}
\rank \Gamma_1+\rank \Gamma_2=\rank \bp \Gamma_1 \\ \Gamma_2\ep=N.
\ee
Let us suppose now that $f$, $g_1$, and $g_2$ vanish in $t<0$ and satisfy high-order corner compatibility conditions at $t=0$, $x_d=0$.   We seek $u$ such that $u=0$ in $t<0$.

We seek a formal boundary-layer expansion
$$
u^\eps (x)= u_0(x,\frac{x_d}{\eps}) +
 + \eps u_1(x,\frac{x_d}{\eps}) + \eps^2
u_2(x,\frac{x_d}{\eps})+ \dots
$$
with profiles
$$
u_j(x,z)= \uu_j(x) + u_j^*(x',z), \quad j \geq 0,
$$
where $\uu_j$ is an ``outer'' solution, and $u^*_j$ is a boundary
layer profile which goes to $0$ as $z \rightarrow \infty$.

Denote by
$\rank \Gamma_1=:\mathcal{D}$ the number of Dirichlet
conditions,
$\rank \Gamma_2=:\mathcal{N}$ the number of Neumann
conditions,
$\dim \EE_+=:\mathcal{I} $ the number of incoming modes,
and $\dim \EE_-=:\mathcal{O} $ the number of outgoing modes,
so that
$$
\mathcal{D}+\mathcal{N}=\mathcal{I}+\mathcal{O}=N.
$$
Henceforth, we may (and do) take $\Gamma_1$ to be a $\cD\times N$ matrix and $\Gamma_2$ to be an $\cN\times N$ matrix.

We divide the analysis into two cases:

(i) $\mathcal{D}\ge \mathcal{I}$, or, equivalently,
$\mathcal{N}\leq \mathcal{O}$, and

(ii) $\mathcal{D}< \mathcal{I}$, or, equivalently,
$\mathcal{N}> \mathcal{O}$.

\noindent The first case is the one considered in \cite{GMWZ5}, and treated
for problem \eqref{NSeq} in \cite{R}.
The second includes the case of Neumann boundary conditions treated here,
and also the case of problem \eqref{NSeq} left untreated in \cite{R}.
As we shall see, they have quite different behavior.  We will see that in case (i) the reduced boundary condition on $\uu_0$ is derived as a solvability condition for obtaining $u^*_0$, while in case (ii) $u^*_0=0$ and the reduced boundary condition on $\uu_0$ is derived as  a solvability condition for obtaining $u^*_1$.  We begin by recalling, with some simplifications possible for this linear problem, the treatment of case (i) in \cite{GMWZ5}.

{\bf Case (i).}  The general solution of $A_d \partial_z u^*_0-\partial_z^2 u^*_0$, which decays to $0$ as $z\to \infty$, has the form
\begin{align}\label{5a}
u^*_0(x',z)=e^{A_d z}d(x')
\end{align}
where $d\in \bE_-(A_d)$ is arbitrary (here and henceforth we suppress $x'$). The $\eps^{-1}$ order boundary condition $\Gamma_2\partial_z u^*_0(0)=0$ implies
\begin{align}\label{5b}
 \partial_z u^*_0(0)\in \ker\left(\Gamma_2|_{\bE_-(A_d)}\right)\text{ and thus }u^*_0(0)\in A_d^{-1}\ker\left(\Gamma_2|_{\bE_-(A_d)}\right).
\end{align}

We make the following transversality assumption:
\begin{align}\label{as}
\begin{split}
&(a)\;\Gamma_2 \text{ has full rank, namely }\cN,\text{ on }\bE_-(A_d)\\
&(b)\;\Gamma_1\text{ has full rank on }X:= A_d^{-1}\ker\left(\Gamma_2|_{\bE_-(A_d)}\right).
\end{split}
\end{align}
Since $\dim \bE_-(A_d)=\cO$, Assumption \ref{as}(a) implies $\dim (\ker\left(\Gamma_2|_{\bE_-(A_d)}\right))=\cO-\cN$ and thus \eqref{as}(b) implies
\begin{align}\label{4}
\dim \Gamma_1 X=\cO-\cN.
\end{align}
Since the subspace $\Gamma_1 X\subset \bR^\cD$ and $\cD=\cI+\cO-\cN$, $\Gamma_1 X$ is equal to the null space of some $\cI\times\cD$ matrix, call it $\cK$.
    Now use the order $\eps^0$ Dirichlet condition
    \begin{align}\label{6}
    \Gamma_1(\uu_0(0)+u^*_0(0))=g_1
    \end{align}
to see that there exists $u^*_0(0)\in X$ satisfying \eqref{6} if and only if
\begin{align}
 \Gamma_1(\uu_0(0))-g_1 \in \Gamma_1X.
 \end{align}
 In other words
  \begin{align}
  \tilde\Gamma_1 (  \uu_0(0))=\tilde g_1,
 \end{align}
 where $\tilde \Gamma_1=\cK\Gamma_1$ and $\tilde g_1=\cK g_1$. Observe that $\tilde \Gamma_1$ is an $\cI\times N$ matrix of rank $\cI$ as required.

The reduced hyperbolic problem is therefore
\begin{align}
\begin{split}
&Lu=f\\
&\tilde\Gamma_1 \uu_0=\tilde g_1\text{ on }x_d=0\\
&\uu_0=0\text{ in }t<0,
\end{split}
\end{align}
which is
well-posed provided that the usual Kreiss Lopatinski condition\footnote{In \cite{GMWZ5} it is shown that both the Kreiss-Lopatinski and transversality conditions follow from a condition on the low-frequency behavior of an Evans function.}
is satisfied.
Continuing this process, one obtains an expansion to all orders.
{\it In this case, boundary layers
are  amplitude $O(1)$
and the reduced boundary conditions are purely Dirichlet}.

{\bf Case (ii).}
We now turn to case (ii), where we make the
assumption
\be\label{gen}
\Gamma_2 \, \hbox{\rm is full rank on }\, \EE_-(A_d).
\ee
Since $\mathcal{N}=\rank \Gamma_2\ge \mathcal{O}=
\dim \EE_-(A_d)$),
we find from the $\eps^{-1}$ order profile equation
$\Gamma_2 \partial_d u_0^*(0)=0$,
and the fact by \eqref{5a} that $\partial_d u_0^*\in \EE_-(A_d)$,
that
\be\label{weaklayer}
\partial_d u_0^*\equiv u_0^*\equiv 0.
\ee
Thus, the boundary-layer expansion features a {\it weak layer}
of amplitude $O(\eps)$, just as in the full Neumann boundary condition
case.
This implies by the order $\eps^0$ boundary condition  $\Gamma_1 u_0=g_1$,
and the weak layer property $u_0=\uu_0$, that the Dirichlet condition
is inherited unchanged by the outer solution, as
\be\label{douter}
\Gamma_1 \uu_0|_{x_d=0}=g_1.
\ee

The order $\eps^0$ Neumann condition is
   \begin{align}\label{1}
    \Gamma_2(\partial_d u_0|_{x_d=0}+\partial_z u^*_1|_{z=0})=g_2.
\end{align}
We deduce the reduced Neumann condition on $u_0$ as a solvability condition that allows us to find a solution $\partial_z u^*_1|_{z=0}\in \bE_-(A_d)$ of \eqref{1}.  Recalling that $\Gamma_2$ is an $\cN\times N$ matrix, we denote by $S$ the subspace of $\bR^{\cN}$ given by
\begin{align}\label{k4}
S=\Gamma_2(\bE_-(A_d)).
\end{align}
By \eqref{gen} the dimension of $S\subset \bR^{\cN}$ is $\cO$.  Thus, $S$ coincides with the kernel of an $(\cN-\cO)\times \cN$ matrix.  Choose one such matrix and call it $M$.

By the definition of $M$, in order to find $\partial_z u^*_1|_{z=0}\in \bE_-(A_d)$ satisfying \eqref{1} we must have
\begin{align}
M\left(\Gamma_2(\partial_d u_0|_{x_d=0}) - g_2\right)= 0,
\end{align}
or in other words
\begin{align}\label{3}
\tilde \Gamma_2 \partial_d u_0|_{x_d=0}=\tilde g_2,
\end{align}
where $\tilde\Gamma_2=M\Gamma_2$ and $\tilde g_2= Mg_2$.  As expected, $\tilde \Gamma_2$ is an $(\cN-\cO)\times N$ matrix of rank $(\cN-\cO)$, giving us the remaining $\cN-\cO$ boundary conditions needed (in addition to the $\cD$ Dirichlet conditions) for the hyperbolic problem.\\

Combining, we obtain
the reduced hyperbolic boundary-value problem
    \begin{align}\label{res2}
\begin{split}
&Lu_0=f\text{ in }x_d>0\\
&\Gamma_1 u_0|_{x_d=0}=g_1\\
&\tilde \Gamma_2 \partial_{d}u_0|_{x_d=0}=\tilde g_2\\
&u_0=0 \text{ in }t<0.
\end{split}
\end{align}

\begin{rem}\label{k3}
a)\;  \textup{In the case of full Neumann boundary conditions we have $\cN=N=\cI+\cO$, and $\Gamma_2$ is a nonsingular $N\times N$ matrix, which we may therefore always take to be $I_N$.   Then we have $S=\bE_-(A_d)$ \eqref{k4} and we may take $M=\tilde\Gamma_2$ to be an $(N-\cO)\times N$ matrix whose rows span $\bE_+(A_d)$.}

b)\; \textup{In the totally incoming case with full Neumann boundary conditions we have $\cO=0$, $S=\{0\}\subset \bR^N$,  and we can take $M=Id_N$. So $\tilde\Gamma_2=\Gamma_2=I$.}

c)\;\textup{In the totally incoming case with one Neumann boundary condition, we have $\cN=1$, $\cD=N-1$, $S=\{0\}\subset \bR^1$, and we may take $M=1$.  Thus, $\tilde \Gamma_2= \Gamma_2$, a $1\times N$ matrix.}

d) \textup{ In the totally incoming case we have $E_-(A_d)=\{0\}$; thus, our construction of the the approximate solution shows that $u_j^*(x',z)=0$ for all $j$.  In other words, the layer is absent (or constant)}.

e) \textup{ In the situation $\mathcal{D}=\mathcal{I}$
on the boundary of case (i), assuming \eqref{gen},
we find by the argument of case (ii) that the amplitude
of boundary layers is $O(\eps)$.
In other words, the layer is absent to lowest order
also in this boundary case.
}
\end{rem}

 By introducing variations on the method of Section \ref{s:linear},
we discuss next two approaches to obtaining
a well-posedness theory for problems of the form \eqref{res2}.
When one has such a theory, one can proceed as in section \ref{s:linear} to construct the boundary layer expansion to any order.

\subsection{The reduced hyperbolic problem: approach based on Kreiss symmetrizers. }\label{wp}

\emph{\quad} Substituting for $\partial_d u_0$ the expression
\be\label{boundary}
\D_d u_0 = -A_d^{-1}(\partial_tu_0+\sum_{j=1}^{d-1} A_j \D_j u_0) + A_d^{-1} f,
\ee
and taking the Laplace-Fourier transform with Laplace frequency
$\gamma+i\tau$, $\gamma, \tau\in \RR^1$, and Fourier frequency $\eta\in \RR^{d-1}$,
we convert the boundary operator appearing in
\eqref{boundary} to the homogeneous degree one boundary symbol
\be\label{bsymbol}
-A_d^{-1}(\gamma +i\tau + \sum_{j=1}^{d-1}i\eta_j A_j)
\ee
The matrix
$(\gamma +i\tau + \sum^{d-1}_{j=1}i\eta_j A_j)$, by symmetry of $A_j$,
is invertible for $\gamma>0$ with $O(\gamma^{-1})$ inverse.
As we saw above $\tilde \Gamma_2$ is of full rank $r:=\cN-\cO$; hence
$\Gamma_2':=-\tilde \Gamma_2  A_d^{-1}(\gamma +i\tau + \sum^{d-1}_{j=1}i\eta_j A_j)$
has the same rank for $\gamma>0$.
Multiplying on the left by $m(\gamma, \tau, \eta):=(i\tau+\gamma +|\eta|)^{-1}$, we obtain a symbol homogeneous of degree zero
\be\label{ortho}
\hat \Gamma_2 (\gamma, \tau, \eta):=
-m(\gamma, \tau, \eta)
\tilde \Gamma_2
A_d^{-1}(\gamma +i\tau + \sum_{j\ne d}i\eta_j A_j).
\ee
The Neumann boundary conditions can be rewritten now as degree-zero Dirichlet conditions
\ba\label{newNeumann}
\hat \Gamma_2 (\gamma, \tau, \eta) \hat u_0(\gamma,\tau, \eta,0)&=
\hat G_2(\gamma,\tau,\eta)\\
& :=
m(\gamma, \tau, \eta)\left(\hat{\tilde{g}}_2(\gamma,\tau,\eta)
-\tilde \Gamma_2
A_d^{-1}\hat f(\gamma, \tau, \eta,0)\right),
\ea
where $\hat { }$ denotes Laplace--Fourier transform.

With this rephrasing of the boundary conditions,
the Laplace--Fourier transformed system becomes a hyperbolic
boundary-value problem of the following form:
\begin{align}\label{newp}
\begin{split}
&\partial_d \hat u_0+A_d^{-1}(\gamma + i\tau+\sum_{j=1}^{d-1}i \eta_j A_j)\hat u_0=A_d^{-1}\hat f(\gamma,\tau,\eta,x_d)\\
&\Gamma_1 \hat u_0(\gamma,\tau,\eta,0)=\hat g_1\\
&\hat\Gamma_2(\gamma,\tau,\eta)\hat u_0(\gamma,\tau,\eta,0)=\hat G_2(\gamma,\tau,\eta) \text{ as in }\eqref{newNeumann}.
\end{split}
\end{align}
Uniform estimates
may be proved for \eqref{newp} using Kreiss symmetrizers
(see, for example, \cite{CP,BS,Met4}, and also Proposition
\ref{pseudoprop}),
provided that: (i) the boundary matrix $\Gamma$ is uniformly well-conditioned,
\be\label{wellcond}
|\Gamma|,\, |\Gamma^\dagger|\le C,
\ee
where $\Gamma^\dagger$ is the pseudoinverse of $\Gamma$, and
(ii) there holds the {\it uniform Lopatinski condition}:
\begin{align}\label{lop}
\det\left(\ker\bp \Gamma_1\\
\hat \Gamma_2(\gamma,\tau,\eta)\ep,\EE_+\left(A_d^{-1}(\gamma + i\tau +i \sum_{j=1}^{d-1}\eta_jA_j)\right)\right)\geq C>0
\end{align}
for some $C$ independent of $(\tau,\eta)\in\bR^d$, $\gamma>0$.  Here one defines the determinant by taking an orthonormal basis for each of the spaces appearing there.  The condition thus expresses ``uniform transversality" of those spaces for all such $(\gamma,\tau,\eta)$.

For discussion below, we recall also the {\it weak Lopatinski condition},
which is defined as in \eqref{lop}, except that $C_\gamma>0$ is allowed to depend on $\gamma>0$.

Assuming that the uniform Lopatinski condition is satisfied, we can use the following proposition to solve the
outer hyperbolic problem.   In the next proposition for $\gamma\geq 1$ we let
\begin{align}\label{norm}
|f|_{s,\gamma}:=\left||\tau,\gamma,\eta|^s\hat f(\tau-i\gamma,\eta,x_d)\right|_{L^2(\tau,\eta,x_d)},
\end{align}
and we let $\langle g\rangle_{s,\gamma}$ denote the corresponding norm on the boundary.   The block structure assumption
made in the next proposition is satisfied by many of the important physical examples (see \cite{MZ2}); we shall omit further discussion of it here.\footnote{The block structure assumption can actually be avoided in the
constant coefficient symmetric case by using the approach of \cite{GMWZ8}.}

\begin{prop}\label{ests}
Suppose that $L$ is an operator that can be conjugated to \emph{block structure} in the sense of \cite{MZ2}.
Assuming well-conditioning \eqref{wellcond} and
uniform stability \eqref{lop}, there exist positive constants $C$, $\gamma_0$ and a unique solution of \eqref{res2} satisfying
\be\label{mixest}
\gamma |u|^2_{0,\gamma} + \langle u\rangle^2_{0,\gamma} \leq C\left( |f|^2_{0,\gamma}/\gamma+|\partial_{x_d}f|_{-1,\gamma}^2+ \langle g_1\rangle_{0,\gamma}^2+ \langle \tilde g_2\rangle^2_{-1,\gamma}\right).
\ee
for $\gamma \geq \gamma_0$.
\end{prop}

\begin{proof}

For the problem  \eqref{newp} with data $(A_d^{-1}f,g_1,G_2)$ one has the standard Kreiss estimate (\cite{CP,BS}):
\begin{align}\label{k1}
\gamma |u|^2_{0,\gamma} + \langle u\rangle^2_{0,\gamma} \leq C\left( \frac{|f|^2_{0,\gamma}}{\gamma}+\langle g_1\rangle^2_{0,\gamma}+\langle G_2 \rangle^2_{0,\gamma}\right).
\end{align}
Existence for the problem \eqref{newp} follows from Proposition \ref{pseudoprop}, which allows for pseudodifferential boundary conditions.   The estimate \eqref{mixest} now follows directly from \eqref{k1} and \eqref{newNeumann} using $|m|\sim |\tau,\gamma,\eta|^{-1}$ and
\begin{align}
\langle f|_{x_d=0}\rangle_{0,\gamma}\leq |f|_{0,\gamma}+|\partial_{x_d}f|_{0,\gamma}.
\end{align}

\end{proof}

Assuming that the uniform Lopatinski condition is satisfied, we can solve the outer hyperbolic problem in this way and as in section \ref{s:linear} proceed to construct the boundary layer expansion to any order.
The following proposition provides some information about when the weak and uniform Lopatinski conditions are satisfied by the problem \eqref{newp}.

\begin{lem}\label{lopsat}

Consider the problem \eqref{newp}, where the $A_j$ are constant, real, symmetric $N\times N$ matrices.

(a) In the totally incoming case with mixed boundary conditions or full Neumann boundary conditions, if the weak Lopatinski condition holds then the uniform Lopatinski condition holds.

(b) Let $d>1$.  For full Neumann boundary conditions the weak Lopatinski condition holds. The uniform Lopatinski condition can fail if the characteristics are not totally incoming.   For example, it fails whenever there exists an eigenvalue $\omega(\tau-i\gamma,\eta)$ of $-A_d^{-1}(\tau-i\gamma+\sum^{d-1}_{j=1}A_j \eta_j)$, analytic in $\tau-i\gamma$, such that $\omega(\tau,\eta)=0$ and  $\partial_\tau \omega(\tau,\eta)<0$ for the chosen $(\tau,\eta)$.

(c) For pure Neumann boundary conditions and $d=1$ the uniform Lopatinski condition is satisfied.

(d) In the totally incoming case with a single Neumann condition, the weak Lopatinski condition holds
if and only if $\bp \Gamma_1\\\Gamma_2 A_d^{-1}\ep$  (in this case
a full $N\times N$ matrix) is invertible.

(e) For mixed boundary conditions the weak Lopatinski condition holds only if
 $\bp \Gamma_1\\\tilde \Gamma_2 A_d^{-1}\ep$ is full rank
on $\EE_+(A_d)$.  There are examples with mixed boundary conditions where weak Lopatinski fails and other examples where uniform Lopatinski holds.

\end{lem}

\begin{proof}

\textbf{(a) }In the totally incoming case $\EE_+(A_d^{-1}(\gamma + i\tau +i \sum_{j=1}^{d-1}\eta_jA_j))=\bC^N$.   If the weak Lopatinski condition holds the determinant \eqref{lop}  is $\pm 1$ for all $\gamma>0$.

\textbf{(b) }In the full Neumann case $\Gamma_1$ is absent and
$\tilde\Gamma_2$ is an $(N-\cO)\times N$ matrix whose rows span $\bE_+(A_d)$ (see Remark \ref{k3}).
Since $A_d^{-1}(\gamma + i\tau + \sum_{j=1}^{d-1}i\eta_j A_j)$
is invertible for $\gamma>0$ and $\EE_+(
A_d^{-1}(\gamma + i\tau + \sum_{j=1}^{d-1}i\eta_j A_j))$
an invariant subspace, we find that the weak Lopatinski condition
is equivalent to
$\tilde \Gamma_2 $
being full rank on $\EE_+(A_d^{-1}(\gamma + i\tau +i \sum_{j=1}^{d-1}\eta_jA_j))$ for $\gamma>0$.
Since the problem
\begin{align}\label{newpa}
\begin{split}
&\partial_d w+A_d^{-1}(\gamma + i\tau+\sum_{j=1}^{d-1}i \eta_j A_j)w=0\\
&\tilde \Gamma_2 w|_{x_d=0}=h
\end{split}
\end{align}
is maximally dissipative, a simple energy estimate shows $|w(0)|\leq C|h|$ when $w\in \EE_+(A_d^{-1}(\gamma + i\tau +i \sum_{j=1}^{d-1}\eta_jA_j))$, so the full rank condition holds.

In the case $A_d$ is not positive definite, the kernel space in \eqref{lop} must be nontrivial.
Taking $\gamma=0$, $|\eta|=1$ and choosing
$\tau$ from among the eigenvalues $\lambda_k(\eta,0)$ (here $\xi_d=0$) of $-\sum_{j=1}^{d-1}\eta_jA_j$
with corresponding eigenvector $v\neq 0$, we find that $\Gamma_2'(\gamma,\tau,\eta)$,
or, equivalently, $\hat \Gamma_2(\gamma,\tau,\eta)$, annihilates $v$.
It can happen that $v$ lies in the limit space as $\gamma\to 0$ of
$\EE_+(A_d^{-1}(\gamma + i\tau +i \sum_{j=1}^{d-1}\eta_jA_j))$.  The Cauchy-Riemann equations imply that this happens, for example, whenever there is  an eigenvalue $\omega(\tau-i\gamma,\eta)$ of $-A_d^{-1}(\tau-i\gamma+\sum^{d-1}_{j=1}A_j \eta_j)$, analytic in $\tau-i\gamma$, such that $\omega(\tau,\eta)=0$ and  $\partial_\tau \omega(\tau,\eta)<0$ for the chosen $(\tau,\eta)$.

Since $v$ is also a limit of vectors lying in $\ker\hat{\Gamma}_2(\gamma,\tau,\eta)$ as $\gamma\to 0$, we see that
for such $(\tau,\eta)$ the determinant in \eqref{lop} converges to zero along some sequence $\gamma_n\to 0$.

\textbf{(c) }When $d=1$ and $\gamma>0$, we have $\ker\hat\Gamma_2(\gamma,\tau)=\ker \tilde \Gamma_2 A^{-1}_d$ and $\bE_+(A^{-1}_d(\gamma+i\tau))=\bE_+(A_d)$.  Thus, both spaces are independent of $(\tau,\gamma)$.  The uniform Lopatinski condition now follows from the fact that $\tilde\Gamma_2$ is full rank on $\bE_+(A_d)$.

\textbf{(d) }Recall from Remark \ref{k3} that $\tilde\Gamma_2=\Gamma_2$ a $1\times N$ matrix.
The assertion follows by the
observation that in this case the real part of the determinant of $\bp \Gamma_1\\
\Gamma_2A_d^{-1}(\gamma+i\tau + i\sum_{j\ne d}\eta_j A_j)\ep$ is
$\gamma \det \bp \Gamma_1\\
\Gamma_2A_d^{-1}\ep$.

\textbf{(e)}The first assertion regarding mixed boundary conditions follows
by inspection of the case $\gamma=1$, $\tau=0$, $\eta=0$.
For the second assertion we refer to the examples given below.

\end{proof}

\begin{rem}\label{k8}
1) \textup{When the problem \eqref{newp} only satisfies the  weak Lopatinski condition, there is a
still a chance of proving well-posedness  for the reduced hyperbolic problem \eqref{res2} using degenerate Kreiss symmetrizers and constructing the WKB expansion.  Indeed, several kinds of weakly stable problems have been studied successfully in this way (see, for example, \cite{BS,Co2}); typically the energy estimates exhibit a loss of derivatives.}

2)\textup{\;Glancing points are points $(\tau,\eta)$  where the matrix
$A_d^{-1}(i\tau + \sum_{j=1}^{d-1}i\eta_jA_j)$ has nontrivial Jordan blocks, or equivalently, where
an eigenvalue
$\lambda_j(\xi,\eta)$ of
$\sum_{j=1}^{d-1}\eta_jA_j + \xi A_d$, is stationary with respect
to $\xi$.   Such points always occur in $d>1$, \emph{except} in the totally incoming or totally outgoing cases, where they never occur (see \cite{GMWZ6}).  Example \ref{neueg} shows that the uniform Lopatinski condition can fail at glancing points.
We know of no proof of well-posedness for the rescaled initial-boundary value
problem in the case when the uniform Lopatinski condition fails in this way.
(We present a different method in Appendix \ref{s:new}
for which this difficulty does not
appear; see Example \ref{neueg2}.)}

3)\textup{ Example \ref{badinceg} shows that even weak stability can fail for the problem \eqref{newp}.}
\end{rem}

\begin{example}\label{neueg}
Consider the simplest example of the first-order wave equation
with drift $\alpha$,
$$
A_1=\bp 0 & 1\\ 1 & 0 \ep,
\quad
A_2 =\bp 1+\alpha & 0 \\ 0 & -1+\alpha \ep,
$$
with full Neumann boundary conditions, so that
$\tilde \Gamma_2= \bp 1 & 0\ep$.
Then
$$
\Gamma_2'=-\tilde \Gamma_2 A_2^{-1}(\gamma+i\tau + i\eta A_1)
=-\bp \frac{\gamma+ i\tau}{1+\alpha} & \frac{i\eta}{1+\alpha} \ep,
$$
which leads to the zero-order boundary matrix $\hat \Gamma_2=-\frac{1}{i\tau+\gamma+|\eta|}\bp
\frac{ \gamma+i\tau}{1+\alpha}
& \frac{ i\eta }{1+\alpha}\ep$.  By Lemma \ref{lopsat}(b) the weak Lopatinski condition holds.  Applying the criterion of Lemma \ref{lopsat}(b),  we find that the uniform Lopatinski condition
fails at $\eta=-1$, $\gamma=0$, $\tau=1$, where
$\lim_{\gamma\to 0^+}\EE_+( A_2^{-1}(\gamma+i\tau + i\eta A_1)=\Span\{ (1,1)^T\}$.
Moreover, the computation
$\lambda_\pm(\xi,\eta)= \alpha \xi \pm \sqrt{\xi^2+\eta^2}$,
where $\lambda_\pm$ are the eigenvalues of $\xi A_2 + \eta A_1$
shows that $\partial \lambda_\pm/\partial \xi=0$ at $\xi=0$,
corresponding to failure at a glancing point,
occurs only for $\alpha=0$ for this choice of $(\tau,\eta)$.
\end{example}

\begin{example}\label{inceg}
Next, consider the totally incoming problem
$$
A_1=\bp 0 & 1\\ 1 & 0\ep,
\quad
A_2 =\Id,
$$
with mixed Dirichlet--Neumann conditions $\Gamma_1=\bp * &1\ep$,
$\Gamma_2= \tilde \Gamma_2= \bp 1 & 0\ep$.
Then,
$$
\bp \Gamma_1 \\ \Gamma_2'\ep=
\bp \Gamma_1 \\ -\tilde \Gamma_2 A_2^{-1}(\gamma+i\tau + i\eta A_1)\ep
=
\bp \Gamma_{11}& \Gamma_{12}\\
-(\gamma+ i\tau) & -i\eta \ep
$$
is full rank on $\EE_+(A_2^{-1}(\gamma+i\tau+i\eta A_1))=\CC^2$ whenever
$
0\ne \det \bp \Gamma_{11}& \Gamma_{12}\\
\gamma+ i\tau & i\eta \ep=
\gamma + i(\tau - \eta \Gamma_{11}),
$
in particular for $\gamma>0$.
Thus, we have weak Lopatinski stability of the zero-order boundary condition.
$\bp \Gamma_1\\\hat \Gamma_2\ep$. By Lemma \ref{lopsat} the uniform Lopatinski condition also holds.

\end{example}

\begin{example}\label{badinceg}
Finally, consider the totally incoming problem
$$
A_1=\bp 0 & 1&a\\ 1 & 1&0\\ a & 0 & 0\ep,
\quad
A_2 =\Id,
$$
with mixed Dirichlet--Neumann conditions $\Gamma_1=\bp 1 & 1 & b \ep$,
$\Gamma_2= \tilde \Gamma_2= \bp 0 & 1& 0 \\ 0 & 0 & 1\ep$.
Then,
$$
\bp \Gamma_1 \\ \Gamma_2'\ep=
\bp \Gamma_1 \\ -\tilde \Gamma_2 A_2^{-1}(\gamma+i\tau + i\eta A_1)\ep
=
-\bp -1 & -1 & -b\\
i\eta& \gamma+ i\tau+ i\eta & 0 \\
i\eta a &0 & \gamma+ i\tau  \ep
$$
is full rank on $\EE_+=\CC^2$ when its determinant is nonvanishing.
An easy row reduction gives
$$
\det \bp \Gamma_1 \\ \Gamma_2'\ep=
\det \Big( (\gamma+i\tau)\Id - i\eta \bp 0 & a\\ b & ab\ep
\Big)= 0
$$
when $(\gamma+i\tau)/i\eta$ is an eigenvalue of $ \bp 0 & a\\ b & ab\ep $,
or
$(\gamma+i\tau)/i\eta= \frac{ab\pm \sqrt{a^2b^2+4ab}}{2}$.
Choosing $a=1$, $b=-1$, we obtain
$(\gamma+i\tau)/i\eta= \frac{-1 \pm i \sqrt{3}}{2}$, or
$\gamma+i\tau= (-i \mp \sqrt{3})(\eta/2)$, and the Lopatinski condition
is violated for $\gamma= \mp \sqrt{3}\eta/2$, $\tau= -\eta/2$.
This shows that the weak Lopatinski condition can fail for the totally
incoming case, even with $\bp \Gamma_1\\\Gamma_2 A_d^{-1}\ep$
full rank.
\end{example}


\begin{example}\label{O.disc1}
This last example comes from a result by B. Fornet (see \cite{F1}, \cite{F2}), and shows that such types of Dirichlet-Neumann boundary conditions
have a natural place in the theory of first order hyperbolic Cauchy problems with discontinuous coefficients.
Let us consider the following scalar Cauchy problem in 1D
\begin{equation*}
\left\{\begin{aligned}{}
&\D_t u+ a(x) \D_{x} u = f \qquad x\in \RR, t > 0  \\
& u|_{t=0}=h \quad
\end{aligned}\right.
\end{equation*}
where the real valued coefficient $a(x)$ satisfies $a(x) = -\alpha < 0$ if $x<0$ and $a(x) = \beta > 0$ if $x\geq 0$, with data $h \in C^\infty(\RR)$, $f\in C^\infty(\RR^2)$ compactly supported. This problem is of course not well-posed due to the lack of uniqueness. In order to select one solution, one can use
for example a vanishing viscosity approach, and look for the limit of the solution $u^\eps$ of
\begin{equation*}
\left\{\begin{aligned}{}
&\D_t u+ a(x) \D_{x} u - \eps \D^2_{x} u= f \qquad x\in \RR, t > 0  \\
& u|_{t=0}=h \quad
\end{aligned}\right.
\end{equation*}
as $\eps \rightarrow 0$. To study the convergence, the problem is written as an initial boundary value problem (or
transmission problem) with $u^\eps_\pm(t,x) = u^\eps(t,\pm x)$ for $x>0$ and $v^\eps=(u^\eps_+,u^\eps_-)^{T}$ leading to the constant coefficient system
\begin{equation}\label{O.disc2}
\D_t v^\eps+ A \D_{x} v^\eps - \eps \D_x^2 v^\eps = (f_+,f_-)^T \ \mathrm{in} \ t>0, x>0
\end{equation}
with
$$
A =\bp \beta & 0\\ 0 & \alpha \ep,
$$
corresponding to totally incoming characteristic fields. The  boundary conditions are $\Gamma_1v=0$ and $\Gamma_2 \D_x v =0$ on $x=0$ with
\begin{equation}\label{O.disc3}
\Gamma_1 = (1, -1), \quad  \Gamma_2 = (1, 1).
\end{equation}
The result is that $v^\eps$ converges in $L^2([0,T]\times \RR_+)$ to the (unique) solution $v^0$ of the limit hyperbolic problem
$$
\D_t v^0+ A \D_{x} v^0 = (f_+,f_-)^T
$$
with the same boundary conditions
$$
\Gamma_1 v^0 _{|x=0} = 0, \quad \Gamma_2 (\D_x v^0 )_{|x=0} = 0,
$$
and initial conditions $(h_+,h_-)^T$. The fact that the problem is one dimensional helps a lot, and as a matter of fact, it is an example where the uniform Evans condition is satisfied (see \cite{F1}). The convergence analysis also uses specific boundary layer expansions. One can find more general situations and examples in the paper \cite{F2} with larger systems, still in 1D.
\end{example}


\begin{rem}\label{hope}
\textup{ Example \ref{inceg} is an example of the mixed, totalling incoming case with
one Neumann condition where the uniform Lopatinski condition holds.
Recall that this case, corresponding to supersonic incoming flow with a Neumann condition on temperature, was  left open in the study of  boundary layers for the full compressible Euler equations
\eqref{NSeq} in \cite{R}.  In Appendix \ref{s:Rao} we provide a criterion (satisfied for example by ideal gases) for the uniform Lopatinski condition to be satisfied in that case.}

\end{rem}

We point out that the
well-conditioning of $\Gamma$, \eqref{wellcond},
fails in many cases.  In particular,
in the totally incoming case, when there is even one Neumann condition,
we find that $\Gamma$ drops rank for $\gamma=0$ at any values of
$\tau, \eta$ for which $\tau + \sum_{j\ne d}\eta_jA_j$ is not invertible,
so that $|\Gamma^\dagger|$ blows up as $\gamma\to 0$.
In this particular case, this may be remedied by simply multiplying
$\Gamma$ and data $g$ both by $\Gamma^{-1}$ to eliminate this difficulty
at the expense of losses on the source;
we explore this approach further in Remark \ref{mrmk} below.

In this section, we have dealt entirely with construction of approximate
solutions.  Convergence to these solutions is a separate issue that
requires estimates on the full hyperbolic--parabolic problem, estimates
that we have for the moment only for the pure Neumann boundary, totally incoming
case.  This is an important direction for further investigation.

\subsection{Second approach based on solving a Cauchy problem on the boundary. }\label{s:second}

\emph{\quad} We return now to the reduced hyperbolic problem in its original form \eqref{res2} in the general case of mixed-type boundary conditions, but assuming that we are in the totally incoming case\footnote{Thus, this approach is relevant to the example of Rao discussed in Appendix \eqref{s:Rao}. }.
Extensions to the general case are discussed in Appendix \ref{s:new}.
Writing $u=u_0$, differentiating the Dirichlet boundary condition
$\Gamma_1 u|_{x_d=0}=g_1$ with respect to time,
and making the usual substitution \eqref{boundary} for $\partial_du_0$, we obtain
the boundary condition
\be\label{tansys}
B u|_{x_d=0}= \bp \partial_t g_1\\ \tilde g_2-\tilde\Gamma_2 A_d^{-1}f\ep,
\qquad
\hbox{\rm where} \;
B:=
\bp \Gamma_1 \\ -\tilde \Gamma_2 A_d^{-1}\ep \partial_ t
+\sum_{j=1}^{d-1} \bp 0\\ -\tilde \Gamma_2 A_d^{-1}A_j\ep \partial_{x_j}.
\ee
The next proposition shows that sometimes this may be treated as
a Cauchy problem in the tangential variables
and solved for complete Dirichlet data $u|_{x_d=0}$.

\begin{defn}\label{hypdef}

Let $p(\tau,\eta):= \det \left(\bp \Gamma_1 \\ -\tilde \Gamma_2 A_d^{-1}\ep
\tau
+
\sum_{j= 1}^{d-1} \bp 0\\ -\tilde\Gamma_2 A_d^{-1}A_j\ep \eta_j\right)$.
We say that the system \eqref{tansys} on the boundary
is:

a) \emph{evolutionary} if
the coefficient of $\partial_t$ is invertible.

b) \emph{weakly hyperbolic} if for any $\eta\in\bR^{d-1}$ the roots in $\tau$ of $p(\tau,\eta)=0$ are real.

\end{defn}

\begin{prop}\label{weakprop}
The system \eqref{tansys} is both evolutionary and weakly hyperbolic if and only if the problem \eqref{newp} satisfies the weak Lopatinski condition.
\end{prop}

\begin{proof}
\textbf{1. }First observe that the factor $m$ in $\hat\Gamma_2$ has no effect on the kernel space in \eqref{lop}.  Suppose the weak Lopatinski condition holds.  Taking $\tau=0$, $\eta=0$, $\gamma>0$ in \eqref{lop}, since the $E_+$ space in \eqref{lop} is $\bC^N$, we see that the coefficient of $\partial_t$ in \eqref{tansys} is invertible.   More generally, the matrix
$\bp \Gamma_1\\
\hat \Gamma_2(\gamma,\tau,\eta)\ep$ is nonsingular when $\gamma>0$, and thus so is the matrix
$\bp \Gamma_1 \\ -\tilde \Gamma_2 A_d^{-1}\ep
(\tau-i\gamma)
+
\sum_{j= 1}^{d-1} \bp 0\\ -\tilde\Gamma_2 A_d^{-1}A_j\ep \eta_j$.

\textbf{2. }The argument can be reversed to prove the other direction.

\end{proof}

 Weak hyperbolicity is not enough to guarantee well-posedness in $H^s$ spaces of the problem \eqref{tansys}.
We refer to \cite{BS} for a discussion of necessary and sufficient conditions for such well-posedness.
One important sufficient condition for well-posedness is that the roots in $\tau$ of $p(\tau,\eta)=0$ are real and semisimple with constant multiplicities for $\eta\neq 0$.   This condition is verified for the  system \eqref{tansys} arising in the Rao example in Appendix \ref{s:Rao}.

In problems where  the trace  $u_0|_{x_d=0}=h$ can be found by solving \eqref{tansys}, we can obtain the solution to the reduced hyperbolic problem \eqref{res2} by solving
\begin{align}
Lu_0=f \text{ in }x_d>0, \;u_0|_{x_d=0}=h,\; u_0=0 \text{ in } t<0.
\end{align}
This problem is maximally dissipative in the totally incoming case.

We record the resulting bounds, which are to be compared to those
of \eqref{mixest}.

\begin{prop}\label{statebds}
Suppose that $L$ is an operator that can be conjugated to \emph{block structure} in the sense of \cite{MZ2}.
Assuming that the roots in $\tau$ of $p(\tau,\eta)=0$ are real and semisimple with constant multiplicities for $\eta\neq 0$,
there exist positive constants $C$, $\gamma_0$ and a unique solution of \eqref{res2} satisfying
\be\label{e:statest}
\gamma |u|^2_{0,\gamma} + \langle u\rangle^2_{0,\gamma} \leq C\left( |f|^2_{0,\gamma}/\gamma+|\partial_{x_d}f|_{0,\gamma}^2/\gamma^2
+ \langle \partial_t g_1\rangle_{0,\gamma}^2/\gamma^2
+ \langle \tilde g_2\rangle^2_{0,\gamma}/\gamma^2 \right).
\ee
for $\gamma \geq \gamma_0$.
\end{prop}

\begin{proof}
Estimating the $\langle\cdot\rangle_{0,\gamma}$ norm of the trace of $f$ at $x_d=0$,
and using this to bound
the data $\bp \partial_t g_1\\ \tilde g_2-\tilde\Gamma_2 A_d^{-1}f\ep$
in \eqref{tansys}, we obtain from standard hyperbolic Cauchy estimates
the bound
$
\langle u|_{x_d=0}\rangle^2_{0,\gamma}\leq
C\left( |f|^2_{0,\gamma}/\gamma^2+|\partial_{x_d}f|_{0,\gamma}^2/\gamma^2
+ \langle \partial_t g_1\rangle_{0,\gamma}^2/\gamma^2)
+ \langle \tilde g_2\rangle^2_{0,\gamma}/\gamma^2 \right),
$
from which \eqref{e:statest} then follows by standard boundary value
estimates for maximally dissipative systems.
\end{proof}

\begin{rem}\label{mrmk}
\textup{
The bounds \eqref{mixest} obtained by method one in Proposition \ref{ests} are stronger than those of
\eqref{e:statest} by  factor $\gamma/|\gamma,\tau,\eta|$
in boundary terms $g_j$ and the term $\partial_{x_d}f$
coming from the trace of $f$.
This reflects the well-conditioning hypothesis \eqref{wellcond} made in Proposition \ref{ests} but not in our derivation of \eqref{e:statest}.
Indeed, when well-conditioning fails (but the other hypotheses of Proposition \ref{ests} hold)
one can apply method one to derive the bounds \eqref{e:statest} provided one can find
for $\gamma>0$ and $r:=\rank \Gamma_1+\rank \tilde \Gamma_2$
an $r\times r$ matrix multiplier
$|m(\gamma,\tau,\eta)|\le C/\gamma$,
such that the rescaled boundary condition
\be\label{genrescale}
m\bp (\gamma+i\tau)\Gamma_1\\
-\tilde \Gamma_2 A_d^{-1}(\gamma+i\tau +\sum_{j\ne d} i\eta_j A_j)\ep
\ee
satisfies the well-conditioning hypothesis \eqref{wellcond} needed
to obtain standard Kreiss-type bounds for the resulting rescaled
boundary-value problem.
One may check that this yields exactly the bounds \eqref{e:statest}.
Thus, this modification allows somewhat wider application of method one.
For example, in the case of totally incoming characteristics, the
uniform Lopatinski condition is trivially satisfied, but \eqref{wellcond}
fails for the multiplier $|\gamma,\tau,\eta|^{-1}$, whereas for
the multiplier $m:=\Gamma^{-1}$, the rescaled
boundary condition $m\Gamma=\Id$ trivially satisfies \eqref{wellcond},
and in favorable cases satisfies $|m|=|\Gamma^{-1}|\le C/\gamma$.
Indeed, this can be recognized as the solution operator of the
Cauchy problem on the boundary just described in method two.
}
\end{rem}

\subsection{Discussion and open problems}\label{s:discussion}

\emph{\quad} To summarize, following up on the analyses initiated in \cite{GMWZ5,GMWZ6}
to accommodate mixed Neumann--Dirichlet boundary conditions in the general
theory of hyperbolic--parabolic boundary layers, we here investigate the
case left open in those works that the number of incoming modes exceeds
the number of Dirichlet conditions imposed on the full hyperbolic--parabolic
solution.
In this case, we find that (i) the resulting reduced,
hyperbolic, ``outer problem'' satisfies Neumann or mixed Neumann--Dirichlet,
rather than Dirichlet conditions as in the standard case, and
(ii) the resulting boundary layers are ``weak'' in the sense that they
are $O(\eps)$ amplitude, where $\eps$ is the order of the viscosity.

Although the existence of this new type of boundary layer, with quite different
behavior from the standard type, is surprising to us, such layers have physical relevance (sse Appendix \ref{s:Rao}). In particular, one must understand these layers in order to treat cases arising in physical applications to suction-reduced
drag in aerodynamics.
Their analysis requires the study of hyperbolic boundary-value problems with
Neumann or mixed Neumann--Dirichlet boundary conditions, an area that appears not to have received much attention, despite the extensive
study of noncharacteristic hyperbolic boundary-value problems.   We have described two approaches to these hyperbolic boundary problems, one involving a reduction to a problem with pseudodifferential Dirichlet conditions, and the other involving a reduction to a Cauchy problem on the boundary.   We have provided examples where each approach works, but much work remains to be done on the general case.    An important example where the second approach works is the case of supersonic inflow for the full compressible Euler equations
considered in Appendix \ref{s:Rao}.

To study the small viscosity limit in the quasilinear hyperbolic-parabolic boundary problems considered here, our approach requires estimates for the linearization of the problem about an approximate solution.   The derivation of such estimates is completely open for cases other than the pure Neumann
totally incoming case treated in the remainder of the paper.

\section{The quasilinear totally incoming case}
\emph{\quad} We turn now to our main task,
the full treatment of the quasilinear case with full Neumann boundary
conditions and totally incoming modes.

\subsection{Construction of an approximate solution}\label{approx}

\quad By plugging $u^a_\eps$ as in \eqref{ab} into the boundary problem \eqref{aa}, Taylor expanding $A_j(u^a_\eps)$ about $u_0$,   and equating coefficients of equal powers of $\eps$ on right and left, we obtain the following sequence of boundary problems:\\
\begin{align}\label{g1}
\begin{split}
&(a)\;\sum^d_{j=0}A_j(u_0)\partial_ju_0=f,  \quad \partial_d u_0|_{x_d=0}=0\\
&(b)\;\sum^d_{j=0}A_j(u_0)\partial_ju_1+\sum^d_{j=1}d_uA_j(u_0)(u_1,\partial_ju_0)=\Delta u_0,  \quad \partial_d u_1|_{x_d=0}=0\\
&(c)\;\sum^d_{j=0}A_j(u_0)\partial_ju_2+\sum^d_{j=1}d_uA_j(u_0)(u_2,\partial_ju_0)=\\
&\qquad \quad \Delta u_1-\sum^d_{j=1}d_uA_j(u_0)(u_1,\partial_ju_1)-\sum^d_{j=1}d^2_uA_j(u_0)(u_1,u_1,\partial_ju_0),  \quad \partial_d u_1|_{x_d=0}=0
\end{split}
\end{align}
and so on, where $f=0$ in $t<0$ and $u_j=0$ in $t<0$ for all $j$.   Here $f\in H^s(\overline{\bR}^{d+1}_+)$ for $s$ large to be specified later.

    To solve \eqref{g1}(a), we first solve the symmetric, pure initial value problem on $x_d=0$:
\begin{align}\label{g2}
\sum^{d-1}_{j=0}A_j(v)\partial_jv= f|_{x_d=0},\quad  v=0\text{ in }t<0,
\end{align}
and then the symmetric, dissipative boundary problem on $\Omega_{T_0}$ for $T_0$ small:
\begin{align}\label{g3}
\sum_{j=0}^dA_j(u_0)\partial_ju_0=f,\quad u_0|_{x_d=0}=v,\quad u_0=0\text{ in }t<0.
\end{align}
From \eqref{g2}, \eqref{g3}, and the invertibility of $A_d(u_0)$ we obtain $\partial_d u_0|_{x_d=0}=0$.   The subsequent linear problems \eqref{g1}(b),(c),...for the unknowns $u_1$, $u_2$,... are solved by the same method.

Standard theory gives
$0<T_0<T_1$ such that\footnote{The drop by three units of regularity at each stage is due application of the Laplacian and the taking of a trace.  Here we have chosen to restrict the Sobolev indices to lie in $\bN$. }
\begin{align}
\begin{split}
&v\in H^{s-1}(b\Omega_{T_1}), \;u_0\in H^{s-1}(\Omega_{T_0}),\;u_1\in H^{s-4}(\Omega_{T_0}),\;u_2\in H^{s-7}(\Omega_{T_0}),\\
&\qquad \quad  \;\dots, u_k\in H^{s-1-3k}(\Omega_{T_0}).
\end{split}
\end{align}
Moreover, as long as $s-3M-2>\frac{d+1}{2}$, it is easy to check that the remainder $R_\eps$ in \eqref{ac} belongs to $H^{s-3M-3}(\Omega_{T_0})$.  We now summarize this construction.

\begin{prop}[Approximate solutions]\label{g4}
Fix $M\in\bN$.  Consider the boundary problem \eqref{aa}, where $f\in H^s(\overline{\bR}^{d+1}_+)$ for some $s>3M+2+\frac{d+1}{2}$. Then \eqref{aa} has an approximate solution of the form
\begin{align}\label{g5}
u^a_\eps(x)=u_0(x)+\eps u_1(x) + \dots +\eps^M u_M(x),
\end{align}
satisfying \eqref{ac}, where $u_k\in H^{s-1-3k}(\Omega_{T_0})$ and the remainder $R_\eps\in H^{s-3M-3}(\Omega_{T_0})$.

\end{prop}

\subsection{Error equation}\label{error}
We look for an exact solution of the form
\begin{align}\label{ad}
u_\eps=u^a+\eps^L v^\eps, \text{ where }1\leq L<M.
\end{align}
To obtain the problem satisfied by $v$ we divide the equation $\cE(u)-\cE(u^a)=-\eps^M R_\eps$ by $\eps^L$ to obtain
\begin{align}\label{ae}
\begin{split}
&\sum^d_{j=0}A_j(u^a+\eps^Lv)\partial_{x_j}v+E(u^a,\nabla u^a,\eps^Lv)v-\eps\Delta v=-\eps^{M-L} R_\eps\\
&\partial_{x_d}v|_{x_d=0}=0\\
&v=0\text{ in }x_0 <0,
\end{split}
\end{align}
where with $\nabla=(\partial_{x_1},\dots,\partial_{x_d})$
\begin{align}\label{af}
E(u^a,\nabla u^a,\eps^Lv)v:=\sum^d_{j=1}\left(\int^1_0\partial_u A_j(u^a+s\eps^L v)\cdot v \;ds\right)\partial_{x_j}u^a.
\end{align}
To obtain a linear operator acting on $v$ on the left we rewrite \eqref{ae} as
\begin{align}\label{ae1}
\begin{split}
&\sum^d_{j=0}A_j(u^a)\partial_{x_j}v+E(u^a,\nabla u^a,0)v-\eps\Delta v=\\
&\qquad -\eps^{M-L} R_\eps+\eps^LB_1(u^a,\eps^Lv)(v,\nabla v)+\eps^LB_2(u^a,\nabla u^a,\eps^L v)(v,v):=\cF_\eps(v,\nabla v)\\
&\partial_{x_d}v|_{x_d=0}=0\\
&v=0\text{ in }x_0 <0.
\end{split}
\end{align}
Here $B_1$ and $B_2$, defined by the equation,  are smooth functions and bilinear in their last two arguments.

Next we rewrite \eqref{ae1} as a $2N\times 2N$ first-order system for the unknown $U=(u_1,u_2)^t:=(v,\eps\partial_{x_d}v)^t$, setting $\partial^{''}=(\partial_{x_1},\dots,\partial_{x_{d-1}})$:
\begin{align}\label{ag}
\begin{split}
&\partial_{x_d}U^{}=\frac{1}{\eps}G(p(x),\eps\partial_{x'})U^{}+F_\eps(U,\partial^{''} U)\\
&\Gamma U^{}:=u^{}_2=0\text{ on }x_d=0\\
&U^{}=0\text{ in }x_0 <0,
\end{split}
\end{align}
where
\begin{align}\label{ah}
\begin{split}
&F_\eps(U)=\begin{pmatrix}0\\-\cF_\eps(v,\nabla v)\end{pmatrix}\text{ and }G(p(x),\eps\partial_{x'})=\begin{pmatrix}0&I\\M& A_d \end{pmatrix}\text{ with }\\
&M=\sum^{d-1}_{j=0}A_j(u^a)\eps\partial_{x_j}+\eps E(u^a,\nabla u^a,0)-\eps^2\Delta_{x''}\text{ and }A_d=A_d(u^a).
\end{split}
\end{align}
In \eqref{ag} we have set
\begin{align}\label{ai}
\begin{split}
&p(x)=(p_1(x),p_2(x),p_3(x))\text{ where }\\
&p_1(x):=u_0,\;\;p_2(x)=u^a-u_0,\;\;p_3(x)=\eps E(u^a,\nabla u^a,0).
\end{split}
\end{align}

To prove weighted estimates we introduce $\tilde{U}=e^{-\gamma x_0}U$, $\tilde{F}=e^{-\gamma x_0}F$, where $\gamma \geq 1$,  and observe that \eqref{ag} is equivalent to
\begin{align}\label{aj}
\begin{split}
&\partial_{x_d} \tilde{U}=\frac{1}{\epsilon}G^\gamma(p(x),\eps\partial_{x'},\eps\gamma) \tilde{U}+\tilde{F}_\eps(U,\partial^{''}U),\\
&\Gamma \tilde U^{}:=\tilde u^{}_2=0\text{ on }x_d=0\\
&\tilde U^{}=0\text{ in }x_0 <0,
\end{split}
\end{align}
where $G^\gamma$ is defined by replacing  $\partial_{x_0}$ by $\partial_{x_0}+\gamma$ the definition of $G$.

\subsection{Symbolic preparation}

The operator $G^\gamma$ in \eqref{aj} is the semiclassical differential operator defined by the symbol
\begin{align}\label{ak}
G(p(x),\beta)=\begin{pmatrix}0&I\\M(p(x),\beta)&A(p(x))\end{pmatrix},
\end{align}
where, with $p=(p_1,p_2,p_3)$, $\beta=(\beta_0,\dots,\beta_{d-1},\gamma')$
\begin{align}\label{al}
\begin{split}
&M(p,\beta):=i\beta_0+\gamma'+\sum^{d-1}_{j=1}A_j(p_1+p_2)i\beta_j+p_3+\sum^{d-1}_{j=i}\beta_j^2\\
&A(p):=A_d(p_1+p_2).
\end{split}
\end{align}

\begin{lem}\label{am}
 For $p_1\in\cU$,  $(p_2,p_3)$ in a small enough neighborhood $\omega_2\times\omega_3$ of $(0,0)$, and $\beta$ in a small enough neighborhood $\omega_\beta$ of $0$, there exists a a $C^\infty$ invertible  matrix $T(p,\beta)$  such that $T^{-1}{G}(p,\beta)T$ has the block diagonal form
\begin{align}\label{an}
T^{-1}{G} T=\begin{pmatrix}H&0\\0&P\end{pmatrix},
\end{align}
where
\begin{align}\label{ap}
T(p,\beta)=\begin{pmatrix}I&A^{-1}\\-A^{-1}M+\tau_1&I+\tau_2\end{pmatrix},
\end{align}
with
\begin{align}\label{ej7.65}
\tau_1(p,\beta)=(O(\beta)+O(p_3))^2,\;\tau_2(p,\beta)=O(\beta)+O(p_3).
\end{align}
and
\begin{align}\label{a0}
\begin{split}
&H(p,\beta)=-A^{-1}M+\tau_1\\
&P(p,\beta)=A+A\tau_2.
\end{split}
\end{align}

\end{lem}

\begin{proof}
The proof is a simple computation.  Look for $T$ of the given form and use the invertibility of $A$ to solve for $\tau_1$, $\tau_2$ by contraction.

\end{proof}

\subsection{Computation of the low frequency Evans function}\label{Evans}

\emph{\quad}Consider  the $N\times N$ parabolic problem
\begin{align}\label{a1}
\begin{split}
&\partial_t u^\eps+\sum^d_{1}A_j(u_\eps)\partial_ju_\eps-\eps\Delta u^\eps=f\\
&\partial_x u^\eps|_{x=0}=0\\
&u^\eps|_{t<0}=0
\end{split}
\end{align}
We now examine the Fourier-Laplace transform of the  linearization of \eqref{a1} about a constant state $u=\uu\in \cU$, where $\cU$ is the neighborhood of $0$ specified in Assumption \ref{assumptions}.  Writing $\zeta=(\tau,\eta,\gamma)$ for now and setting
$A_j=A_j(\uu)$ and $\cA(i\eta):=\sum^{d-1}_1 A_j i\eta_j$, we obtain:
\begin{align}\label{a4}
\begin{split}
&(i\tau +\gamma) v +\cA(i\eta) v +A_d v_{x_d}+\eps|\eta|^2 v-\eps v_{x_dx_d}=f\\
&v_{x_d}=0\text{ on }x_d=0.
\end{split}
\end{align}
By multiplying through by $\eps$ and rescaling $x_d$ and frequencies ($x_d\to\frac{x_d}{\eps}$, $\zeta\to \eps\zeta$)  we reduce to the case $\eps=1$. Rewriting \eqref{a4} as a first order system we obtain with $U=(u^1,u^2)^t:=(v,v_{x_d})^t$:
\begin{align}\label{a5}
\begin{split}
&\partial_{x_d}U=G(\zeta)U+F\\
&\Gamma U=u_2=0\text{ on }x_d=0,
\end{split}
\end{align}
where
\begin{align}\label{a6}
F=\begin{pmatrix}0\\-\eps f\end{pmatrix},\;\;G(\zeta):=\begin{pmatrix}0&I\\i\tau+\gamma+\cA(i\eta)+|\eta|^2&A_d\end{pmatrix}.
\end{align}

For $\zeta\neq 0$ let $E^-(\zeta)$ be the stable generalized eigenspace of $G(\zeta)$.  Define the Evans function
\begin{align}\label{a7}
D(\zeta)=\det(E^-(\zeta),\ker\Gamma).
\end{align}
Nonvanishing of the high frequency Evans function (a rescaled version of $D(\zeta)$) was verified in \cite{GMWZ5}, Prop. 3.8.  For fixed $0<r<R$ the fact that $D(\zeta)\neq 0$ for $r\leq |\zeta|\leq R$ is proved in section 4.1 of \cite{GMWZ5}.\footnote{More precisely, the estimates (4.7) and (4.8) in \cite{GMWZ5} are also true with $\Re\lambda$ replaced by $|\lambda|$ on the left.  Those estimates and Sobolev's inequality readily imply the trace estimate
\begin{align}\label{a8}
|v(0)|\leq C(r,R)|v_x(0)|
\end{align}
for $(v(0),v_x(0))\in E^-(\zeta)$ and $\zeta$ in this frequency range.}Thus, we focus now on the low frequency region.

We show that the Evans function vanishes in the limit as $\zeta\to 0$.  For $|\zeta|$ small we conjugate $G(\zeta)$ to a block diagonal form
\begin{align}\label{a9}
S^{-1}(\zeta)G(\zeta)S(\zeta)=\begin{pmatrix}H(\zeta)&0\\0&P(\zeta)\end{pmatrix}:=G_{H,P}\;\;,
\end{align}
where
\begin{align}\label{HP}
H(\zeta)=-A_d^{-1}\left(i\tau+\gamma+\cA(i\eta)\right)+O(\rho^2)\;\;(\rho=|\zeta|),\;P(\zeta)=A_d+O(\rho),
\end{align}
and the conjugator can be chosen to have the form
\begin{align}\label{a10}
S(\zeta)=\begin{pmatrix} I&S_{12}(\zeta)\\S_{21}(\zeta)&I \end{pmatrix}\text{ with }S_{21}(\zeta)=O(\rho).
\end{align}
To construct $S$ one can simply look for a matrix of the form \eqref{a10} satisfying $GS=SG_{H,P}$, and
use the invertibility of $A_d$ to solve for the off-diagonal blocks of $S$ and the error terms in \eqref{HP}.

Writing $GS=SG_{H,P}$ and equating $(1,1)$ entries we obtain
\begin{align}\label{a11}
S_{21}(\zeta)=H(\zeta).
\end{align}
Set $U=S(\zeta)\cU$, where $\cU:=\begin{pmatrix}u_H\\u_P\end{pmatrix}$ and consider the equivalent problem
\begin{align}\label{a12}
\begin{split}
&\cU_{x_d}=G_{H,P}\cU+S^{-1}F\\
&\tilde{\Gamma}(\zeta)\cU:=\Gamma S(\zeta)\cU=H(\zeta)u_H+u_P.
\end{split}
\end{align}

Let $F^-(\zeta)=S^{-1}(\zeta)E^-(\zeta)$.
Since $A_d$ is positive, $F^-(\zeta)=\{(z,0):z\in\bC^N\}$.  On the other hand we have from \eqref{a12}
\begin{align}\label{a13}
\ker\tilde{\Gamma}(\zeta)=\{(w,-H(\zeta)w):w\in\bC^N\}.
\end{align}
This gives immediately
\begin{align}\label{a14}
D(\zeta)=\det(F^-(\zeta),\ker\tilde{\Gamma}(\zeta))=\det H(\zeta) \text{ for }\rho \text{ small},
\end{align}
where each equality holds up to a factor that remains bounded away from zero for $\rho$ small.

\subsection{Resolvent estimates by degenerate symmetrizers}\label{symm}

\emph{\quad}Recall that $F^-(\zeta)=\{\cU=(u_H,0):u_H\in\bC^N\}$.   Thus, for $\cU\in F^-(\zeta)$ we have
\begin{align}\label{a15}
\tilde{\Gamma}(\zeta)\cU=H(\zeta)u_H,
\end{align}
so
\begin{align}\label{a16}
|\cU|=|u_H|=|H^{-1}(\zeta)\tilde{\Gamma}(\zeta)\cU|.
\end{align}
This gives the degenerate trace estimate
\begin{align}\label{a17}
|\tilde{\Gamma}(\zeta)\cU|\geq R(\zeta)|\cU|, \text{ where }R(\zeta):=|H^{-1}(\zeta)|^{-1}, \text{ for }\cU\in F^-(\zeta).
\end{align}

\begin{prop}\footnote{This Proposition does not require $A_d>0$; it remains true when $H$ satisfies the generalized block structure property of \cite{GMWZ6}.}\label{a18}
Let $\rho:=|\zeta|$.
Then for $r>0$ small enough we have
\begin{align}\label{a19}
|R(\zeta)|\geq C(\gamma+\rho^2) \text{ for }0<|\zeta|\leq r.
\end{align}

\end{prop}

\begin{proof}
\textbf{1. }   Write $H(\zeta)=\rho\check H(\check\zeta,\rho)$ and  fix $\underline{\check\zeta}\in \overline{S}^d_+$. For $(\check\zeta,\rho)$ in a neighborhood of $(\underline{\check\zeta},0)$ we use the smooth block reduction of (\cite{GMWZ6}, (3.20))
\begin{align}\label{a20}
V^{-1}\check H V=\mathrm{diag}(\check H_k),
\end{align}
where  $\check H_k$ has spectrum in a small disk centered at  $\umu_k$, for  $\umu_k$ the $k$th distinct eigenvalue of $\check H(\underline{\check\zeta},0)$.  By compactness of  $\overline{S}^d_+$ it suffices to show
\begin{align}\label{a21}
|\check H_k^{-1}(\check\zeta)|\leq C\frac{1}{(\check\gamma+\rho)}
\end{align}
for $\check\zeta$ in a neighborhood of any fixed $\underline{\check\zeta}\in\overline{S}^d_+$ and $\rho$ small.

\textbf{2. } Let $\Sigma$ be a Kreiss symmetrizer constructed as in \cite{GMWZ6}  for $\check H_k$.   The symmetrizer $\Sigma$ satisfies
\begin{align}\label{a21a}
\begin{split}
&(a)\;\Re(\Sigma\check H_k)\geq C(\check\gamma+\rho), \\
&(b)\;|\Sigma|\leq C
\end{split}
\end{align}
near the basepoint.  The estimate \eqref{a21a}(a) implies that $\Sigma\check H_k$ is invertible near the basepoint for $\check\gamma+\rho>0$, and since the same is true for $\check H_k$, we see that $\Sigma$ itelf is invertible near the basepoint for $\check\gamma+\rho>0$.  The estimate \eqref{a21a}(a) also implies

\begin{align}\label{a21b}
C(\check\gamma+\rho)|u|^2\leq \Re(\Sigma\check H_k u,u)\leq |\Sigma\check H_k u||u|,
\end{align}
so
\begin{align}\label{a21c}
|(\Sigma\check H_k)^{-1}|\leq \frac{C}{\check\gamma+\rho}
\end{align}
and thus (since $\Sigma$ is invertible)
\begin{align}\label{a21d}
|\check H_k^{-1}|=|(\Sigma\check H_k)^{-1}\Sigma|\leq \frac{C'}{\check\gamma+\rho}.
\end{align}

\end{proof}

\subsubsection{Resolvent estimates}

With $\cU=\begin{pmatrix}u_H\\u_P\end{pmatrix}$ as in \eqref{a12}, we have
\begin{align}\label{a33a}
F^-(\zeta)=\{(u_H,0):u_H\in\bC^N\},\;F^+(\zeta)=\{(0,u_P):u_P\in\bC^N\},
\end{align}
where $F^\mp(\zeta)$ is the negative (resp. positive) generalized eigenspace of $G_{H,P}(\zeta)$.   Writing $\cU$ as $U$ now, we consider the problem
\begin{align}\label{a34}
\begin{split}
&\partial_{x_d}U=G_{H,P}U+F\\
&\tilde{\Gamma}(\zeta)U=g.
\end{split}
\end{align}

Let $|u_H|_2$ denote the $L^2[0,\infty)$ norm, and let $|u|$ be the norm of the trace at $x_d=0$.
\begin{prop}\label{a35}
Fix $r>0$ small.  For $0<|\zeta|\leq r$ we have the following estimate for solutions of \eqref{a34}:
\begin{align}\label{a36}
(\gamma+\rho^2)^3|u_H|_2^2+|u_P|^2_2 +(\gamma+\rho^2)^2|u_H|^2+|u_P|^2\leq C \left(|F_P|^2_2+(\gamma+\rho^2)|F_H|^2_2+|g|^2\right).
\end{align}

\end{prop}

\begin{proof}

\textbf{1. } We use a degenerate symmetrizer of the form
\begin{align}\label{b1}
S_k(\zeta)=\begin{pmatrix}(\gamma+\rho^2)^2 S_H(\zeta)&0\\0&k S_P(\zeta) I\end{pmatrix},
\end{align}
where $k>0$ will be chosen sufficiently large,  $S_H(\zeta)$ is a standard symmetrizer for the $H(\zeta)$ block (constructed as in \cite{MZ1}, e.g.) and satisfies
\begin{align}\label{b2}
\begin{split}
&S_H^*=S_H\\
&\Re (S_H H)\geq C (\gamma+\rho^2)\\
&S_H u_H\cdot u_H\geq -|u_H|^2,
\end{split}
\end{align}
while $S_P$ satisfies
\begin{align}\label{b2a}
\begin{split}
&S_P^*=S_P\\
&\Re (S_P P)\geq I \\
&S_P u_P\cdot u_P\geq |u_P|^2.
\end{split}
\end{align}

Taking the real part of the $L^2[0,\infty)$ inner product, $(\cdot,\cdot)$, of $-S_k U$ with \eqref{a34} and integrating by parts gives
\begin{align}\label{b3}
\frac{1}{2}S_k U(0)\cdot U(0)+(U,\Re (S_k  G_{H,P})U)=\Re(-S_k U,F),
\end{align}
so
\begin{align}\label{b4}
\begin{split}
&\frac{1}{2}\left(k|u_P(0)|^2-(\gamma+\rho^2)^2|u_H(0)|^2\right)+(\gamma+\rho^2)^3|u_H|_2^2+k|u_P|^2_2\leq \\
&\quad|(\gamma+\rho^2)^2S_H u_H,F_H)|+k|(S_Pu_P,F_P)|\leq \\
&\quad\quad \delta (\gamma+\rho^2)^3|u_H|^2_2+C_\delta (\gamma+\rho^2)|F_H|^2_2+\delta k|u_P|^2_2+C_\delta k|F_P|^2_2.
\end{split}
\end{align}
After absorbing interior terms in the obvious way from the right, it remains only to estimate the boundary terms.

\textbf{2. }
Using \eqref{a17} and \eqref{a19}, we have for  the boundary terms,
\begin{align}\label{b5}
\begin{split}
&k|u_P(0)|^2-(\gamma+\rho^2)^2|u_H(0)|^2=k|u_P(0)|^2+(\gamma+\rho^2)^2|u_H(0)|^2-2(\gamma+\rho^2)^2|u_H(0)|^2\geq\\
&\quad k|u_P(0)|^2+(\gamma+\rho^2)^2|u_H(0)|^2-C\left|\tilde\Gamma(\zeta)\begin{pmatrix}u_H(0)\\0\end{pmatrix}\right|^2\geq\\
&\quad\quad k|u_P(0)|^2+(\gamma+\rho^2)^2|u_H(0)|^2-C|g|^2-C|u_P(0)|^2.
\end{split}
\end{align}
For $k$ large enough \eqref{b5} and \eqref{b4} imply the estimate \eqref{a36}.

\end{proof}

\subsection{The basic variable coefficient $L^2$ estimate}\label{basic}

\begin{notation}\label{c0a}
1.  For $u(x)\in L^2(\overline{\mathbb{R}}_+,H^s(\mathbb{R}^d_{x'}))$ and $\zeta=(\zeta',\gamma)=(\
\zeta_0,\zeta'',\gamma)$, set
\[
|u|_{s,\gamma}=|\langle\zeta\rangle^s \hat{u}(\zeta',x_d)|_{L^2(\zeta',x_d)}.
\]

2.  For $u(x')\in H^s(\mathbb{R}^d)$ set $\langle u\rangle_{s}=|\langle\zeta\rangle^s\hat{u}|_{L^2(\zeta')}$.

3.  Let $\Lambda(\epsilon\zeta)=(1+(\epsilon\gamma)^2+(\epsilon\zeta_0)^2+|\epsilon\zeta''|^4)^{\frac{1}{4}}$.  For $u(x)$, $v(x')$ set
\[
|u|_\Lambda=|\Lambda(\epsilon\zeta) \hat{u}(\zeta',x_d)|_{L^2(\zeta',x_d)},\; \langle v\rangle_\Lambda=|\Lambda(\epsilon\zeta) \hat{v}(\zeta')|_{L^2(\zeta')},
\]
and similarly define $|u|_\phi$, $\langle v\rangle_\phi$ for other weights $\phi=\phi(\eps,\zeta)$.

4. For $u(x)$ set $\langle u\rangle_\phi=\langle u(x',0)\rangle_{\phi}$.

\end{notation}

For given $p(x)$, $F$, and $g$  we now consider the following \emph{linear} boundary problem corresponding to \eqref{aj}, where now we drop tildes and the superscript $\gamma$ on $G$:
\begin{align}\label{c0b}
\begin{split}
&\partial_{x_d} U-\frac{1}{\epsilon}G(p(x),\eps\partial_{x'},\eps\gamma)U=F\\
&\Gamma U=g \text{ on }x_d=0\\
&U=0 \text{ in } x_0<0:
\end{split}
\end{align}
Our goal  is to prove the following (degenerate) $L^2$ estimate for solutions of \eqref{c0b}.

\begin{thm}[Main $L^2$ estimate]\label{c0c}
Under Assumption \ref{assumptions}, there exist positive constants $C$, $\epsilon_0$, $\gamma_0$ such that for all $\gamma>\gamma_0$, $0<\epsilon<\epsilon_0$ with $\epsilon\gamma\leq 1$,  solutions to \eqref{c0b} satisfy
\begin{align}\label{c0d}
\epsilon|U|_0+\epsilon\langle U\rangle_0\leq C\left(\sqrt{\eps}|F|_0+\langle g \rangle_0\right).
\end{align}
\end{thm}

The preceding estimate is a composite of three more precise estimates corresponding to the three natural frequency regimes in the problem, the regimes in which $\epsilon\zeta$ is of small, medium, or large size.

Recall $\beta=(\beta',\gamma')\in\mathbb{R}^d\times\mathbb{R}_+$ is a placeholder for $\epsilon\zeta$. We shall localize with respect to the size of $\beta$ using smooth cutoff functions $\chi_j(\beta)$, $j=S,M,L$,  such that
\begin{align}\label{c0e}
\chi_S(\beta)+\chi_M(\beta)+\chi_L(\beta)= 1,
\end{align}
where for some constants $R_1$ (sufficiently small), $R_2$ (sufficiently large)
\begin{align}\label{c0f}
\begin{split}
&\text{ supp }\chi_S\subset \{0\leq |\beta|\leq R_1\} \\
&\text{ supp }\chi_M\subset \{\frac{3}{4}R_1\leq |\beta|\leq R_2\} \\
&\text{ supp }\chi_L\subset \{\frac{3}{4}R_2\leq |\beta|\}.
\end{split}
\end{align}

\begin{notation}\label{c0g}
1. We will occasionally use the symbol $\chi_M$ to denote a cutoff distinct from the one in \eqref{c0f}, but also supported in a bounded region strictly away from the origin.  Similar statements apply as well to $\chi_S$, $\chi_L$.

2. Choose smooth cutoffs $\chi_1(\beta)$, $\chi_2(\beta)$ identically equal to $1$ near $\beta=0$ and compactly supported in $\omega_\beta$ such that
\begin{align}\label{c7a}
\chi_1\chi_2=\chi_1,\;\;\chi_S\chi_1=\chi_S.
\end{align}

3.  The symbol $r_0$ will always denote a symbol or operator of order zero.

4. Denote by $O(\epsilon D)$ a semiclassical operator with symbol $s(x,\beta)$ such that $s=\beta\cdot f(x,
\beta)$ for some smooth $f$.\footnote{Since $s$ must be bounded, we must then have $|f|=O(1/|\beta|)$ for $|\beta|$ large.}  $O(\epsilon)$ denotes an operator with symbol $s=\epsilon f(x,\beta)\in\mathcal{S}_\infty$.

In a similar way define $O(\epsilon^2)$, $O((\epsilon D)^2)$, etc..  When speaking of symbols instead of operators we'll use, as before, the notation $O(\epsilon\zeta)$, $O(\epsilon)$, etc..  In ambiguous cases like $O(\epsilon)$, the intent (symbol or operator) should be clear from the context.

5. Write the solution to \eqref{c0b} as $U=(u,v)$.  Define
\begin{align}\label{c0h}
U_\Lambda=(\Lambda u,v),
\end{align}
where $\Lambda(\epsilon D)$ is the multiplier associated to the symbol defined in Notation \ref{c0a}.

\end{notation}

Here are the  estimates by frequency size:

\begin{prop}\label{c2a}
Using the notation just introduced, we have the following estimates for solutions to \eqref{c0b}.  Let $R_1$, $R_2$ be as in \eqref{c0f}.  For $R_1$ sufficiently small and $R_2$ sufficiently large, there exist constants $C$, $\gamma_1$, $\epsilon_1$ such that for all $\gamma>\gamma_1$, $0<\epsilon<\epsilon_1$ with $\epsilon\gamma\leq 1$
\begin{align}\label{c2}
\begin{split}
&(a)\;|\chi_{S,D} U|_{\eps\gamma^{\frac{3}{2}}+\eps^{\frac{5}{2}}\rho^3} +\langle\chi_{S,D}U\rangle_{\eps\gamma+\eps^2\rho^2}\leq\\
& \qquad C\left(\sqrt{\eps}|F|_0+\langle g\rangle_0+\epsilon|U^{}|_0+|\chi_{2,D}U|_{\eps^{\frac{3}{2}}\rho+\eps\gamma+\eps^2\rho^2}+|\chi_{M,D}U|_0+\eps\langle U \rangle_0\right)\\
&(b)\;|\chi_{M,D} U^{}|_0+\sqrt{\epsilon}\langle \chi_{M,D} U^{}\rangle_0\leq C\left(\epsilon|F^{}|_0+\sqrt{\eps}\langle g\rangle_0+\epsilon|U^{}|_0+\epsilon\langle U^{}\rangle_0\right)\\
&(c)\;|\chi_{L,D} U^{}_\Lambda|_{\sqrt{\Lambda}}+\sqrt{\epsilon}\langle\chi_{L,D} U^{}_\Lambda\rangle_0\leq C\left(\epsilon|F^{}|_{\Lambda^{-1/2}}+\sqrt{\eps}\langle g\rangle_0+\epsilon|U^{}_\Lambda|_{\Lambda^{-1/2}}+\epsilon\langle U^{}_\Lambda\rangle_{\Lambda^{-1/2}}\right).
\end{split}
\end{align}

\end{prop}

\begin{proof}
The estimates \eqref{c2a}(b),(c) are proved in \cite{MZ1}.  In the latter case we have applied the high frequency estimate of Proposition 4.6 of \cite{MZ1} after commuting $(\Lambda^{-1/2})_D$ through the problem. We concentrate now on proving \eqref{c2a}(a).

\textbf{a. Localize to small frequency region.} Commuting $\chi_{S,D}$ through \eqref{c0b}, we see that    $\chi_{S,D}U$ satisfies
\begin{align}\label{c3}
\begin{split}
&\chi_{S,D} U_{x_d}-\frac{1}{\epsilon}{G}_D\chi_{S,D} U=\chi_{S,D} F+\frac{1}{\epsilon}[\chi_{S,D},{G}_D]U\\
&\Gamma\chi_{S,D}U=\chi_{S,D} g\text{ on } x_d=0.
\end{split}
\end{align}
There is a high frequency contribution to the commutator because of the $x'$ dependence of ${G}$, and to get a good estimate for this we use the semiclassical calculus.\footnote{Even though the symbol of $G$ is not bounded, one can use and directly estimate  the formula for the remainder given in (A.6) of \cite{GMWZ2} to prove \eqref{c4}.}
Since
\begin{align}\label{c4}
\chi_{S,D}{G}_D=(\chi_S {G})_D+\frac{\epsilon}{i}(\partial_{\beta'}\chi_S\partial_{x'}{G})_D+\epsilon^2 r_0,
\end{align}
we have
\begin{align}\label{c5}
\frac{1}{\epsilon}[\chi_{S,D},{G}_D]U=\frac{1}{i}(\partial_{\beta'}\chi_S\partial_{x'}{G})_D U+\epsilon r_0 U.
\end{align}
Thus $U_a=\chi_{S,D} U$ satisfies
\begin{align}\label{c6}
\begin{split}
\partial_{x_d}U_a-\frac{1}{\epsilon}{G}_D U_a=F_a\\
\Gamma U_a=g_a \text{ on }x_d=0,
\end{split}
\end{align}
where
\begin{align}\label{c7}
|F_a|_0\leq C|F|_0 +|(\partial_{\beta'}\chi_S\;r_0)_D U|_0 +\epsilon|U|_0,\;\;\langle g_a\rangle_0\leq \langle g\rangle_0.
\end{align}

To prove \eqref{c2}(a) it suffices to prove the same estimate with $\chi_{S,D}U$, $F$, and $g$ replaced by $U_a$, $F_a$ and $g_a$.

\textbf{b. Conjugate to $G_{HP,D}$. }Let $T(p,\beta)$  be the conjugator constructed in Lemma \ref{am} and set
\begin{align}\label{c8}
&G_{HP}=\begin{pmatrix}H&0\\0&P+\eps r_0\end{pmatrix}
\end{align}
Extend  $T(p(x),\beta)$ smoothly to all $\beta\in\mathbb{R}^d\times\overline{\mathbb{R}}_+$ as a semiclassical symbol with a uniformly bounded inverse, and use the calculus to  construct right and left (approximate) inverses $T_{-1,D}$ satisfying
\begin{align}\label{c9}
\begin{split}
&T_{D}T_{-1,D}=I+\epsilon^2 r_0\\
&T_{-1,D}T_{D}=I+\epsilon^2 r_0.
\end{split}
\end{align}
The right and left inverses are not equal, but we use the same notation for both.  The symbol $T_{-1}$ in each case has the form
\begin{align}\label{c10}
T_{-1}(p(x),\beta)=T^{-1}+\epsilon r_0.
\end{align}

Defining $V=T_{-1,D}U_a$, we have
\begin{align}\label{c11}
\begin{split}
&(a)\;T_{D}V=U_a +\epsilon^2 r_0 U_a\\
&(b)\;(\partial_{x_d} T_{D})V+T_{D}\partial_{x_d} V=\partial_{x_d} U_a+O(\epsilon)(r_0U_a+\epsilon F_a)=\\
&\qquad \frac{1}{\epsilon}G_{D}T_{D}V +F_a +O(\epsilon)(r_0U_a+\epsilon F_a).
\end{split}
\end{align}

We have the following symbol equalities
\begin{align}\label{c12}
\begin{split}
&(a)\;T=\begin{pmatrix}I&A^{-1}\\0&I\end{pmatrix}+O(\epsilon\zeta)+O(\epsilon)\\
&(b)\;T_{-1}=\begin{pmatrix}I&-A^{-1}\\0&I\end{pmatrix}+O(\epsilon\zeta)+O(\epsilon)\\
&(c)\;T_{-1}\partial_{x_d}T=\begin{pmatrix}0&r_0\\0&0\end{pmatrix}+O(\epsilon\zeta)+O(\epsilon)\\
&(d)\;G T\chi_2=\begin{pmatrix}0&I\\0&A\end{pmatrix}\chi_2(\epsilon\zeta)+O(\epsilon\zeta)+O(\epsilon)\\
&(e)\;\eps\partial_{\beta'}T_{-1}=\epsilon O(\epsilon\zeta)+O(\epsilon)\\
&(f)\;\frac{1}{\epsilon}(\eps\partial_{\beta'}T_{-1})\partial_{x'}(G T\chi_2)=\begin{pmatrix}0&r_0\\0&r_0\end{pmatrix}\chi_2(\epsilon\zeta)+O(\epsilon\zeta)+O(\epsilon)\\
&(g)\;\frac{1}{\epsilon}T_{-1}G T\chi_2=\frac{1}{\epsilon}G_{HP}\chi_2+\begin{pmatrix}0&r_0\\0&r_0\end{pmatrix}\chi_2(\epsilon\zeta)+O(\epsilon\zeta)+O(\epsilon).
\end{split}
\end{align}
For \eqref{c12}(g) we  used \eqref{c10}, \eqref{an}, and \eqref{c12}(d).

Applying the operator $T_{-1,D}$ to \eqref{c11}(b) and using the semiclassical calculus, we obtain in view of the symbol equalities \eqref{c12}:
\begin{align}\label{c13}
\partial_{x_d}V=\frac{1}{\epsilon}\begin{pmatrix}H_{D}&\epsilon
r_0\\0&P_{D}+\epsilon r_0\end{pmatrix}\chi_{2,D} V+r_0F_a+O(\epsilon)U_a+O(\epsilon D)V+O(\epsilon)V.
\end{align}
Observe that terms on the right in \eqref{c12}(c),(f), and (g) all make contributions to the $r_0$ entries of the first matrix on the right in \eqref{c13}.
Using the calculus to commute $\chi_{1,D}$ through  \eqref{c13}, we obtain
\begin{align}\label{c14}
\partial_{x_d} (\chi_{1,D}V)=\frac{1}{\epsilon}\begin{pmatrix}H_{D}&\epsilon
r_0\\0&P_{D}+\epsilon r_0\end{pmatrix}(\chi_{1,D}V)+r_0F_a +O(\eps)U+O(\eps D)\chi_{1,D}U+(r_0\partial_{\beta'}\chi_1 )_DU.
\end{align}
Next define
\begin{align}\label{c14a}
F_b:=r_0F_a +O(\eps)U+O(\eps D)\chi_{1,D}U+(r_0\partial_{\beta'}\chi_1 )_DU,
\end{align}
and observe that since $U_a=T_D V-\eps^2 r_0 U_a=T_D\chi_{1,D}V+\eps^2r_0U$ and $U_b=\chi_{1,D}V$ satisfies
\begin{align}\label{c15}
\begin{split}
&\partial_{x_d}U_b=\frac{1}{\epsilon}\begin{pmatrix}H_{D}&\epsilon r_0\\0&P_{D}+\epsilon r_0\end{pmatrix}U_b+F_b\\
&\Gamma  T_DU_b=g_a+\epsilon^2 r_0 U:=g_b \text{ on }x_d=0,
\end{split}
\end{align}
to prove \eqref{c2}(a) it now suffices to prove the same estimate with $\chi_{S,D}U$, $F$, and $g$ replaced by $U_b$, $F_b$ and $g_b$.   Observe that $\sqrt{\eps}F_b$ is a sum of terms including $\sqrt{\eps}O(\eps D)\chi_{1,D}U$.  The latter term is absorbed using the following Lemma, whose proof is elementary.

\begin{lem}\label{c16}
 Fix $\delta>0$. Then for $\gamma$ large we have
 \begin{align}
\begin{split}
 &(1)\;\eps^{\frac{3}{2}}\rho\leq \delta \left(\eps\gamma^{\frac{3}{2}}+\eps^{\frac{5}{2}}\rho^3\right)\\
&(2)\;\eps^2\rho^2\leq \delta \left(\eps\gamma^{\frac{3}{2}}+\eps^{\frac{5}{2}}\rho^3\right).
\end{split}
\end{align}

\end{lem}

Define
\begin{align}\label{c17}
G_{b}(p(x),\beta)=\begin{pmatrix}H&\epsilon r_0\\0&P+\epsilon r_0\end{pmatrix}.
\end{align}
A direct computation using the invertibility of $P$ shows that for $\beta\in\omega_\beta$ one can choose a matrix symbol $T_{c}$ of the form
\begin{align}\label{c18}
T_{c}(p(x),\beta)=\begin{pmatrix}I&\epsilon r_0\\0&I\end{pmatrix}
\end{align}
such that
\begin{align}\label{c19}
T_{c}^{-1}G_bT_{c}=\begin{pmatrix}H&0\\0&P+\epsilon r_0\end{pmatrix}=G_{HP}.
\end{align}
As before we extend and invert $T_{c,D}$.  The operator $T_{c,-1,D}$ associated to the symbol
\begin{align}\label{c20}
T_{c,-1}=\begin{pmatrix}I&-\epsilon r_0\\0&I\end{pmatrix}
\end{align}
is easily seen to be a right and left inverse satisfying  the analogue of \eqref{c9}.

Redefine $V=T_{c,-1,D}U_b$.  Now repeat the preceding argument line for line, but note, for example, that instead of \eqref{c12}(c),(e),(f) we have, respectively,
\begin{align}\label{c21}
\begin{split}
&T_{c,-1}\partial_{x_d} T_{c}=O(\epsilon)\\
&\eps\partial_{\beta'}T_{c,-1}=\begin{pmatrix}0&\epsilon^2 r_0\\0&0\end{pmatrix}\\
&\frac{1}{\epsilon}(\eps\partial_{\beta'}T_{c,-1})\partial_{x'}(G_{b}T_{c})=O(\epsilon).
\end{split}
\end{align}

We set $U_{c}=\chi_{1,D}V$ and use the calculus just as before to find that $U_c$ satisfies
\begin{align}\label{c22}
\begin{split}
&(a)\;\partial_{x_d}U_{c}=\frac{1}{\epsilon}G_{HP,D} U_{c}+F_c\\
&(b)\;\Gamma T_DT_{c,D} U_{c}=g_a+\epsilon^2 r_0 U:=g_c\text{ on }x_d=0,
\end{split}
\end{align}
where $F_c$ has a formula like \eqref{c14a} (with $F_b$ in place of $F_a$).  Thus, to prove \eqref{c2}(a) it now suffices to prove the same estimate with $\chi_{S,D}U$, $F$, and $g$ replaced by $U_c$, $F_c$ and $g_c$.

\textbf{c. Block structure. }Recall that $G_{HP}$ is given  by \eqref{c8}, where $H(p,\beta)$ and $P(p,\beta)$ are as in Lemma \ref{am}. Let $p'=(p_1,p_2)$, define $\cH(p',\beta)=H(p',0,\beta)$, $\cP(p',\beta)=P(p',0,\beta)$, and set
\begin{align}\label{c22a}
\cG_{HP}(p',\beta)=\begin{pmatrix}\cH(p',\beta)&0\\0&\cP(p',\beta)\end{pmatrix}.
\end{align}
Note that for $\beta\in\omega_\beta$
\begin{align}\label{c22b}
\begin{pmatrix}H(p,\beta)&0\\0&P(p,\beta)\end{pmatrix}=\cG_{HP}(p',\beta)+\begin{pmatrix}O(p_3)&0\\0&O(p_3)\end{pmatrix},
\end{align}
and thus
\begin{align}\label{c22c}
G_{HP}(p(x),\beta)=\cG_{HP}(p'(x),\beta)+\begin{pmatrix}\eps r_0&0\\0&\eps r_0\end{pmatrix}.
\end{align}
To proceed further we need to conjugate $\cG_{HP}$ to block structure form, which is especially simple in the totally incoming case.   Introduce polar coordinates
\begin{align}\label{c22d}
\beta=\rho'\hat\beta,\text{ where }\hat\beta\in S^d_+=\{(\widehat{\beta'},\whgamma)\in S^d:\whgamma\geq 0\},\;\;\rho'=|\beta|
\end{align}
and write
\begin{align}\label{c22e}
\cH(p',\beta)=\rho'\hat\cH(p',\hat \beta,\rho').
\end{align}
Similarly we set $\hat\zeta=(\widehat{\zeta'},\widehat{\gamma})=\zeta/|\zeta|$ and $\rho=|\zeta|$.

\begin{prop}[Block structure]\label{c23}
Let $\underline{p'}\in \cU$.  For each $\underline{\widehat{\beta}}\in S^d_+$ there is a neighborhood $\cO$ of $(\underline{p'},\underline{\widehat{\beta}},0)$ in $\mathbb{R}^{2N}\times S^d_+\times\overline{\mathbb{R}}_+$ and  a $C^\infty$ matrix $V(p',\widehat{\beta},\rho')$ defined on $\cO$ such that $V^{-1}\hat\cH V$ has the following block diagonal structure:

1. If $\underline{\whgamma}>0$, then $V^{-1}\hat\cH V=Q$ where $\Re Q=(Q+Q^*)/2<c<0$.

2. When  $\underline{\whgamma}=0$, we have
\begin{align}\label{c24}
V^{-1}\hat\cH V=\begin{bmatrix} q_1&\cdots&0\\\vdots&\ddots&\vdots\\0&\cdots& q_N\end{bmatrix}(p',\widehat{\beta},\rho'):=\hat \mrh(p',\widehat{\beta},\rho'),
\end{align}
where the $q_j$ are scalars, not necessarily distinct, such that $\Re q_j=0$ when $\widehat{\gamma'}=\rho'=0$, $\partial_{\widehat{\gamma'}}(\Re q_j)<c<0$ and $\partial_{\rho'}(\Re q_j)<c<0$.

There is a $C^\infty$ matrix $W(p',\beta)$ defined on a neighborhood  of $(\underline{p'},0)$ such that
\begin{align}\label{c25}
W^{-1}\cP W= \mrp(p',\beta), \text{ where }\Re \mrp>C_p>0.
\end{align}
\end{prop}

\begin{proof}
A general block structure result that applies in our case is Lemma 2.10 of \cite{MZ1}.  The simplification due to the totally incoming assumption, $A_2>0$, is explained in Corollary 7.9 of \cite{GMWZ6}.\footnote{This assumption rules out glancing modes, and also guarantees that all blocks are $1\times 1$ near points where $\underline{\whgamma}=0$.}
\end{proof}

\textbf{d. Degenerate symmetrizers. }The simple block structure described in Proposition \ref{c23} permits the following simple construction of degenerate symmetrizers.  Let $\Omega_T=\{x\in \overline{\bR}^{1+d}_+:0\leq x_0\leq T\}$ and $b\Omega_T=\{x\in\Omega_T:x_d=0\}$.

\begin{prop}\label{c26}
Fix $\ux\in b\Omega_T$ and $\underline{\hzeta}\in S^d_+$ and consider a neighborhood
$\cO$ of $(p'(\ux),\hzeta,0)$ in $\mathbb{R}^{2N}\times S^d_+\times\overline{\mathbb{R}}_+$ on which a conjugator $V(p',\hat \beta,\rho')$ as in Proposition \ref{c23} is defined.  For $(x,\zeta)$ such that $(p'(x),\hat\zeta,\eps\rho)\in\cO$, define
\begin{align}\label{c27}
S(\hat\zeta,\eps\rho)=\begin{pmatrix}S_h&0\\0&S_p\end{pmatrix},
\end{align}
where the $N\times N$ matrices $S_h$, $S_p$ are given by
\begin{align}\label{c28}
S_h=-(\eps^2\gamma^2+\eps^4\rho^4)I_N,\quad S_p=K I_N, \;\;K>0,
\end{align}
and set $\mrh(p',\widehat{\beta},\rho'):=\rho'\hat \mrh(p',\widehat{\beta},\rho')$.
Then, depending on $\cO$, either
\begin{align}\label{c29}
\Re \frac{1}{\eps}S_h(\hat\zeta,\eps\rho)\mrh(p'(x),\hat\zeta,\eps\rho)=\rho (\eps^2\gamma^2+\eps^4\rho^4)k(x,\hat\zeta,\eps\rho), \text{ where }k>C>0
\end{align}
or
\begin{align}\label{c30}
\Re \frac{1}{\eps}S_h\mrh=\begin{bmatrix} (\eps^2\gamma^2+\eps^4\rho^4)(\gamma b_{0,1}+\eps\rho^2b_{1,1})&\cdots&0\\\vdots&\ddots&\vdots\\0&\cdots& (\eps^2\gamma^2+\eps^4\rho^4)(\gamma b_{0,N}+\eps\rho^2b_{1,N})\end{bmatrix},
\end{align}
where $b_{0,j}(x,\hat\zeta,\eps\rho)>C>0$, $b_{1,j}(x,\hat\zeta,\eps\rho)>C>0$.   Also,
\begin{align}\label{c30a}
\Re \frac{1}{\eps}S_p \;\mathrm{p}=\frac{1}{\eps}K\;\Re\mrp(p'(x),\beta) \geq \frac{1}{\eps}KC_p,  \text{ where }K>0, C_p>0.
\end{align}
Finally, for $u=(u_h,u_p)\in\bC^{2N}$ we have
\begin{align}\label{c30b}
(S(\hat\zeta,\eps\rho)u,u)\geq K|u_p|^2-(\eps^2\gamma^2+\eps^4\rho^4)|u_h|^2.
\end{align}

\end{prop}
\begin{proof}
The equalities \eqref{c30a} and \eqref{c30b} are immediate.
Consider $\cO$ as in case 2 of Proposition \ref{c23}.    The properties of $q_j$ stated there imply
\begin{align}\label{c31}
q_j(p',\hat\beta,\rho')=\whgamma a_{0,j}+\rho'a_{1,j}+i d_j
\end{align}
where $a_{0,j}$, $a_{1,j}$, $d_j$ are real functions of $(p',\hat\beta,\rho')$ such that $a_{0,j}<c<0$, $a_{1,j}<c<0$.
Setting $\beta=\eps \zeta$ and noting that $\hat\beta=\hat\zeta$ and $\rho'=\eps\rho$, we see that \eqref{c30} holds with
\begin{align}\label{c32}
b_{0,j}(x,\hat\zeta,\eps\rho)=-a_{0,j}(p'(x),\hat\zeta,\eps\zeta),\;b_{1,j}(x,\hat\zeta,\eps\rho)=-a_{1,j}(p'(x),\hat\zeta,\eps\zeta).
\end{align}
Similarly, \eqref{c29} holds when $\cO$ is as in case 1 of Proposition \ref{c23}.

\end{proof}

\textbf{e. Microlocalize. }Next we construct a pseudodifferential partition of unity that will allow us to prove estimates using the symmetrizers just constructed.

Let $K_{p'}$ be the compact set given by the closure of the range of $p'(x)$ on $\Omega_T$.\footnote{Recall that, by finite propagation speed of the hyperbolic problem, $p(x)$ is constant outside a compact subset of $\Omega_T$.}  For $\delta>0$ small we can choose an open cover $\{\cO_k\}$ of $K_{p'}\times S^d_+ \times [0,\delta]$ by sets $\cO_k$ on which conjugators as in Proposition \ref{c23} are defined.  Choose partitions of unity $\kappa_l(x)$ and $\psi_m(\hat\beta,\rho')$ subordinate to open covers of $\Omega_T$ and $S^d_+\times [0,\delta]$, respectively, with the property that for any pair $(l,m)$ there exists a $k$ such that
\begin{align}\label{c33}
(x,\hat\beta,\rho')\in \mathrm{supp} \;\kappa_l(x)\psi_m(\hat\beta,\rho')\Rightarrow (p'(x),\hat\beta,\rho')\in \cO_k.
\end{align}
Next define the bounded families of classical symbols
\begin{align}\label{c34}
\phi^\eps_{l,m}(x,\zeta)=\kappa_l(x)\psi_m(\hat\zeta,\eps\rho).
\end{align}
After re-indexing this family as $\phi^\eps_l(x,\zeta)$, we rewrite the unknown $U_c$ in \eqref{c22} as the finite sum
\begin{align}\label{c35}
U_c=\sum_l U_l, \text{ where }U_l:=\phi^\eps_{l,D}U_c.
\end{align}

Next we commute $\phi^\eps_{l,D}$ through the problem \eqref{c22}. Observe that \eqref{c22}(a) is unchanged if $G_{HP,D}$ is replaced by $G_{HP,D}\;\chi_{2,D}$, and henceforth we include the factor $\chi_2$ (often suppressed) in the definitions of $\cH$ and $\cP$ \eqref{c22a}.  Observe that $\cP\chi_2$ is a classical symbol of order zero and that
\begin{align}\label{c36}
\cH=\eps\rho \hat \cH(p'(x),\hat\zeta,\eps\rho)\chi_2(\eps\zeta)=\eps \cH^*,
\end{align}
where $H^*$ is a classical symbol of order one.  The leading terms in the symbols of the commutators $[\cH^*_D,\phi^\eps_{l,D}]$ and  $\frac{1}{\eps}[\cP_D,\phi^\eps_{l,D}]$ are, respectively,
\begin{align}\label{c36a}
\begin{split}
&(a)\;\partial_{\zeta'}\cH^*\;D_{x'}\phi^\eps_l-\partial_{\zeta'}\phi^\eps_l\; D_{x'}\cH^*\in\cC^0\\
&(b)\;\frac{1}{\eps}\left(\partial_{\zeta'}\cP\; D_{x'}\phi^\eps_l-\partial_{\zeta'}\phi^\eps_l\;D_{x'}\cP\right)\in\frac{1}{\eps}\;\cC^{-1}.
\end{split}
\end{align}
Thus, from \eqref{c22c} and \eqref{c22}(a)we find:
\begin{align}\label{c37}
\partial_{x_d}U_{l}=\frac{1}{\epsilon}\cG_{HP,D} U_{l}+\begin{pmatrix} r_0&0\\0& \frac{1}{\eps}r_{-1}\end{pmatrix}U_c+\phi^\eps_{l,D}F_c
\end{align}
where $r_{-1}\in\cC^{-1}$.

The boundary operator $\Gamma T_DT_{c,D}$ can be viewed as an element of $\cC^0$.  Its leading symbol is
\begin{align}\label{c38}
\Gamma TT_c=\begin{pmatrix}0&I\end{pmatrix}\begin{pmatrix}I&A^{-1}\\-A^{-1}M+\tau_1&I+\tau_2\end{pmatrix}\begin{pmatrix}I&\eps r_0\\0&I\end{pmatrix}=
\begin{pmatrix}\cH+\eps r_0&I+\eps r_0+O(\eps\zeta)\chi_2\end{pmatrix}.
\end{align}
Using \eqref{c36a}(a) (and a similar computation for the commutator $[O(\eps D)\chi_{2,D},\phi^\eps_{l,D}]$) we find
\begin{align}\label{c39}
\Gamma T_DT_{c,D}U_l=\phi^\eps_{l,D}g_c+\eps r_0 U_c.
\end{align}

For a fixed $l$, set
\begin{align}\label{c39a}
U_c=(U_h,U_p)^t,\;\;U_l=\phi^\eps_{l,D}U_c=(u_h,u_p)^t, \text{ and }\phi^\eps_{l,D}F_c=(f_h,f_p)^t.
\end{align}
 We can rewrite \eqref{c37} and \eqref{c39} as follows:
\begin{align}\label{c40}
\begin{split}
&(a)\;\partial_{x_d}u_h=\frac{1}{\eps}\cH_Du_h+r_0U_h+f_h\\
&(b)\;\partial_{x_d}u_p=\frac{1}{\eps}\cP_Du_p+\frac{r_{-1}}{\eps}U_p+f_p,\;\;r_{-1}\in\cC^{-1}\\     &(c)\;\cH_Du_h+I^\eps_Du_p=r_0 g_c+\eps r_0 U_c.
\end{split}
\end{align}
Here we have set $I^\eps_D=I+O(\eps D)\chi_{2,D}$ and used \eqref{c38} and the semiclassical calculus to compute $\Gamma T_D T_{c,D}$.

\textbf{f. Conjugate with $V_D$ and $W_D$. }For the same fixed $l$ as in \eqref{c39a} suppose that for $(x,\zeta)\in\mathrm{supp}\;\phi^\eps_l(x,\zeta)$, $(p'(x),\hat\zeta,\eps \rho)$  is contained in an open set $\cO$ as in case 2 of Proposition \ref{c23}\footnote{We omit the details for case 1, which is similar but easier.}, and let $V(p'(x),\hat\zeta,\eps\rho)\in\cC^0$ be the corresponding conjugator as in \eqref{c24}.   Extend $V$ and (approximately) invert $V_D$ in the classical calculus to obtain left and right inverses such that
\begin{align}\label{c41}
V_DV^{-1}_D=I+r_{-1},\;\;  V^{-1}_DV_D=I+r_{-1}, \text{ where }r_{-1}\in\cC^{-1}. \end{align}
Defining $w_h=V^{-1}_Du_h$ we obtain by a computation similar to \eqref{c11}(b)
\begin{align}\label{c42}
(\partial_{x_d}V_D)w_h+V_D\partial_{x_d}w_h=\frac{1}{\eps}\cH_D V_D w_h+r_0 U_h+r_0 f_h.
\end{align}
Here we have used the fact that
\begin{align}\label{c43}
\frac{1}{\eps}\cH_D \text{ is a bounded family in }\cC^1.
\end{align}
Applying $V^{-1}_D$ to \eqref{c42} and using \eqref{c43} again, we find
\begin{align}\label{c44}
\partial_{x_d}w_h=\frac{1}{\eps}V^{-1}_D\cH_D V_D w_h+r_0 U_h+r_0f_h=\frac{1}{\eps}\mrh_Dw_h+ r_0 U_h+r_0f_h, \end{align}
where $\mrh=\mrh(p'(x),\zeta,\eps\rho)$ is as in \eqref{c29}, and hence $\frac{1}{\eps}\mrh_D\in\cC^1$.

Similarly, extend the conjugator $W(p',\eps\beta)$ in \eqref{c25} and construct approximate inverses of $W_D$ in the semiclassical calculus such that
\begin{align}\label{c45}
W_DW^{-1}_D=I+\eps r_{0},\;\;  W^{-1}_DW_D=I+\eps r_{0}.
\end{align}
Defining $w_p=W^{-1}_D u_p$ we obtain by using the semiclassical calculus and computing as above\footnote{Here we have used $r_0\chi_2(\eps\zeta)\in\frac{1}{\eps}\cC^{-1}$, which holds since $\eps\chi_2(\eps\zeta)\in\cC^{-1}.$}
\begin{align}\label{c46}
\partial_{x_d}w_p=\frac{1}{\eps}\mrp_D w_p+\frac{r_{-1}}{\eps}U_p +r_0 f_p,\text{ where }r_{-1}\in\cC^{-1}. \end{align}

\textbf{g. Interior estimates. }We quantize $S_h$ and $S_p$ as in \eqref{c28} by setting
\begin{align}\label{c47}
S_{h,D}=-\left(\eps^2\gamma^2 I_N+\eps^4(\rho^4I_N)_D\right)\chi_{2,D},\; S_{p,D}=KI_N\chi_{2,D}, \end{align}
where $(\rho^4I_N)_D \in\cC^4$.  Pairing \eqref{c44} and \eqref{c46} with $S_{h,D}w_h$ and $S_{p,D}w_p$, we obtain the identities
\begin{align}\label{c48}
\begin{split}
&\langle S_{h,D}w_h,w_h\rangle+\Re \frac{1}{\eps}(S_{h,D}\;\mrh_D w_h,w_h)=-2\Re(r_0 U_h+r_0f_h,S_{h,D}w_h)\\
&\langle S_{p,D}w_p,w_p\rangle+\Re \frac{1}{\eps}(S_{p,D}\;\mrp_D w_p,w_p)=-2\Re(\frac{r_{-1}}{\eps} U_p+r_0f_p,S_{p,D}w_p).
\end{split}
\end{align}
Since
\begin{align}\label{c47a}
\begin{split}
&|(r_0f_h,S_{h,D}w_h)|=|\left(r_0f_h,(\eps^2\gamma^2+\eps^4\rho^4)_D\;\chi_{2,D}w_h\right)|\\
&\quad=\left|\left(\sqrt{\eps}(\eps^2\gamma^2+\eps^4\rho^4)^{\frac{1}{4}}_D\;\chi_{2,D}\;r_0f_h,\frac{1}{\sqrt{\eps}}(\eps^2\gamma^2+\eps^4\rho^4)^{\frac{3}{4}}_D\; \chi_{2,D}w_h\right)\right| \leq\\
&\qquad\qquad C_\delta\;\eps |f_h|_0^2+ \delta |w_h|_{\eps\gamma^{3/2}+\eps^{5/2}\rho^3}^2,
\end{split}
\end{align}
we see that
\begin{align}\label{c49}
-2\Re(r_0 U_h+r_0f_h,S_{h,D}w_h)\leq C|U_h|_{\eps\gamma+\eps^2\rho^2}^2+C_\delta\;\eps |f_h|_0^2+ \delta |w_h|_{\eps\gamma^{3/2}+\eps^{5/2}\rho^3}^2.
\end{align}
Similarly,   since $|r_{-1}U_p|_0\leq \frac{C}{\gamma}|U_p|_0$,
\begin{align}\label{c50}
-2\Re(\frac{r_{-1}}{\eps} U_p+r_0f_p,S_{p,D}w_p)\leq  \frac{CK}{\gamma\eps}|U_p|^2_0+C_\delta\eps|f_p|^2_0+\frac{\delta}{\eps}|w_p|^2_0.
\end{align}

Next set $\mrh^*=\frac{1}{\eps}h\in\cC^1$ and note that
\begin{align}\label{c51}
|(S_{h,D}\mrh^*_D w_h,w_h)-((S_h\mrh^*)_Dw_h,w_h)|\leq C\eps^2\gamma^2|w_h|^2_0+C\eps^4|w_h|_{\rho^2}^2.
\end{align}
Here we have used the classical calculus to obtain, for example,
\begin{align}\label{c52}
\left(\eps^4\rho^4\chi_2(\eps\zeta)\right)_D\mrh^*_D=   (\eps^4\rho^4\chi_2(\eps\zeta)\mrh^*)_D+\eps^4 r_{4,D},\;\;r_4\in\cC^4.
\end{align}
The semiclassical calculus implies $\chi_{2,D}\mrp_D=(\chi_2\mrp)_D+\eps r_0$, so
\begin{align}\label{c53}
\left|\frac{1}{\eps}(S_{p,D}\;\mrp_D w_p,w_p)-\frac{1}{\eps}((S_{p}\mrp)_D w_p,w_p)\right|\leq C|w_p|^2_0.
\end{align}

   Now we can use \eqref{c30} and the Garding inequality for the classical calculus to get estimates from below:
\begin{align}\label{c54}
\begin{split}
&\Re \frac{1}{\eps}((S_{h}\mrh)_D w_h,w_h)\geq C(\eps^2\gamma^3|w_h|^2_0+\eps^5|w_h|_{\rho^3}^2)-C(\eps^2\gamma^3|U_h|^2_{-1}+\eps^5|U_h|^2_{\rho^2})\\
&\qquad\qquad\qquad -C(\eps^3\gamma^2|U_h|^2_0+\eps^4\gamma|U_h|^2_\rho).
\end{split}
\end{align}
To obtain \eqref{c54} we have used, for example, the Garding estimate:
\begin{align}\label{c55}
\Re\left((\eps^5\rho^6 b_{1,j}\chi_2(\eps\zeta))_Dw_{h,j},w_{h,j}\right)\geq    C\eps^5|w_{h,j}|_{\rho^3}^2)-C\eps^5|U_h|^2_{\rho^2},
\end{align}
where $b_{1,j}$ is as in \eqref{c30} and $w_{h,j}$ is the $j$-th component of $w_h$.  The error terms in the second line of
\eqref{c54} come from ``cross-term" estimates like
\begin{align}\label{c56}
\Re\left((\eps^4\gamma\rho^4 b_{0,j}\chi_2(\eps\zeta))_Dw_{h,j},w_{h,j}\right)\geq    C\eps^4\gamma |w_{h,j}|_{\rho^2}^2)-C\eps^4\gamma |U_h|^2_{\rho},
\end{align}
Another application of the classical Garding inequality gives
\begin{align}\label{c57}
\Re\frac{1}{\eps}((S_p\mrp)_D w_p,w_p)\geq \frac{KC_p}{\eps}|w_p|^2_0-\frac{C}{\eps}|U_p|^2_{-1}.
\end{align}

Combining the above estimates for $w_h$ we obtain
\begin{align}\label{c58}
\begin{split}
&(\eps^2\gamma^3|w_h|^2_0+\eps^5|w_h|_{\rho^3}^2) +\langle S_{h,D}w_h,w_h\rangle \leq C_\delta\;\eps |f_h|_0^2+ C|U_h|_{\eps\gamma+\eps^2\rho^2}^2+\\
&\qquad C(\eps^2\gamma^3|U_h|^2_{-1}+\eps^5|U_h|^2_{\rho^2}+\eps^3\gamma^2|U_h|^2_0+\eps^4\gamma|U_h|^2_\rho),
\end{split}
\end{align}
after absorbing $w_h$ norms from the right using Lemma \ref{c16}.   Similarly, we find
\begin{align}\label{c59}
\frac{K}{\eps}|w_p|^2_0+\langle S_{p,D}w_p,w_p\rangle\leq C_\delta\eps|f_p|^2_0+\frac{CK}{\gamma\eps}|U_p|^2_0+\frac{C}{\eps}|U_p|^2_{-1},
\end{align}
after absorbing $w_p$ norms from the right.

\textbf{h. Boundary terms. }We clearly have
\begin{align}\label{d1}
\langle S_{h,D}w_h,w_h\rangle\geq -C_2(\eps^2\gamma^2\langle w_h\rangle^2_0+\eps^4\langle w_h\rangle_2^2),
\end{align}
and an application of the classical Garding inequality gives
\begin{align}\label{d2}
\langle S_{p,D}w_p,w_p\rangle\geq K\langle w_p\rangle_0^2-C\langle U_p\rangle^2_{-1}.
\end{align}

We use the classical calculus and the fact that $\cH_D\in\eps\cC^1$ to rewrite \eqref{c40}(c) as
\begin{align}\label{d2a}
V^{-1}_D\cH_DV_Dw_h+V^{-1}_DI^\eps_DW_Dw_p=r_0g_c+\eps r_0 U_c,
\end{align}
which implies with a new $\eps r_0U_c$
\begin{align}\label{d3}
\cB_D\begin{pmatrix}w_h\\w_p\end{pmatrix}:=\mrh_Dw_h+\cI^\eps_D w_p=r_0g_c+\eps r_0 U_c,\text{ where }\cI^\eps_D=V^{-1}_DI^\eps_DW_D.
\end{align}
Clearly,
\begin{align}\label{d4}
\langle\cB_D^*\cB_D \begin{pmatrix}w_h\\w_p\end{pmatrix},\begin{pmatrix}w_h\\w_p\end{pmatrix}\rangle\leq C\langle g_c\rangle^2_0+C\eps^2\langle U_c\rangle^2_0,
\end{align}
and we now proceed to estimate $\langle\cB_D^*\cB_D \begin{pmatrix}w_h\\w_p\end{pmatrix},\begin{pmatrix}w_h\\w_p\end{pmatrix}\rangle$ from below.

\begin{lem}\label{d4a}
\begin{align}\label{d4b}
\begin{split}
&\langle\cB_D^*\cB_D \begin{pmatrix}w_h\\w_p\end{pmatrix},\begin{pmatrix}w_h\\w_p\end{pmatrix}\rangle\geq \\
&\qquad C_1\eps^2\gamma^2\langle w_{h}\rangle^2_0+C_1\eps^4\langle w_{h}\rangle^2_2-C_3\langle w_p\rangle_0^2-C\left(\eps^2\gamma^2 \langle U_h\rangle ^2_{-1}+\eps^4 \langle U_h\rangle ^2_1+\eps^3\gamma \langle U_h\rangle ^2_0\right).
\end{split}
\end{align}

\end{lem}

\begin{proof}
We start from
\begin{align}\label{d5}
\begin{split}
\langle\cB_D^*\cB_D
&\begin{pmatrix}w_h\\w_p\end{pmatrix},\begin{pmatrix}w_h\\w_p\end{pmatrix}\rangle=\langle\mrh_D^*\mrh_Dw_h,w_h\rangle+2\Re\langle \mrh_Dw_h,\cI^\eps_Dw_p\rangle+|\cI^\eps_Dw_p|^2_0\geq \\
&\qquad\qquad \frac{1}{2}\langle\mrh_D^*\mrh_Dw_h,w_h\rangle-C_3\langle w_p\rangle_0^2
\end{split}
\end{align}
Using \eqref{c31} we see that
\begin{align}\label{d6}
\mrh=\mrh(p'(x),\hzeta,\eps\rho)=\mathrm{diag} (\mrh_j)=\mathrm{diag} (A_j+iD_j),
\end{align}
where $A_j=\eps\gamma a_{0,j}+\eps^2\rho^2 a_{1,j}$ and $D_j=\eps\rho d_j$.  Now
\begin{align}\label{d7}
h_{j,D}^*h_{j,D}=A_{j,D}^*A_{j,D}+D_{j,D}^*D_{j,D}+i(A_{j,D}^*D_{j,D}-D_{j,D}^*A_{j,D}),
\end{align}
where
\begin{align}\label{d8}
i(A_{j,D}^*D_{j,D}-D_{j,D}^*A_{j,D})=\eps^2\gamma r_0+\eps^3r_2, \;\;r_2\in\cC^2,
\end{align}
since, for example, the classical calculus implies
\begin{align}\label{d9}
(\eps^2\rho^2 a_{1,j})_D^*(\eps\rho d_j)_D-(\eps\rho d_j)_D^*(\eps^2\rho^2 a_{1,j})_D=\eps^3 r_2.
\end{align}
Next we compute
\begin{align}\label{d10}
A_{j,D}^*A_{j,D}=\left(\eps^2\gamma^2(a_{0,j}^2)_D+\eps^2\gamma^2r_{-1}\right)+\left(\eps^4(\rho^4 a_{1,j})_D+\eps^4 r_3\right)+\left(2\eps^3\gamma(\rho^2a_{0,j}a_{1,j})_D+\eps^3\gamma r_1\right).
\end{align}
The classical Garding inequality and \eqref{d10} imply
\begin{align}\label{d11}
\begin{split}
&\langle A_{j,D}^*A_{j,D}w_{h,j},w_{h,j}\rangle\geq C_1\eps^2\gamma^2\langle w_{h,j}\rangle^2_0+C_1\eps^4\langle w_{h,j}\rangle^2_2+C\eps^3\gamma \langle w_{h,j}\rangle^2_1\\
&\qquad -C\left(\eps^2\gamma^2 \langle U_h\rangle ^2_{-1}+\eps^4 \langle U_h\rangle ^2_1+\eps^3\gamma \langle U_h\rangle ^2_0\right)\\
&\qquad\qquad -C\left(\eps^2\gamma^2\langle w_{h,j}\rangle^2_{-1/2}+\eps^4\langle w_{h,j}\rangle^2_{3/2}+\eps^3\gamma \langle w_{h,j}\rangle^2_{1/2}\right),
\end{split}
\end{align}
where the error terms in the second line of \eqref{d11} are Garding errors, while those in the third line arise from the composition errors in \eqref{d10}.   From \eqref{d7}, \eqref{d8}, and \eqref{d11}  we obtain, after absorbing error terms involving $w_h$ and $w_p$ by taking $\gamma$ large:
\begin{align}\label{d12}
\langle\mrh_D^*\mrh_Dw_h,w_h\rangle\geq C_1\eps^2\gamma^2\langle w_{h}\rangle^2_0+C_1\eps^4\langle w_{h}\rangle^2_2-C\left(\eps^2\gamma^2 \langle U_h\rangle ^2_{-1}+\eps^4 \langle U_h\rangle ^2_1+\eps^3\gamma \langle U_h\rangle ^2_0\right).
\end{align}
Using \eqref{d5}, and \eqref{d12} we obtain the estimate of the Lemma with new constants.
\end{proof}

Combining the estimates of this paragraph we find, for constants as in \eqref{d1}, \eqref{d2}, Lemma \ref{d4a} and  some $M>0$ to be chosen:
\begin{align}\label{d13}
\begin{split}
&\langle S_{h,D}w_h,w_h\rangle+\langle S_{p,D}w_p,w_p\rangle+M\langle\cB_D^*\cB_D
\begin{pmatrix}w_h\\w_p\end{pmatrix},\begin{pmatrix}w_h\\w_p\end{pmatrix}\rangle\geq \\
&\qquad (MC_1-C_2)\left(\eps^2\gamma^2\langle w_{h}\rangle^2_0+\eps^4\langle w_{h}\rangle^2_2\right)+(K-MC_3)\langle w_p\rangle_0^2\\
&\qquad\qquad -MC\left(\eps^2\gamma^2 \langle U_h\rangle ^2_{-1}+\eps^4 \langle U_h\rangle ^2_1+\eps^3\gamma \langle U_h\rangle ^2_0\right)-C\langle U_p\rangle^2_{-1}.
\end{split}
\end{align}

\textbf{i. Conclusion. }To finish the proof of Proposition \ref{c2a}  we first add estimates \eqref{c58} and \eqref{c59}, and then add $M\langle\cB_D^*\cB_D
(w_h,w_p)^t,(w_h,w_p)^t\rangle$ to both sides of the resulting inequality.  Boundary terms on the left in the estimate so obtained are estimated from below using \eqref{d13}; on the right one uses \eqref{d4}.  After choosing $M$ so that $MC_1>C_2$ and then $K$ such that $K>MC_3$, we get (with a new $C$)

\begin{align}\label{e1}
\begin{split}
&\left(\eps^2\gamma^3|w_h|^2_0+\eps^5|w_h|_{\rho^3}^2 +\frac{K}{\eps}|w_p|^2_0\right)+\left(\eps^2\gamma^2\langle w_{h}\rangle^2_0+\eps^4\langle w_{h}\rangle^2_2+\langle w_p\rangle_0^2\right)\leq\\
&\qquad C\left(\eps |f_h|_0^2+ \eps|f_p|^2_0+\langle g_c\rangle^2_0\right)+\\
&C\left(|U_h|_{\eps\gamma+\eps^2\rho^2}^2+ \eps^2\gamma^3|U_h|^2_{-1}+\eps^5|U_h|^2_{\rho^2}+\eps^3\gamma^2|U_h|^2_0+\eps^4\gamma|U_h|^2_\rho+\frac{K}{\gamma\eps}|U_p|^2_0+\frac{1}{\eps}|U_p|^2_{-1}\right)+\\
&\qquad C\left(\eps^2\gamma^2 \langle U_h\rangle ^2_{-1}+\eps^4 \langle U_h\rangle ^2_1+\eps^3\gamma \langle U_h\rangle ^2_0+\langle U_p\rangle^2_{-1}+\eps^2\langle U_c\rangle^2_0\right),
\end{split}
\end{align}
where the last two lines are ``error" terms.
Since
\begin{align}\label{e2}
u_h=V_Dw_h+r_{-1}u_h \text{ and }u_p=W_D w_p+\eps r_0 u_p
\end{align}
for $V_D$, $W_D$ as in \eqref{c41}, \eqref{c45}, the estimate \eqref{e1} holds with $(w_h,w_p)^t$ replaced by $(u_h,u_p)^t=U_l$.   Recalling that
\begin{align}\label{e2a}
U_c=(U_h,U_p)^t=\sum_l U_l  \text{ and }\phi^\eps_{l,D}F_c=(f_h,f_p)^t,
\end{align}
summing the estimates over $l$, and absorbing error terms from the right using Lemma \ref{c16}, we conclude
\begin{align}\label{e3}
\begin{split}
&\left(\eps^2\gamma^3|U_h|^2_0+\eps^5|U_h|_{\rho^3}^2 +\frac{K}{\eps}|U_p|^2_0\right)+\left(\eps^2\gamma^2\langle U_{h}\rangle^2_0+\eps^4\langle U_{h}\rangle^2_2+\langle U_p\rangle_0^2\right)\leq\\
&\qquad \qquad\qquad \qquad C\left(\eps |F_h|_0^2+ \eps|F_p|^2_0+\langle g_c\rangle^2_0\right)
\end{split}
\end{align}
This estimate is stronger than the estimate described at the end of paragraph \textbf{b} as being sufficient to prove \eqref{c2}(a). This concludes the proof of Proposition \ref{c2a}.

\end{proof}

\subsection{Higher derivative estimates}\label{higher}

\emph{\quad} In this section we'll use the notation for norms introduced in section \ref{basic}.  We use $\partial$ to denote some tangential derivative, one of $\partial_0,\dots,\partial_{d-1}.$ Sometimes $\partial U$ will denote the tangential gradient of $U$, instead of just a single partial derivative of $U$.

\begin{notation}\label{e9.54}
1.  For $k=1,2,\dots$ let $U^{*,k}=((\frac{\gamma}{\epsilon^2})^k U,(\frac{\gamma}{\epsilon^2})^{k-1}\partial U,\dots,\partial^k U)$.  Here $\partial^j U$ represents all possible tangential derivatives of $U$ order $j$.

2.  Define $U^{*,k}_\Lambda$ simply by replacing $U$ by $U_\Lambda$  in the definition of $U^{*,k}$.
\end{notation}

\begin{prop}\label{e9.55}
Under the assumptions of section 2, there exist positive constants $C$, $\epsilon_0$, $\gamma_0$ such that for all $\gamma>\gamma_0$, $0<\epsilon<\epsilon_0$ with $\epsilon\gamma\leq 1$,  solutions to \eqref{c0b} satisfy
\begin{align}\label{ej9.56}
|U^{*,k}|_0+\langle U^{*,k}\rangle_0\leq C\left(\frac{|F^{*,k}|_0}{\sqrt{\epsilon}}+\frac{\langle g^{*,k}\rangle_0}{\eps}\right).
\end{align}
\end{prop}

This follows immediately from the following more precise estimates.

\begin{prop}\label{e10.1}
Using the notation just introduced, we have the following estimates for solutions to \eqref{c0b}.  Let $R_1$, $R_2$ be as in \eqref{c0f}.  For $R_1$ sufficiently small and $R_2$ sufficiently large, there exist constants $C$, $\gamma_1$, $\epsilon_1$ such that for all $\gamma>\gamma_1$, $0<\epsilon<\epsilon_1$ with $\epsilon\gamma\leq 1$
\begin{align}\label{e10.2}
\begin{split}
&(a)\;|\chi_{S,D} U^{*,k}|_{\eps\gamma^{\frac{3}{2}}+\eps^{\frac{5}{2}}\rho^3} +\langle\chi_{S,D}U^{*,k}\rangle_{\eps\gamma+\eps^2\rho^2}\leq\\
& \qquad C\left(\sqrt{\eps}|F^{*,k}|_0+\langle g^{*,k}\rangle_0+\epsilon|U^{*,k}|_0+|\chi_{2,D}U^{*,k}|_{\eps^{\frac{3}{2}}\rho+\eps\gamma+\eps^2\rho^2}+|\chi_{M,D}U^{*,k}|_0+\eps\langle U^{*,k} \rangle_0\right)\\
&(b)\;|\chi_{M,D} U^{*,k}|_0+\sqrt{\epsilon}\langle \chi_{M,D} U^{*,k}\rangle_0\leq C\left(\epsilon|F^{*,k}|_0+\sqrt{\eps}\langle g^{*,k}\rangle_0+\epsilon|U^{*,k}|_0+\epsilon\langle U^{*,k}\rangle_0\right)\\
&(c)\;|\chi_{L,D} U^{*,k}_\Lambda|_{\sqrt{\Lambda}}+\sqrt{\epsilon}\langle\chi_{L,D} U^{*,k}_\Lambda\rangle_0\leq C\left(\epsilon|F^{*,k}|_{\Lambda^{-1/2}}+\sqrt{\eps}\langle g^{*,k}\rangle_0+\epsilon|U^{*,k}_\Lambda|_{\Lambda^{-1/2}}+\epsilon\langle U^{*,k}_\Lambda\rangle_{\Lambda^{-1/2}}\right).
\end{split}
\end{align}
\end{prop}

\begin{proof}
The estimates in (b) and (c) follow directly from the higher derivative estimates of \cite{MZ1} in the medium and large freqency regions.  These are estimates with $\gamma$ weights for the linearized problem, so one can simply apply them to the problems satisfied by $\frac{U}{(\epsilon^2)^j}$ for various $j$.

As usual, therefore, we focus on the small frequency region.  If we simply differentiate the equation and throw commutators on the right as forcing, those new forcing terms are too large to absorb in a straightforward way.  To get around this problem we reprove $L^2$ estimates for an appropriate enlarged system.

\textbf{1. Enlarging the system.}  We begin with a solution $U$ of the linear system \eqref{c0b}
\begin{align}\label{e10.3}
\begin{split}
&\partial_d U-\frac{1}{\epsilon}{G}U=F\\
&\Gamma U=g \text{ on }x_d=0\\
&U=0 \text{ in } x_0<0:
\end{split}
\end{align}

Let $\partial$ denote one of $\partial_0,\dots,\partial_{d-1}$.  Observe that $(\frac{\gamma}{\epsilon^2}U,\partial U)$ satisfies the enlarged system
\begin{align}\label{e10.4}
\begin{split}
&\partial_d\begin{pmatrix}\frac{\gamma}{\epsilon^2}U\\\partial U\end{pmatrix}-\frac{1}{\epsilon}\begin{pmatrix}{G}&0\\0&{G}\end{pmatrix}\begin{pmatrix}\frac{\gamma}{\epsilon^2}U\\\partial U\end{pmatrix}=\begin{pmatrix}\frac{\gamma}{\epsilon^2}F\\\partial F\end{pmatrix}+\begin{pmatrix}0\\\frac{\epsilon}{\gamma}[\partial,{G}]\begin{pmatrix}\frac{\gamma}{\epsilon^2}U\end{pmatrix}\end{pmatrix},\\
&\begin{pmatrix}\Gamma&0\\0&\Gamma\end{pmatrix}\begin{pmatrix}\frac{\gamma}{\epsilon^2}U\\\partial U\end{pmatrix}=\begin{pmatrix}\frac{\gamma}{\epsilon^2}g\\\partial g\end{pmatrix}\text{ on }x_d=0,\\
&\begin{pmatrix}\frac{\gamma}{\epsilon^2}U\\\partial U\end{pmatrix}=0 \text{ in }x_0<0.
\end{split}
\end{align}

\textbf{2. Localize to small frequency region.}
Let $\chi_S(\epsilon\zeta)$ be a small frequency cutoff as before.  Commuting $\chi_{S,D}$ through \eqref{e10.4} we obtain (writing $\chi_S$ for $\chi_{S,D}$)
\begin{align}\label{e10.5}
\begin{split}
&\partial_d (\chi_S U^{*,1})-\frac{1}{\epsilon}\begin{pmatrix}{G}&0\\0&{G}\end{pmatrix}(\chi_S U^{*,1})=\\
&\qquad\qquad\chi_S F^{*,1}+\chi_S\begin{pmatrix}0\\\frac{\epsilon}{\gamma}[\partial,{G}](\frac{\gamma}{\epsilon^2}U)\end{pmatrix}+\frac{1}{\epsilon}\left[\chi_S,\begin{pmatrix}{G}&0\\0&{G}\end{pmatrix}\right]U^{*,1}=F',
\end{split}
\end{align}
where
\begin{align}\label{e10.6}
|F'|_0\leq C(|F^{*,1}|_0+|(\partial_{\beta'}\chi_S)U^{*,1}|_0+\epsilon|U^{*,1}|_0).
\end{align}
The second commutator was computed like the corresponding term in
the previous section \eqref{c5}.

The boundary condition is
\begin{align}\label{e10.7}
\begin{pmatrix}\Gamma&0\\0&\Gamma\end{pmatrix}\chi_S U^{*,1}=\chi_S g^{1,*}.
\end{align}

The problem \eqref{e10.5},\eqref{e10.7} can be treated just like \eqref{c6}.
We may now repeat the argument of the previous section to obtain the desired estimate of $U^{*,1}$.  Iteration completes the proof.

\end{proof}

\begin{rem}\label{e10.8}
If $U^{*,1}$ had been defined instead as $\begin{pmatrix}\frac{\gamma}{\epsilon}U\\\partial U\end{pmatrix}$, the first commutator in \eqref{e10.5} would have produced an unacceptable $O(|U^{*,1}|_0)$ error.
\end{rem}

\subsection{Nonlinear stability} \label{nonlinear}

\begin{notation}\label{e11.1}

1. Recall $|u|_{k,\gamma}=|\langle\zeta\rangle^k\hat{u}(\zeta,x_d)|_0$.  For $k\in\mathbb{N}$ we have the equivalence of norms
\begin{align}\label{e11.2}
|u|_{k,\gamma}\sim\sum_{|\alpha|\leq k}\gamma^{k-|\alpha|}|\partial^\alpha u|_0.
\end{align}

2. Set $|u|_*=|u|_{L^\infty}$.

3. Define
\begin{align}\label{e11.3}
\|u\|_{k,\gamma}=|u|_{k,\gamma}+|\epsilon\partial u|_{k,\gamma}.
\end{align}

4. Let $M$ and $L<M$ be the positive integers appearing in the nonlinear error equation \eqref{ae1}.  They can be taken arbitrarily large as long as the approximate solution $u^a$ is constructed with sufficiently many terms.

5. $\phi(\gamma)$ always denotes an increasing function of $\gamma$.  It may change from term to term.

6. Set $\partial''=(\partial_1,\dots,\partial_{d-1})$.
\end{notation}

We return to the nonlinear error equation \eqref{aj}, and again drop tildes and the superscript $\gamma$.
Let $\kappa(x_0)$ be a smooth cutoff which is identically one on $[0,T_0]$.
We will solve \eqref{aj} on $[0,T_0]$ using the following iteration scheme:
\begin{align}\label{e11.5}
\begin{split}
&\partial_d U_{n+1}-\frac{1}{\epsilon}{G}U_{n+1}=\kappa(x_0){F}_\epsilon(U_n,\partial'' U_n),\\
&\Gamma U_{n+1}=0\text{ on }x_d=0,\\
&U_{n+1}=0\text{ in }x_0<0,
\end{split}
\end{align}
where from \eqref{ae1} and \eqref{ah} we see that $F_\epsilon(U_n,\partial'' U_n)$ has the form
\begin{align}\label{e11.6}
\begin{split}
&{F}_\epsilon(U_n,\partial'' U_n)=\epsilon^{L-3}f_1(u^a,\nabla u^a,\epsilon^L U_n,\epsilon,e^{\gamma x_0})(\epsilon U_n,\epsilon U_n)\\
&\qquad+\epsilon^{L-3}f_2(u^a,\nabla u^a,\eps^L U_n,\epsilon,e^{\gamma x_0})(\epsilon U_n,\epsilon \partial'' U_n)\\
&\qquad\qquad +\epsilon^{M-L}R_\eps\\
&\qquad\qquad\qquad :=\mathcal{A}+\mathcal{B}+\mathcal{C},
\end{split}
\end{align}
for smooth functions $f_1$, $f_2$.
For $\mathbb{F}(U,\partial''U):=\kappa(x_0){F}_\epsilon(U,\partial'' U)$ consider the nonlinear error equation
\begin{align}\label{e11.7}
\begin{split}
&\partial_d U-\frac{1}{\epsilon}{G}U=\mathbb{F}(U,\partial'' U),\\
&\Gamma U=0\text{ on }x_d=0,\\
&U=0\text{ in }x_0<0.
\end{split}
\end{align}

\begin{thm}\label{e11.8}
Recall $d$ is the number of space dimensions.  Fix constants $k,L,M$ satisfying
\begin{align}\label{e11.9}
\begin{split}
&k-3>\frac{d}{2}\\
&M-L-2k-\frac{1}{2}>1\\
&L-3-2k-\frac{1}{2}>1.
\end{split}
\end{align}
Suppose the forcing term $f$ in \eqref{aa} is chosen in $H^s(\overline{\bR}^{d+1}_+)$, where $s\geq 3M+3+k$, so that
 $u^a$ as constructed in Proposition \ref{g4} yields a remainder $R_\eps \in H^k(\Omega_{T_0})$.\footnote{This is the same $R_\eps$ that appears in \eqref{e11.6}.}
Then there exist constants $\epsilon_0$, $\gamma_0$ such that for all $0<\epsilon\leq\epsilon_0$, $\gamma\geq\gamma_0$ satisfying $\epsilon\gamma\leq 1$, the error equation \eqref{e11.7} has a unique solution $U$ satisfying the estimates
\begin{align}\label{e11.10}
\begin{split}
&\|U\|_{k,\gamma}\leq \epsilon^{M-L-2k-\frac{1}{2}}\phi(\gamma)\\
&|U|_*\leq 1\\
&|\partial U|_*\leq 1
\end{split}
\end{align}
for some $\phi(\gamma)$, an increasing function of $\gamma$.
\end{thm}

\begin{proof}
The first few points are some preliminaries.

\textbf{1. Sobolev inequalities. }

For $k-3>\frac{d}{2}$ we have
\begin{align}\label{e11.11}
\begin{split}
&(a)\epsilon|\partial U|_*\leq C(\gamma)(\epsilon|U|_{k-2,\gamma}+\epsilon|\partial_d U|_{k-2,\gamma})\\
&(b)\epsilon|U|_*\leq C(\gamma)(\epsilon|U|_{k-3,\gamma}+\epsilon|\partial_d U|_{k-3,\gamma}).
\end{split}
\end{align}

\textbf{2. Moser inequalities. }

For $k\in\mathbb{N}$ let $\alpha=(\alpha_1,\dots,\alpha_r)$ with $|\alpha|=\alpha_1+\cdots +\alpha_r\leq k$, $\alpha_i\in\mathbb{N}$. Suppose $|v_i|_{k,\gamma}+|v_i|_*<\infty$.  Then
\begin{align}\notag
\;\gamma^{k-|\alpha|}|(\partial^{\alpha_1}v_1)\cdots (\partial^{\alpha_r}v_r)|_0\leq C\sum^r_{i=1}|v_i|_{k,\gamma}(\prod_{j\neq i}|v_i|_*)
\end{align}

\textbf{3. Relations between norms. }Directly from the definitions we see
\begin{align}\label{e11.12}
\begin{split}
&(a)|U|_{k,\gamma}\leq C|U^{*,k}|_0\\
&(b)|U^{*,k}|_0\leq \frac{C}{\epsilon^{2k}}|U|_{k,\gamma}.
\end{split}
\end{align}

Let $\chi_L(\epsilon\zeta)$ be a high frequency cutoff like the one in \eqref{e10.2}(c).  Observe that
\begin{align}\label{e11.12a}
\|U\|_{k,\gamma}\sim |U|_{k,\gamma}+|\chi_L(\epsilon\partial U)|_{k,\gamma}.
\end{align}

\textbf{4. High frequency estimate. }Here we make use of a slightly modified form of the high frequency estimate in \eqref{e10.2}(c) with $g=0$:
\begin{align}\label{e11.12aa}
|\chi_L U^{*,k}_\Lambda|_\Lambda+\sqrt{\epsilon}\langle\chi_L U^{*,k}_\Lambda\rangle_{\sqrt{\Lambda}}\leq
 C\left(\epsilon|F^{*,k}|_0+\epsilon|U^{*,k}_\Lambda|_0+\epsilon\langle U^{*,k}_\Lambda\rangle_0\right).
\end{align}
We can absorb the high frequency pieces of $U^{*,k}_\Lambda$ in the two terms on the right in \eqref{e10.1}(c) to obtain
\begin{align}\label{e11.13}
|\chi_L U^{*,k}_\Lambda|_\Lambda\leq
C\left(\epsilon|F^{*,k}|_0+\epsilon|U^{*,k}|_0+\epsilon\langle U^{*,k}\rangle_0\right),
\end{align}
and then use the main $L^2$ estimate \eqref{ej9.56} to replace the right side of the above inequality by $C\sqrt{\eps}|F^{*,k}|_0$.
When $|\epsilon\zeta|$ is large, we have $\frac{\Lambda^2}{\epsilon}\geq C\langle\zeta\rangle$.
Thus, with \eqref{e11.12}(a) we may conclude
\begin{align}\label{e11.14}
|\chi_L (\epsilon\partial U)|_{k,\gamma}\leq C\sqrt{\eps}|F^{*,k}|_0.
\end{align}

\textbf{5. Induction assumption. }Let the first iterate $U_1$ be $0$. Assume there exist $\epsilon_1(\gamma)$, $\gamma_1$ such that for $0<\epsilon\leq\epsilon_1$, $\gamma\geq\gamma_1$, and some $\phi(\gamma)$
\begin{align}\label{e11.15}
\begin{split}
&\|U_n\|_{k,\gamma}\leq 2\epsilon^{M-L-2k-\frac{1}{2}}\phi(\gamma)\\
&|U_n|_*\leq 1\\
&|\partial U_n|_*\leq 1
\end{split}
\end{align}
The main step is to show, after decreasing $\epsilon_1$ if necessary, that $U_{n+1}$ satisfies the same estimates.

\textbf{6. Estimate $\mathbb{F}_n:=\mathbb{F}(U_n,\partial''U_n)$. }
Set $\mathbb{A}=\kappa(x_0)\mathcal{A}$ for $\cA$ as in \eqref{e11.6}, and define $\mathbb{B}$ and $\mathbb{C}$ similarly.

Applying the Moser inequalities we have
\begin{align}\label{e11.16}
|\mathbb{A}|_{k,\gamma}\leq C(\gamma)\epsilon^{L-2}|U_n|_{k,\gamma},
\end{align}
where $C(\gamma)$ depends on $L^\infty$ norms of $(u^a,\nabla u^a)$ and $\epsilon U_n$.

Write $\epsilon\partial U_n=(1-\chi_L)(\epsilon\partial U_n)+\chi_L(\epsilon\partial U_n)$, and corresponding to this decomposition set $\mathbb{B}=\mathbb{B}_1+\mathbb{B}_2$.  Since $|\epsilon\zeta|\leq C$ on supp $(1-\chi_L(\epsilon\zeta))$, we have just as above
\begin{align}\label{e11.17}
|\mathbb{B}_1|_{k,\gamma}\leq C(\gamma)\epsilon^{L-2}|U_n|_{k,\gamma}.
\end{align}
For $\mathbb{B}_2$ we have
\begin{align}\label{e11.18}
|\mathbb{B}_2|_{k,\gamma}\leq C(\gamma)(\epsilon^{L-2}|U_n|_{k,\gamma}+\epsilon^{L-3}|\chi_L(\epsilon\partial U_n)|_{k,\gamma}).
\end{align}
Moreover, we have
\begin{align}\label{e11.19}
|\mathbb{C}|_{k,\gamma}\leq \phi(\gamma)\epsilon^{M-L}.
\end{align}
Summing these estimates we obtain
\begin{align}\label{e11.20}
|\mathbb{F}_n|_{k,\gamma}\leq C(\gamma)(\epsilon^{L-2}|U_n|_{k,\gamma}+\epsilon^{L-3}|\chi_L(\epsilon\partial U_n)|_{k,\gamma})+\epsilon^{M-L}\phi(\gamma).
\end{align}

\textbf{7. Estimate $\|U_{n+1}\|_{k,\gamma}$. }
In view of the main estimate \eqref{ej9.56}, \eqref{e11.12}, and \eqref{e11.20} we have
\begin{align}\label{e11.21}
\begin{split}
&|U_{n+1}|_{k,\gamma}\leq C|U^{*,k}_{n+1}|_0\leq\frac{C}{\sqrt{\epsilon}}|\mathbb{F}^{*,k}_n|_0\leq\frac{C}{\epsilon^{2k+\frac{1}{2}}}|\mathbb{F}_n|_{k,\gamma}\\
&\leq C(\gamma)(\epsilon^{L-2-2k-\frac{1}{2}}|U_n|_{k,\gamma}+\epsilon^{L-3-2k-\frac{1}{2}}|\chi_L(\epsilon\partial U_n)|_{k,\gamma})+\epsilon^{M-L-2k-\frac{1}{2}}\phi(\gamma).
\end{split}
\end{align}

From \eqref{e11.14} and \eqref{e11.20} we obtain
\begin{align}\label{e11.22}
\begin{split}
&|\chi_L(\epsilon\partial U_{n+1})|_{k,\gamma}\leq C|\mathbb{F}^{*,k}_n|_0\leq\frac{C}{\epsilon^{2k}}|\mathbb{F}_n|_{k,\gamma}\\
&\leq C(\gamma)(\epsilon^{L-2-2k}|U_n|_{k,\gamma}+\epsilon^{L-3-2k}|\chi_L(\epsilon\partial U_n)|_{k,\gamma})+\epsilon^{M-L-2k}\phi(\gamma).
\end{split}
\end{align}

Adding the previous two estimates we find
\begin{align}\label{e11.23}
\|U_{n+1}\|_{k,\gamma}\leq \epsilon^{L-3-2k-\frac{1}{2}}C(\gamma)\|U_n\|_{k,\gamma}+\epsilon^{M-L-2k-\frac{1}{2}}\phi(\gamma).
\end{align}

Provided $\epsilon_1(\gamma)$ is chosen so that $\epsilon^{L-3-2k-\frac{1}{2}}C(\gamma)\leq \frac{1}{2}$, the induction assumption and \eqref{e11.23} imply
\begin{align}\label{e11.24}
\|U_{n+1}\|_{k,\gamma}\leq 2\epsilon^{M-L-2k-\frac{1}{2}}\phi(\gamma).
\end{align}

\textbf{8. $L^\infty$ estimates. }The equation gives
\begin{align}\label{e11.25}
\epsilon |\partial_d U_{n+1}|_{k-2,\gamma}\leq C |U_{n+1}|_{k,\gamma}+\epsilon|\mathbb{F}_n|_{k-2,\gamma}.
\end{align}
From \eqref{e11.20} we get
\begin{align}\label{e11.26}
|\mathbb{F}_n|_{k,\gamma}\leq \epsilon^{L-3}C(\gamma)\|U_n\|_{k,\gamma}+\epsilon^{M-L}\phi(\gamma).
\end{align}

Thus,
\begin{align}\label{e11.27}
\epsilon |\partial_d U_{n+1}|_{k-2,\gamma}\leq 2\epsilon^{M-L-2k-\frac{1}{2}}\phi(\gamma).
\end{align}
This together with the inequalities \eqref{e11.11} and the assumption \eqref{e11.9} immediately implies that for $\epsilon_1$ small enough
\begin{align}\label{e11.28}
\begin{split}
&\epsilon|U_{n+1}|_*\leq \epsilon\\
&\epsilon|\partial U_{n+1}|_*\leq \epsilon.
\end{split}
\end{align}

This completes the inductive step.

\textbf{9. Contraction.}
 Thus, the sequence of iterates satisfies  the estimates \eqref{e11.15}. One can now consider the problem satisfied by $U_{n+1}-U_n$ and use estimates like those above (but simpler) to show that for $\epsilon_1$ small enough, the sequence converges to some $U$ in the $\|\;\|_{0,\gamma}$ norm.  A standard argument (involving interpolation and weak convergence) implies that $U$ solves the error equation \eqref{e11.7} and satisfies the estimates \eqref{e11.10} in Theorem \ref{e11.8}.

This completes the proof of Theorem \ref{e11.8}, and the paper.
\end{proof}

\medskip

{\bf Acknowledgement:}  The third and fourth authors thank
Indiana University, Bloomington, University of North Carolina,
Chapel Hill,
University of Provence, University of Bordeaux 1,
E.N.S., Paris, University of Paris 13, and the
Foundation Sciences Math\'ematiques de Paris for their hospitality
in supporting a series of visits in which this work was
initiated and partially carried out.
\appendix

\section{Appendix: Classical and semiclassical pseudodifferential calculi}

\subsection{Semiclassical calculus}

\emph{\quad}Our proof of the $L^2$ estimate in the small frequency region requires the use of classical and  semiclassical pseudodifferential operators with finite regularity in $x'$.  Here we summarize the needed properties of those calculi, referring the reader to the Appendix of \cite{GMWZ2} for all the proofs.  We are not able to use paradifferential operators (which might have allowed us to assume much less regularity in $x'$), because the process of paralinearization introduces $O(|U|_{L^2})$ errors at a stage when they are too big to be absorbed by the left side of our degenerate $L^2$ estimate.

\begin{notation}\label{ee8.4}
1.  Let $\zeta'=(\zeta_0,\zeta'')\in \mathbb{R}^d$ denote variables dual to the tangential variables $x'=(x_0,x'')$, and set $\zeta=(\zeta',\gamma)$, where we always take $\gamma\geq 1$.  Set $\langle\zeta\rangle=\sqrt{|\zeta|^2}=\sqrt{|\zeta',\gamma|^2}$ and, with slight abuse, $\langle\zeta'\rangle=\sqrt{|\zeta',1|^2}$.

2. For $\epsilon>0$ let $\beta=(\beta',\gamma')\in\mathbb{R}^d\times\overline{\mathbb{R}}_+$ (resp. $\beta'\in\mathbb{R}^d$) denote a placeholder for $\epsilon\zeta$ (resp. $\epsilon\zeta'$).

3. We will ignore powers of $2\pi$ in all formulas involving pseudodifferential operators and Fourier transforms.

4. On $H^s(\mathbb{R}^d)$ define the norms $|u|_{s,\gamma}=|\langle\zeta\rangle^s\hat{u}|_{L^2}$.

5. The notation
\[
T_{\epsilon,\gamma}:\mathcal{X}\to\mathcal{Y}
\]
for a family of linear operators mapping one function space into another means that the operator norm is uniformly bounded with respect to $\epsilon,\gamma$ for $0<\epsilon<1$ and $\gamma\geq 1$. For a particular $s\in\mathbb{R}$ we say $T_{\epsilon,\gamma}$ is \emph{of order $k$ on $H^s$} if
\begin{align}\label{ee8.4a}
T_{\epsilon,\gamma}:H^s(\mathbb{R}^d)\to H^{s-k}(\mathbb{R}^d).
\end{align}
 When the domain and target spaces of $T$ are clear from the context, we'll write simply $|T|$ for the operator norm.

6. We will sometimes denote spaces like $C^M(\mathbb{R}^d_{x'},C^\infty(\mathbb{R}^d\times\overline{\mathbb{R}}_+))$ by $C^M(x',C^\infty(\beta))$ when the domains of the variables involved are clear.

\end{notation}

\begin{rem}\label{ee8.4z}
\textup{Our pseudodifferential operators are defined by symbols with finite regularity in $x'$.  Such an operator is generally of order $k$ on $H^s$ only for $s$ in a proper subinterval of $\mathbb{R}$.}
\end{rem}

The semiclassical operators are built from ``symbols" in the set

\begin{align}\label{ee8.5}
\begin{split}
&\mathcal{S}_M=\{p(x',\beta)\in
C^M(\mathbb{R}^d_{x'},C^\infty(\mathbb{R}^d\times\overline{\mathbb{R}}_+)): \\
&\qquad p\text{ is independent of }x' \text{ for }|x'| \text{ large and } \sup_{|\mu|\leq M}|\partial_{x'}^\mu\partial_{\beta'}^\nu p(x',\beta)|\leq C_\nu\}.
\end{split}
\end{align}
Let $\mathcal{S}_\infty=\cap_M \mathcal{S}_M$.   Define symbol norms
\begin{align}\label{ee8.5z}
|p|_{M,K}=\sup_{|\mu|\leq M}\sup_{|\nu|\leq K}\sup_{(x',\beta)}|\partial_{x'}^\mu\partial_{\beta'}^\nu p(x',\beta)|.
\end{align}
To each $p(x',\beta)\in \mathcal{S}_M$ we associate the operator defined by
\begin{align}\label{ee8.6}
p(x',\epsilon D)u = \int e^{ix'\zeta'}p(x',\epsilon\zeta)\hat{u}(\zeta')d\zeta'.
\end{align}

\begin{prop}\label{ee8.7}
If $p\in \mathcal{S}_M$ and $M\geq d+1$ then
\[
p(x',\epsilon D):L^2(\mathbb{R}^d)\to L^2(\mathbb{R}^d).
\]
\end{prop}

\begin{defn}\label{ee8.8}
A family of linear operators $r_{\epsilon,\gamma}$ is said to be \emph{of order $\epsilon^k$} if
$r_{\epsilon,\gamma}=\epsilon^k \mathcal{R}_{\epsilon,\gamma}$ where
\[
\mathcal{R}_{\epsilon,\gamma}:L^2(\mathbb{R}^d)\to L^2(\mathbb{R}^d).
\]
\end{defn}

\begin{prop}[Products]\label{ee8.9}
Suppose $p\in \mathcal{S}_{M_1}$ and $q\in \mathcal{S}_{M_2}$, where $M_1\geq d+1$ and $M_2\geq M_1+(d+1)+k+1$ for some $k\geq 1$. Set
\begin{align}\label{ee8.10}
t(x',\beta)=\sum_{|\alpha|\leq k-1}\frac{1}{\alpha!}\epsilon^{|\alpha|}\partial_{\beta'}^\alpha p(x',\beta)D_{x'}^\alpha q(x',\beta).
\end{align}
Then $t(x',\beta)\in \mathcal{S}_{M_1}$ and
\begin{align}\label{ee8.11}
A\equiv p(x',\epsilon D)q(x',\epsilon D)=t(x',\epsilon D)+r_{\epsilon,\gamma},
\end{align}
where $r_{\epsilon,\gamma}$ is  of order $\epsilon^k$. Precisely, $r_{\epsilon,\gamma}=\epsilon^k T$, where
\[
|T|\leq C|p|_{d+1,k}|\partial_{x'}q|_{M_2-1,0}.
\]
\end{prop}

\begin{prop}[Adjoints]\label{ee8.13a}
Suppose $p\in \mathcal{S}_M$, where $M\geq (d+1)+k+1$, for some $k\geq 1$.  Set
\begin{align}\label{ee8.14a}
t(x',\beta)=\sum_{|\alpha|\leq k-1}\frac{1}{\alpha!}\epsilon^{|\alpha|}\partial_{\beta'}^\alpha D^\alpha_{x'}p^*(x',\beta).
\end{align}
Then $t\in \mathcal{S}_{M-k+1}$ and
\[
p(x',\epsilon D)^*=t(x',\epsilon D)+r_{\epsilon,\gamma},
\]
where $r_{\epsilon,\gamma}$ is  of order $\epsilon^k$.  We have $r_{\epsilon,\gamma}=\epsilon^k T$, where
\[
|T|\leq C|\partial_{x'} p|_{M-1,k}.
\]
\end{prop}

\subsection{Classical calculus}\label{classical}

  For $m\in\mathbb{R}$ define the classical  symbol classes
\begin{align}\label{ee8.23}
\begin{split}
&\mathcal{C}^m_M=\{p(x',\zeta)\in
C^M(\mathbb{R}^d_{x'},C^\infty(\mathbb{R}^d\times\{\gamma\geq 1\}):
p\text{ is independent of }x' \\
&\quad\text{ for }|x'| \text{ large and } \sup_{|\mu|\leq M}|\partial_{x'}^\mu\partial_{\zeta'}^\nu p(x',\zeta)|\leq C_\nu\langle\zeta\rangle^{m-|\nu|}\},
\end{split}
\end{align}
and set $\mathcal{C}^m_\infty=\cap_M \mathcal{C}^m_M$.   Define associated symbol norms
\begin{align}\label{ee8.24z}
|p|_{M,K}=\sup_{|\mu|\leq M}\sup_{|\nu|\leq K}\sup_{(x',\zeta)}|\partial_{x'}^\mu\partial_{\zeta'}^\nu p(x',\zeta)|\langle\zeta\rangle^{|\nu|-m}.
\end{align}

To an element $p(x',\zeta)\in\mathcal{C}^m_M$ we associate the classical operator
\begin{align}\label{ee8.26}
p(x',D)u = \int e^{ix'\zeta'}p(x',\zeta)\hat{u}(\zeta')d \zeta'.
\end{align}

\begin{prop}[Classical products]\label{ee8.54}
Suppose
\[
p(x',\zeta)\in\mathcal{C}^{m_1}_{M_1} \text{ and } q(x',\zeta)\in\mathcal{C}^{m_2}_{M_2},
\]
where $M_1\geq d+1$ and $M_2\geq 2(d+1)+|m_1|+3$.   Set $t(x',\zeta)=p(x',\zeta)q(x',\zeta)$.  Then $t\in\mathcal{C}^{m_1+m_2}_{M_1}$ and
\begin{align}\label{ee8.55}
p(x',D)q(x',D)=t(x',D)+r,
\end{align}
where $r$ is of order $m_1+m_2-1$.  We have
\[
|r|\leq C|p|_{d+1,1}|\partial_{x'}q|_{M_2-1,0}.
\]
\end{prop}

\begin{prop}[Classical adjoints]\label{ee8.60}
Suppose
\[
p(x',\zeta)\in\mathcal{C}^m_M,
\]
where $M\geq (d+1)+|m|+3$. Set $t(x',\zeta)=p^*(x',\zeta)$.  Then
\begin{align}\label{ee8.61}
p(x',D)^*=t(x',D)+r_{\epsilon,\gamma},
\end{align}
where $r$ is of order $m-1$ and $|r|\leq C|\partial_{x'}p|_{M-1,1}$.
\end{prop}

\subsection{Garding inequality}

\begin{notation}\label{ee8.14b}
1.  Let $(u,v)$ denote the $L^2$ pairing, which can be extended as the duality pairing on $H^s\times H^{-s}$.
\end{notation}

2.  For a matrix $a$ (symbol or operator) set $\Re a = \frac{a+a^*}{2}$.

\vspace{.1in}

The following Garding inequality is used in the proof of the $L^2$ estimate to obtain bounds from below both in the interior and on the boundary.

\begin{prop}[Garding inequality]\label{ee8.17e}
Consider \;$n\times n$\; matrix symbols $a\in \mathcal{C}^{m}_{M_1}$, $w\in \mathcal{C}^{0}_{M_2}$, where $M_1\geq 2(d+1)+\max(\frac{m}{2},m)+2+[|\frac{m}{2}|]$ and $M_2\geq 2(d+1)+m+2+[|\frac{m}{2}|]$.  Suppose there is a scalar symbol $\chi\in \mathcal{C}^0_{M_1}$ and $c>0$ such that $\chi^2 w=w$ and
\begin{align}\label{ee8.18}
\Re a(x',\zeta)\geq c\langle\zeta\rangle^m \text{ on supp }\chi.
\end{align}
Let $A=a(x',D)$ and $W=w(x',D)$.
Then there exists $C>0$ such that for all $u\in H^{\frac{m}{2}}$
\begin{align}\label{ee8.18a}
\frac{c}{2}|Wu|^2_{\frac{m}{2},\gamma}\leq \Re(AWu,Wu)+C|u|^2_{\frac{m}{2}-1,\gamma}.
\end{align}
The constant $C$ depends on symbol norms of $a$, $w$, and $\chi$.

\end{prop}

\section{Appendix: An existence result for systems with pseudodifferential boundary conditions}\label{s:pseudobc}
\emph{\quad}In this appendix we discuss how to solve
variable-coefficient
hyperbolic initial boundary-value problems with pseudodifferential boundary conditions, such as arise from applying the procedures described in Section \ref{wp} to the quasilinear case.   The problem \eqref{newp} is also of this type, but with constant coefficients.

Consider a
noncharacteristic hyperbolic boundary-value problem
on the half-space $\{x=(x',x_d)=(t,x'',x_d):x_d\geq 0\}$:
\begin{equation}\label{hyp}
\begin{aligned}
&u_{x_d} - G(x,D') u= f\in e^{\gamma t}L^2(x),\\
&e^{\gamma t}\Gamma(x',D')e^{-\gamma t}u(x',0)=g\in e^{\gamma t}L^2(x'),\\
\end{aligned}
\end{equation}
where $G$ is a differential operator with symbol
\begin{equation}\label{G}
G(x, i\tau, i\eta):= -A_d^{-1}(x)
(i\tau + \sum_{j=1}^{d-1}i\eta_j A_j(x)),
\end{equation}
derived from the hyperbolic equation
$u_t+ \sum_{j=1}^d A_j(x) \partial_{x_j}u=A_d f$.   Here the $A_j(x)\in C^\infty$ are $N\times N$ matrices and constant outside some compact set. We suppose  $A_d(x)$ is invertible.  Assume also that  the eigenvalues $\lambda_j(x,\xi)$ of $\sum^d_{j=1}A_j(x)\xi_j$ are $C^\infty$, real, semisimple, and of constant multiplicity for all $(x,\xi)$.\footnote{Matrices $A_j(x)$ of finite regularity can be treated by similar arguments if one uses paradifferential operators. Our assumption on the $\lambda_j$, which implies that the operator $\partial_d-G(x,D')$ in \eqref{hyp} can be conjugated to \emph{block structure} in the sense of \cite{MZ2},  can be weakened as discussed in \cite{MZ2}.}  The boundary operator is a classical pseudodifferential operator of degree zero associated to a $C^\infty$, $p\times N$ matrix symbol $\Gamma(x',\tau,\gamma,\eta)$ using the quantization of section \ref{classical}.

Because of the pseudodifferential boundary conditions this problem is not covered by the standard theory presented, for example, in \cite{CP,BS}.  To state our result,  we first define $D(x',\tau,\gamma,\eta)$ to be the Lopatinski determinant
\begin{align}\label{D}
D(x',\tau,\gamma,\eta)=\det\left(\ker \Gamma(x',\tau,\gamma,\eta), E_+(-G(x',0,\tau,\gamma,\eta))\right),
\end{align}
where the spaces appearing in the determinant are defined just by freezing $x'$.
Here, following \cite[pp. 135--136]{Met4}, we mean the determinant obtained by substituting
orthonormal bases of the spaces involved, so that the modulus of the
determinant is independent of the choice of basis.

\begin{prop}\label{pseudoprop}
With the  above assumptions on \eqref{hyp}, suppose also that the uniform Lopatinski condition holds:
\be\label{Slop}
|D(x',\tau,\gamma,\eta)|\geq C>0, \text{ for }C \text{ independent of }(x',\tau,\eta) \text{ and }\gamma>0.
\ee
Moreover, suppose there hold also the uniform bounds
\be\label{lop2}
\hbox{\rm $|\Gamma|,\,
|\Gamma^\dagger|\le C$,}
\ee
where $\Gamma^\dagger:= \Gamma^*(\Gamma \Gamma^*)^{-1}$ is the pseudo-inverse
of $\Gamma$.
Then, there exists $\gamma_0$ such that for $\gamma\geq \gamma_0$ the problem \eqref{hyp} has a unique solution satisfying
\be\label{maxhyp}
\gamma |u|_{0,\gamma}^2 + \langle u\rangle_{0,\gamma}^2 \le C\left(\frac{|f|_{0,\gamma}^2}{\gamma} + \langle g\rangle_{0,\gamma}^2\right),
\ee
where the constant $C>0$ is independent of $\gamma$.
\end{prop}

\begin{proof}
\textbf{1. }The proof of the a priori estimate \eqref{maxhyp} may be carried out using Kreiss symmetrizers as in \cite{CP,BS}, since $\partial_{x_d}-G(x,D')$ may be conjugated to block structure and the uniform Lopatinski condition holds.
See also \cite[pp. 135--136]{Met4}, particularly Lemma 6.2.4,
for details of the Kreiss symmetrizer
argument and the role of bounds \eqref{lop2}.

\textbf{2. }Next we define the approximate
adjoint problem
\begin{align}\label{hypadj}
\begin{split}
&u_{x_d} + G(x,D')^* u= \tilde f \in e^{-\gamma t}L^2(x),\\
&e^{-\gamma t}\tilde{\Gamma}(x',D')e^{\gamma t}u(x',0)=\tilde g\in e^{-\gamma t}L^2(x')
\end{split}
\end{align}
where $G(x,D')^*$ is the formal adjoint of $G(x,D')$ and the $(N-p)\times N$ matrix symbol $\tilde \Gamma$ is chosen so that $\kernel \tilde \Gamma =(\kernel \Gamma)^{\perp}$ for each $(x',\tau,\gamma,\eta)$.
Without loss of generality, we may take both $\Gamma$ and $\tilde \Gamma$
to have orthonormal columns, so that the symbols satisfy
$\Id= \Gamma^*\Gamma + \tilde \Gamma^* \tilde \Gamma$.
Let us now write $\Gamma(x',D')=\Gamma_{D'}$, $G(x,D')^*=(G_{D'})^*$, etc..
Since
\begin{align}
\begin{split}
&(\Gamma_{D'})^*\Gamma_{D'}=(\Gamma^*\Gamma)_{D'}-R_{D'}\text{ and }\\ &(\tilde{\Gamma}_{D'})^*\tilde\Gamma_{D'}=(\tilde\Gamma^*\tilde\Gamma)_{D'}-\tilde R_{D'},
\end{split}
\end{align}
where $R_{D'}$ and $\tilde R_{D'}$ are operators of order $-1$,  we have
\begin{align}
Id=(\Gamma_{D'})^*\Gamma_{D'}+R_{D'}+(\tilde{\Gamma}_{D'})^*\tilde\Gamma_{D'}+\tilde R_{D'}.
\end{align}

\textbf{3. }Observe that for $\gamma$ large we can invert the $N\times N$ matrix operator $\bp \Gamma_D\\\tilde \Gamma_D\ep$ \emph{exactly} on $L^2$ by first using the calculus to construct an approximate inverse, and then using a Neumann series.
Thus we can solve the operator equation
\begin{align}
R_{D'}+\tilde R_{D'}=\bp (S_{D'})^*&(\tilde S_{D'})^*\ep \bp \Gamma_{D'}\\\tilde \Gamma_{D'}\ep = (S_{D'})^*\Gamma_{D'}+(\tilde S_{D'})^*\tilde\Gamma_{D'}
\end{align}
for $(S_{D'})^*$ and $(\tilde S_{D'})^*$. This gives
\begin{align}
Id=(\Gamma_{D'}+S_{D'})^*\Gamma_{D'}+(\tilde\Gamma_{D'}+\tilde S_{D'})^*\tilde\Gamma_{D'}.
\end{align}
The operators $S_{D'}$ and $\tilde S_{D'}$ have norms of size $O(\gamma^{-1})$ on $L^2$.

\textbf{4. } Let us define
\begin{align}
\Gamma_{D',\gamma}:=e^{\gamma t}\Gamma_{D'}e^{-\gamma t}\text{ and }
\Gamma_{D',-\gamma}:=e^{-\gamma t}\Gamma_{D'}e^{\gamma t}.
\end{align}
Integrating by parts we obtain the duality relation
\ba\label{duality}
&\left((\partial_{x_d} - G_{D'})u,v\right) -\left( u,(-\partial_{x_d} -
(G_{D'})^* )v\right)=
\langle u,v\rangle=\langle e^{\gamma t}u,e^{-\gamma t}v\rangle=\\
&\quad\langle(\Gamma_{D',\gamma}+S_{D',\gamma})u,\Gamma_{D',-\gamma}v\rangle+\langle(\tilde\Gamma_{D',\gamma}+\tilde S_{D',\gamma})u,\tilde \Gamma_{D',-\gamma}v\rangle
\ea
Setting $E=\{v\in C^\infty_c(\overline{\bR}^{d+1}_+):\tilde\Gamma_{D',-\gamma}v=0\}$ and recalling \eqref{hypadj}, we define a linear functional on $(\partial_{x_d} +
(G_{D'})^* )E$ by
\begin{align}
\ell\left((-\partial_{x_d} -
(G_{D'})^* )v\right):=(f,v)-\langle g,\Gamma_{D',-\gamma}v\rangle.
\end{align}
One checks that the forward Lopatinski condition \eqref{D} implies that the adjoint problem \eqref{hypadj} satisfies the backward uniform Lopatinski condition (see \cite{CP}),
and thus
\begin{align}\label{est2}
\gamma|v|^2_{0,-\gamma}+\langle v(0)\rangle^2_{0,-\gamma}\leq \frac{|(\partial_{x_d} +
(G_{D'})^* )v|^2_{0,-\gamma}}{\gamma}\text{ for }v\in E,
\end{align}
where $|u|_{0,-\gamma}:=|e^{\gamma t}u|_{L^2}$.  Using
\begin{align}
|\langle g,\Gamma_{D',-\gamma}v\rangle|\leq C \langle g \rangle_{0,\gamma} \langle v \rangle_{0,-\gamma},
\end{align}
a similar estimate for $(f,v)$, and \eqref{est2}, one obtains readily that $\ell$ is continuous on
$(\partial_{x_d} +
(G_{D'})^* )E$ for the topology induced by $e^{-\gamma t}L^2$.  The Riesz theorem then yields
$u\in e^{\gamma t}L^2$ such that
\begin{align}\label{j1}
(u,(-\partial_{x_d} -
(G_{D'})^* )v)=(f,v)-\langle g,\Gamma_{D',-\gamma}v\rangle \text{ for all }v\in E.
\end{align}
Taking $v\in C^\infty_c\{x_d>0\}$ we obtain from \eqref{j1} that $\partial_{x_d}u-G_Du=f$. In view of \eqref{duality},\eqref{j1} we obtain
\begin{align}
\langle(\Gamma_{D',\gamma}+S_{D',\gamma})u-g,\Gamma_{D',-\gamma}v\rangle=0\text{ for all }v\in E.
\end{align}
Since the matrix symbol $\Gamma$ is surjective when restricted to $\ker\tilde \Gamma$, it follows (by density) that
\begin{align}\label{j5}
(\Gamma_{D',\gamma}+S_{D',\gamma})u=g.
\end{align}

\textbf{5. }Given $(f,g)\in Y_\gamma: =e^{\gamma t}L^2(x)\times e^{\gamma t}L^2(x')$, we have found $u\in e^{\gamma t}L^2$ satisfying the nearby problem
\begin{equation}\label{j3}
\begin{aligned}
&u_{x_d} - G(x,D') u= f\in e^{\gamma t}L^2(x),\\
&(\Gamma_{D',\gamma}+S_{D',\gamma})u(x',0)=g\in e^{\gamma t}L^2(x').
\end{aligned}
\end{equation}
Applying  the estimate \eqref{maxhyp} and treating $S_{D',\gamma}u(x',0)$ as an absorbable error,  we have
\begin{align}\label{j4}
\langle u\rangle_{0,\gamma}\leq C\left(\frac{|f|_{0,\gamma}}{\sqrt{\gamma}}+\langle g\rangle_{0,\gamma}\right),
\end{align}
and thus
\begin{align}\label{j8}
\langle S_{D',\gamma} u\rangle_{0,\gamma} \leq \frac{C}{\gamma}\left(\frac{|f|_{0,\gamma}}{\sqrt{\gamma}}+\langle g\rangle_{0,\gamma}\right).
\end{align}

\textbf{6. }Let $T_a^{-1}$ denote the inverse we have just constructed for the operator
\begin{align}\label{j9}
T_au:=\bp (\partial_{x_d}-G_{D'})u\\ (\Gamma_{D',\gamma}+S_{D',\gamma})u|_{x_d=0}\ep,\text{ and set }Tu:=\bp (\partial_{x_d}-G_{D'})u\\ \Gamma_{D',\gamma}u|_{x_d=0}\ep.
\end{align}
We have $T=T_a+\bp 0\\ -S_{D',\gamma}\ep$, so $TT_a^{-1}=I+\bp 0\\ -S_{D',\gamma}\ep T_a^{-1}:=I+M_\gamma$.  The estimate \eqref{j8} implies that $M_\gamma$ has norm $<1$ on $Y_\gamma$ for $\gamma$ large, so we can invert $I+M_\gamma$ on $Y_\gamma$ by a Neumann series, and thereby invert $T$.

\end{proof}

\section{The example of Rao}\label{s:Rao}
We now consider the example left untreated in the analysis of
\eqref{NSeq} in \cite{R}, lying in the case (ii) described in
Section \ref{s:mixed}.
Consider perturbations about a constant
boundary-layer solution of \eqref{NSeq} in dimension $d=2$,
with nonvanishing tangential velocity
\be\label{nonvan}
u \ne 0,
\ee
and normal velocity $v$ noncharacteristic, i.e.,
$0\ne v, v\pm c,$
where $c:=\sqrt{p_\rho + \frac{p p_T}{\rho^2 c_v}}$ is sound speed.
Here, we are assuming a polytropic temperature law $e=c_v T$
(used but not stated in Chapter 5 of \cite{R}) and an unspecified
pressure law $p=p(\rho, T)$, with $ p_\rho + \frac{p p_T}{\rho^2 c_v}>0$.

Working with variables $U:=(\rho, u, v, T)^T$, we find, following
\cite{R}, that the equations \eqref{NSeq} may
be expressed in quasilinear form as
\be\label{quasilin}
A_0 \partial_t U + \sum_{j=1}^2A_j \partial_{x_j} U=
\eps \sum_{j,k}\partial_{x_j}(B_{jk}\partial_k U),
\ee
with
\be\label{A0s}
A_0=\bp 1&0&0&0\\
u & \rho & 0 & 0 \\
v&0 & \rho & 0\\
E&\rho u& \rho v &\rho c_v\ep,
\quad
A_0^{-1}=\bp 1&0&0&0\\
-u/\rho  & 1/\rho & 0 & 0 \\
-v/\rho &0 & 1/\rho & 0\\
-E/\rho c_v & -u/c_v& -v/c_v &1/\rho c_v\ep,
\ee
\be\label{As}
A_1=\bp
u& \rho &0& 0\\
p_\rho/\rho & u &0& p_T/\rho\\
0 & 0& u &0\\
0 &  p/\rho c_v &  0& u\ep
\quad
A_2=\bp
v&0& \rho &0\\
0 & v& 0 &0\\
p_\rho/\rho & 0 &v& p_T/\rho\\
0 & 0 & p/\rho c_v &  v\ep
\ee
and
\be\label{Bs}
B_{jk}=\bp 0 & (0,0,0)\\
0& \beta_{jk}\ep,
\ee
where $\beta_{jk}$ is elliptic in the
sense that the eigenvalues of $\sum_{jk}\beta_{jk}\xi_j\xi_k$
have real part $\ge c|\xi|^2$, $c>0$, for all $\xi\in \RR^d$.

Note that the parabolic terms involving $B_{jk}$ are
of the more general form treated in \cite{GMWZ5,GMWZ6,R}
and not the Laplacian form to which we have restricted for
simplicity in \eqref{aa}.
In particular, they are degenerate parabolic, no diffusion being
present in the equation for density $\rho$, whose principal
part $\rho_t + (u,v)^T \cdot \nabla_x \rho$
is a hyperbolic transport equation.
Though it does not change the theory in any essential way,
this will require a bit of discussion as we go along.

To start with, and most significantly, the
fact that the $\rho$ equation is hyperbolic means
(see \cite{Se2,SZ,Z,GMWZ5,GMWZ6})
that a boundary condition is needed for $\rho$ only
when the convection $v$ in the normal direction is inward on
the boundary, i.e., $v>0$.
Boundary conditions on the ``parabolic variables'' $(u,v,T)$
must always be imposed, and may be Neumann or Dirichlet type.
Here, we will impose Dirichlet conditions on $\rho$ (when needed),
$u$, $v$, and Neumann, or ``insulative'' conditions $\partial_d T=0$ on $T$,
following (one version of) engineering practice as described in the
introduction.

Assume, now, that we are in
the {\it incoming supersonic} case in
the notation of \cite{GMWZ5,R},
the case left open in the treatment of \cite{R}, that
\be\label{subin}
0<c < v.
\ee
By the previous discussion,
we require a full set of boundary conditions for the
hyperbolic--parabolic problem, including a condition on $\rho$.
We thus take boundary conditions
\be\label{Rbc}
\hbox{\rm $(\rho,u,v)=(g_1,g_2,g_3)$ and
$\partial_{x_2} T=0=g_4$ at $x_d=0$.}
\ee
The first-order hyperbolic part of \eqref{quasilin},
comprising the Euler equations, may be written as
\be\label{hypeg}
U_t + \sum_j \tilde A_j \partial_{x_j} U=0,
\ee
where $\tilde A_j:=A_0^{-1}A_j$.
Recalling the standard computation that the characteristics
of $A_2$ are $v-c$, $v$, and $v+c$ (with multiplicity two),
we find, consulting \eqref{subin},
that this hyperbolic problem is {\it totally incoming},
i.e., all eigenvalues of $A_2$ are strictly positive.

Though we shall not carry it out here, it is not difficult to see
using the methods of \cite{GMWZ5,GMWZ6} that in this situation
there can exist no small-amplitude boundary layers other than
the trivial, constant layer, for the simple reason
that any rest state satisfying \eqref{subin} must, by the dimensional
counting arguments of \cite{GMWZ5}, be a repellor for the standing-wave
ODE, so cannot be
the limit as $x_d\to +\infty$ of a nonconstant
standing wave (boundary layer).
Thus, the same derivation as in Section \ref{s:mixed}, case (ii),
of a formal boundary-layer expansion applies, yielding an outer problem with
$\mathcal{D}=3$ Dirichlet conditions
and $\mathcal{N}=1$ Neumann boundary conditions,
the same ones imposed on the full hyperbolic--parabolic problem.

At this point, having derived an outer problem,
we can forget its hyperbolic--parabolic origins and analyze it
as in Section \ref{s:mixed}.
As noted in Lemma \ref{lopsat}(d), we have that the outer
problem is weakly Lopatinski stable, and solvable for the constant-coefficient
problem, provided
$\bp \Gamma_1\\ \Gamma_2 \tilde A_2^{-2}\ep$
or, equivalently,
$\bp \Gamma_1 A_0^{-1} A_2\\ \Gamma_2  \ep$
is full rank, where
\be\label{Gammas}
\Gamma_1:=\bp 1 & 0 & 0&0\\
0& 1 & 0&0  \\
0& 0& 1 &0  \ep,
\quad
\Gamma_2:=\bp 0 & 0 & 0&1\ep.
\ee
are the boundary matrices corresponding to the boundary conditions
described above.

Computing, we have
$$
\bp \Gamma_1 A_0^{-1} A_2\\ \Gamma_2  \ep=
\bp M & *\\
O_{1\times 3} &  1\ep,
$$
where $M= \bp 1 & 0 & 0\\
-u/\rho & 1/\rho  & 0\\
-v/\rho & 0 & 1/\rho \ep
\bp v & 0 & p\\
0 & v & 0\\
p_\rho/\rho & 0 & v \ep
$
is invertible provided that
$$
0\ne \det
\bp v & 0 & p\\
0 & v & 0\\
p_\rho/\rho & 0 & v \ep
= v(v^2-p_\rho)=
v^2-c^2 + pp_T/\rho^2,
$$
or
\be\label{wellposedness}\tag{Weak Lop}
v^2-c^2 + pp_T/\rho^2.
\ee
By the supersonic condition \eqref{subin},
this is evidently true under the standard assumption
\be\label{suff}
p_T>0,
\ee
satisfied in most typical applications,
in particular, for an ideal gas pressure law $p=R \rho T$, where $R>0$
is the universal gas constant.
However, in general \eqref{suff}, hence apparently also
\eqref{wellposedness}, need not be true, for example for nonstandard
materials/phases such as ice, for which pressure decreases with temperature.

Continuing, let us check not only reality but also semisimplicity/constant
multiplicity of the characteristic roots of the boundary problem
$$
\bp
\bp \Id_3 & 0\ep \\
\Gamma_1 A_2^{-1}A_0 \ep \partial_t v
+
i\eta
\bp
\bp \Id_3 & 0\ep \\
\Gamma_2 A_2^{-1} A_1\ep \partial_1 v=0
$$
derived above, or, equivalently, eigenvalues of
\be\label{trimat}
-i\eta
\bp
\bp \Id_3 & 0\ep \\
\Gamma_2 A_2^{-1}A_0 \ep^{-1}
\bp
\bp \Id_3 & 0\ep \\
\Gamma_2 A_2^{-1} A_1\ep
=
-i\eta
\bp
\Id_3 & 0\\
* & (\Gamma_2 A_2^{-1}A_0 \Gamma_2^T)^{-1}
\Gamma_2 A_2^{-1} A_1\Gamma_2^T \ep,
\ee
where, by our previous calculations,
$\Gamma_2 A_2^{-1}A_0 \Gamma_2^T\ne 0$.
By the lower block triangular structure of \eqref{trimat}, this
is implied by the nonvanishing property
$\Gamma_2 A_2^{-1} A_1\Gamma_2^T \ne 0.$
Noting that $A_1\Gamma_2^T=\bp 0 \\p_\rho/\rho\\0\\u\ep$,
while, by Kramer's rule,
$$
\Gamma_2 A_2^{-1}=(\det A_2)^{-1}\bp vp_Tp/\rho^2c_v & 0&
-v^2p/\rho c_v & (v^2-p_\rho)v\ep,
$$
we find that
$
\Gamma_2 A_2^{-1} A_1\Gamma_2^T = uv(v^2-p_\rho)/\det A_2
$
is nonvanishing if \eqref{wellposedness} holds,
under the nonvanishing
tangential velocity assumption \eqref{nonvan}.
The condition of constant multiplicity is trivially satisfied, since
$\eta$ is dimension one.

On the other hand, when \eqref{nonvan} is violated, we have,
computing,
$$
\Gamma_2 A_2^{-1} A_1 =
(\det A_2)^{-1}
\bp * & 0 & v^2\big(\rho(v^2-p_\rho)-p/c_v\big) & 0 \ep,
$$
which in general does not vanish, and so the matrix on
the righthand side of \eqref{trimat} contains a Jordan block,
violating semisimplicity. (Specifically, it has all zero eigenvalues,
by lower triangular form and vanishing on the diagonal, but is
not identically zero.)
This shows in passing that semisimplicity (hence maximal bounds)
can fail for method two when the weak Lopatinski condition is satisfied.

Thus, when \eqref{wellposedness} holds,
in particular under the standard assumption \eqref{suff},
we have also weak hyperbolicity together with semisimplicity/constant
multiplicity of characteristics of the boundary problem,
so obtain by the theory of Section \ref{s:second}
(together with the observation above that boundary layers are absent
at all orders in this case)
existence of approximate solutions to all orders.
However, the question of convergence is still open up to now, for lack
of associated hyperbolic--parabolic estimates.

\begin{rem}
\textup{
Example \ref{inceg} is closely related, and gives
a similar conclusion, as does any problem with a single Neumann condition.
}
\end{rem}

\section{Extension of the second approach to the non-totally incoming case}\label{s:new}

\emph{\quad} We return now to the reduced hyperbolic problem in its original form \eqref{res2} in the general case of mixed-type boundary conditions, assuming that $A_d$ is nonsingular and that the  $A_j$ are symmetric.   We also assume that the operator $L$ can be conjugated to block structure in the sense of \cite{MZ2}.

We split the problem
\eqref{res2} into two parts: a problem
with homogeneous Neumann boundary conditions
\be\label{Npart}
Lv=f,\quad \pi_+(A_d)\partial_d v|_{x_d=0}=0,
\ee
and a problem with homogeneous forcing and mixed boundary conditions
\be\label{mpart}
Lw=0,\quad
\Gamma_1 w |_{x_d=0}=g_1-\Gamma_1 v|_{x_d=0}:=h_1,
\quad
\tilde \Gamma_2 \partial_d w|_{x_d=0}=\tilde g_2-\tilde\Gamma_2 \partial_d v|_{x_d=0}:=h_2
\ee
where $u_0:=v+w$.   In $t<0$ we have $v=0$ and $w=0$.

The first problem may be solved as in Section \ref{s:linear}.   To solve the second problem we take the Laplace--Fourier transform,
substitute the usual boundary symbol for $\partial_d$, and multiply $\Gamma_1$ by $(\gamma+i\tau)$ to obtain
\be\label{step}
\bp (\gamma+i\tau)\Gamma_1 \\ -\tilde \Gamma_2
A_d^{-1}(\gamma +i\tau +\sum_{j=1}^{d-1} i\eta_j A_j)\ep \hat{w}(0)=
\bp (\gamma+i\tau) \hat{h}_1\\ \hat{h}_2\ep.
\ee
Note that $w$ is  a decaying solution of $Lw=0$ if and only if
$\hat{w}(0)\in \EE_+(A_d^{-1}(\gamma +i\tau
+\sum_{j=1}^{d-1} i\eta_j A_j))$\footnote{Here and in solving \eqref{Npart} we use the assumption of constant coefficients.}.   Thus, the problem $Lw=0$ on $x_d\geq 0$ with boundary conditions \eqref{step} is equivalent
to the problem \emph{on the boundary} with enlarged boundary conditions
\be\label{stepen}
\bp \Gamma_0(\gamma,\tau,\eta) (\gamma+i\tau)\\ \Gamma_1 (\gamma+i\tau)\\ -\tilde \Gamma_2
A_d^{-1}(\gamma +i\tau +\sum_{j=1}^{d-1} i\eta_j A_j)\ep \hat w(0)=
\bp 0\\ \hat{h}_1\\ \hat{h}_2\ep,
\ee
where $\Gamma_0(\gamma,\tau,\eta)$ is a matrix whose rows are orthogonal to $\EE_+
( A_d^{-1}(\gamma +i\tau +\sum_{j= 1}^{d-1} i\eta_j A_j))$\footnote{By the block structure assumption this space extends continuously to $\gamma =0$ \cite{MZ3}}.  But, this
is equivalent to the Cauchy problem on the boundary
\be\label{cauchyprob}
\bp \Gamma_0(\gamma,\tau,\eta) \\ \Gamma_1 \\ \tilde \Gamma_2 A_d^{-1}\ep
w_t
+
\sum_{j= 1}^{d-1} \bp 0 \\ 0\\ \tilde\Gamma_2 A_d^{-1}A_j\ep w_{x_j}=\bp 0 \\{h}_1\\{h}_3\ep \text{ where }h_3=-h_2.
\ee

\begin{defn}\label{hypdef2}
Let $\overline{S}^d_+=\{(\tau,\gamma,\eta):|\tau,\gamma,\eta|=1,\gamma\geq 0\}$. Parallel to
Definition \ref{hypdef}, we
we say that \eqref{cauchyprob}
is \emph{ evolutionary  at $(\tau_0,\gamma_0,\eta_0) \in \overline{S}^d_+$} if
$$
\mathcal{A}_0=
\bp \Gamma_0(\gamma_0, \tau_0,\eta_0) \\ \Gamma_1 \\ \tilde \Gamma_2 A_d^{-1}\ep
$$
is invertible.   We say \eqref{cauchyprob} is \emph{weakly hyperbolic at $(\tau_0,\gamma_0,\eta_0)\in \overline{S}^d_+$} if, in addition, the first-order system
\be\label{firstordersys}
(\gamma + i\tau)
\hat w
+
\sum_{j=1}^{d-1}
i\eta_j
\bp \Gamma_0(\gamma_0, \tau_0,\eta_0) \\ \Gamma_1 \\ \tilde \Gamma_2 A_d^{-1}\ep^{-1}
\bp 0 \\ 0\\  \tilde \Gamma_2 A_d^{-1} A_j\ep \hat w=
\bp \Gamma_0(\gamma_0, \tau_0,\eta_0) \\ \Gamma_1 \\ \tilde \Gamma_2 A_d^{-1}\ep^{-1}
\bp 0 \\\hat h_1\\\hat h_3\ep
\ee
has pure imaginary characteristics, defined as eigenvalues of
the homogeneous degree one symbol
$\mathcal{A}(\eta; \gamma_0 ,\tau_0 ,\eta_0 ):=
\sum_{j=1}^{d-1} i\eta_j \tilde{ \mathcal{A}}_j(\gamma_0,\tau_0, \eta_0)$ with homogeneous
degree zero coefficients
$$
\tilde{\mathcal{A}}_j(\gamma_0,\tau_0, \eta_0):=
\mathcal{A}_0^{-1}(\gamma_0,\tau_0,\eta_0)\mathcal{A}_j=
\bp \Gamma_0(\gamma_0, \tau_0,\eta_0) \\ \Gamma_1 \\ \tilde \Gamma_2 A_d^{-1}\ep^{-1}
\bp 0 \\ 0\\  \tilde \Gamma_2 A_d^{-1} A_j\ep,
\quad
\mathcal{A}_j:=
\bp 0 \\ 0\\  \tilde \Gamma_2 A_d^{-1} A_j\ep.
$$
\end{defn}

Unlike the situation of the totally incoming case,
the evolutionarity and weak hyperbolicity conditions  must be checked separately, and do not
follow simply by the weak Lopatinski condition. However, we have:

\begin{prop}
System \eqref{cauchyprob}
is evolutionary and weakly hyperbolic at all $(\gamma_0, \tau_0,\eta_0) \in \overline{S}^d_+$
only if the original problem \eqref{res2} satisfies
the weak Lopatinski condition.
\end{prop}

\begin{proof}
If the weak Lopatinski condition fails for \eqref{res2}, then
for some $\gamma>0$ and $\tau,\eta$, there exists $\hat w$
such that $\Gamma_0(\gamma,\tau,\eta) \hat w=0$,
$ \Gamma_1 \hat w=0$, and $\tilde \Gamma_2
A_d^{-1}(\gamma +i\tau +\sum_{j=1}^{d-1} i\eta_j A_j) \hat w=0$,
whence \eqref{stepen} and \eqref{cauchyprob} hold with
$\hat h_1=0$, $\hat h_3=0$, as, by the evolutionarity assumption,
does \eqref{firstordersys}.
But, this, by inspection, means that $\gamma+i\tau$ is
an eigenvalue of $\mathcal {A}(\eta; \gamma, \tau, \eta)$
that is not purely imaginary, contradicting weak hyperbolicity
at $(\gamma,\tau,\eta)$.
\end{proof}

When the  evolutionarity and weak hyperbolicity conditions do hold, we have a situation
analogous to that of Section \ref{s:second},
but for the generalized first-order Cauchy problem
\be\label{nonloc}
\partial_t w + \sum_{j=1}^{d-1} \mathbb{A}_j \partial_{x_j}w= \mathbb{G},
\ee
on the boundary $x_d=0$, where
$\widehat{\mathbb{A}_j w}:= \tilde{\mathcal{A}}_j(\gamma,\tau, \eta)\hat w$ and
$
          \hat{ \mathbb{G}}:=
\bp \Gamma_0(\gamma, \tau,\eta) \\ \Gamma_1 \\ \tilde \Gamma_2 A_d^{-1}\ep^{-1}
\bp 0 \\\hat h_1\\\hat h_3\ep.
$

As in the standard case, evolutionarity plus weak hyperbolicity is
not sufficient to give well-posedness of the problem \ref{nonloc},
but requires some additional structural assumptions.
The following sufficient conditions are often applicable.

\begin{lem}\label{wellcond2}
Assuming evolutionarity at all $(\tau_0,\gamma_0,\eta_0)\in \overline{S}^d_+$,
a sufficient condition for well-posedness of \eqref{nonloc}\footnote{By well-posedness of \eqref{nonloc} we mean that there exists $\gamma_0$ such that for $\gamma\geq\gamma_0$, if $\bG\in e^{\gamma t}L^2$, then there is a solution $w\in e^{\gamma t}L^2$.}is that the eigenvalues of $\mathcal{A}(\eta;\gamma_0, \tau_0, \eta_0)$
be semisimple, pure imaginary, and of constant multiplicity
for all $(\tau_0,\gamma_0,\eta_0)\in \overline{S}^d_+$,  $\eta\ne 0$.  In
in this case we have the uniform resolvent estimate
\be\label{resestD}
| (\gamma+i\tau+ \sum_{j=1}^{d-1} i\eta_j \tilde{\mathcal{A}}_j(\gamma,\tau,\eta) )^{-1}|\le
C/\gamma,
\quad \gamma>0.
\ee
\end{lem}

\begin{proof}
By taking the Laplace-Fourier transform of the problem \eqref{nonloc}, we see that well-posedness follows from the
estimate \eqref{resestD}.   We prove \eqref{resestD} by first proving the family of estimates
\be\label{strongerest}
| (\gamma+i\tau+ \sum_{j= 1}^{d-1} i\eta_j \tilde{\mathcal{A}}_j(\gamma_0,\tau_0,\eta_0) )^{-1}|\le
C/\gamma,
\quad \gamma>0
\ee
for the problems frozen at $(\gamma_0,\tau_0,\eta_0)\in\overline{S}^d_+$.
For a given $(\gamma_0,\tau_0,\eta_0)$ this estimate can be obtained by conjugating the frozen system to Jordan form and using semisimplicity.
Using compactness of the unit half-sphere and continuity of the $\tilde\cA_j$, we conclude that the estimate is uniform for all
$(\tau_0,\gamma_0,\eta_0)\in\overline{S}^d_+$.  Restricting to the diagonal $(\gamma_0, \tau_0, \eta_0)=(\gamma,\tau,\eta)/|\gamma,\tau,\eta|$ and using degree zero homogeneity
of $\mathcal{A}_j$,
we obtain \eqref{resestD}.
\end{proof}

The conditions of Lemma \ref{wellcond2}, which refer to a large family of frozen systems,  seem hard to check and possibly over-restrictive; however,
as examples below show, the conditions are sometimes satisfied.  Since the matrix
$\mathcal{A}$ has a guaranteed large block of semisimple zero
eigenvalues
associated with the boundary conditions $\Gamma_0$
and $\Gamma_1$ , in simple cases verification
amounts to checking nonvanishing of a single remaining eigenvalue (see the proof of Proposition \ref{sharp}).
For systems of size $N=3$, for example,
with $2$ Neumann conditions,
$1$ Dirichlet condition, and $2$ incoming modes, the matrix
$\mathcal{A}$ has a two-dimensional kernel, so we are in the
situation described.
Likewise, for systems with $\leq 2$ Neumann boundary conditions,
if there is even a single outgoing mode, then we are in this case.
Of course, one might always attempt to establish
\eqref{resestD} by direct matrix inversion, as well.

When we do have conditions sufficient to
obtain well-posedness of \eqref{cauchyprob}, hence
bounds \eqref{resestD} on $w|_{x_d=0}$,
we may then from the trace information on
$w|_{x_d=0}$ obtain interior bounds on $w$ using any
Lopatinski stable Dirichlet boundary condition, for example a maximally
dissipative one.
Putting the estimates for $v$ and $w$ together, we obtain a
Kreiss-type estimate with losses for the solution $u_0$, similarly
as in the totally incoming case.

\begin{rem}\label{caveat2}
\textup{Even when the semisimplicity and constant multiplicity assumptions of Lemma \ref{wellcond2} are not satisfied, it may still be possible to implement method two in some problems.  This would require refined estimates for the Cauchy problem that to  our knowledge have not yet been carried out.}

\end{rem}

\subsubsection{A sharp condition in a special case}\label{s:sharp}
Let the number of Neumann conditions be one greater than the number
of incoming hyperbolic modes, i.e., the number of reduced Neumann conditions
be one.
\begin{prop}\label{sharp}
Let the number of reduced Neumann conditions be one.
Then, assuming evolutionarity at all $(\tau,\gamma,\eta)\in\overline{S}^d_+$,
a necessary and sufficient condition for well-posedness of \eqref{nonloc} (i.e., for the
estimate \eqref{resestD})
is that the eigenvalues of $\mathcal{A}(\eta ;0, 0, \eta)$
be semisimple, pure imaginary, and of constant multiplicity with respect
to $\eta$, for all $|\eta|=1$, or, equivalently, the scalar condition
$\sum_{j\ne d}i\eta_j \bar \alpha_j\ne 0$ for $\eta\ne 0$, where
$\bar \alpha_j:=  (\mathcal{A}_j \mathcal{A}_0^{-1})_{NN}(0,0,\eta)$.
\end{prop}

\begin{proof}
Conjugating by the uniformly invertible $\mathcal{A}_0$,
we have that \eqref{resestD} is equivalent to
\be\label{resestC}
\hbox{\rm
    $| ((\gamma+i\tau)+ \sum_{j\neq d} i\eta_j
{\mathcal{A}}_j(\gamma,\tau,\eta)\mathcal{
A}_0^{-1} )^{-1}|\le C/\gamma$
for $\gamma>0$,}
\ee
or, computing explicitly in block-diagonal form,
\ba\label{ulx1}
C/\gamma&\geq
\Big|     \bp \gamma+i\tau & 0\\
\sum_{j\ne d} i\eta_j\beta_j &
\gamma +i\tau + \sum_{j\ne d}i\eta_j \alpha_j \ep^{-1} \Big|
\\ &=
\Big|
\bp \frac{1}{\gamma+i\tau}&0\\
\frac{-(\sum_{j\ne d}i\eta_j \beta_j)
(\gamma +i\tau + \sum_{j\ne d}i\eta_j \alpha_j)^{-1}}
{\gamma+i\tau } &
(\gamma +i\tau + \sum_{j\ne d}i\eta_j \alpha_j)^{-1}
\ep\Big|
,
\ea
where the upper blocks consist of $N-1$ rows, the lower blocks of $1$ row,
and
\ba\label{alb}
\alpha_j(\gamma,\tau,\eta)&:=
(\mathcal{A}_j \mathcal{A}_0^{-1})_{NN}(\gamma,\tau,\eta),\\
\beta_j&=(\beta_j^1,\dots, \beta_j^{N-1}),
\quad
\beta_j^i(\gamma,\tau,\eta):=
(\mathcal{A}_j \mathcal{A}_0^{-1})_{Ni}(\gamma,\tau,\eta)
\ea
are real and scalar.
This holds for $|\gamma,\tau|<<|\eta|$ if and only if
$\sum_{j\ne d}i\eta_j\alpha_j (0,0,\eta)\ne 0$ for $\eta\ne 0$,
or, equivalently, $|\sum_{j\ne d}\eta_j\alpha_j (0,0,\eta)|\ge |\eta|/C>0$.
For $|\eta|\le C|\gamma,\tau|$, it holds always,
in the $1-1$, $1-2$, and $2-2$ blocks by inspection
using the fact that $\gamma+i\tau +\sum_{j\ne d}i\eta_j \alpha_j$
is scalar,
and in the $2-1$ block by
$
\Big|\frac{-(\sum_{j\ne d}i\eta_j \beta_j) } {\gamma+i\tau } \Big|\le
C|\eta|/|\gamma,\tau|\le C_2$.
\end{proof}

\begin{rem}\label{noscalar}
\textup{
Though the proof of Proposition \ref{sharp} relied heavily on the
scalar nature of block $\gamma+i\tau +\sum_{j\ne d}i\eta_j \alpha_j$,
the same argument shows in the general case that, assuming evolutionarity,
necessary and sufficient conditions for \eqref{resestD} are well-posedness
of the reduced symbol in the $2-2$ block,
\be\label{reducedres}
\hbox{\rm
$|(\gamma+i\tau +\sum_{j\ne d}i\eta_j \alpha_j)^{-1}|\le C/\gamma$
for $\gamma>0$,}
\ee
and noncharacteristicity at $(\gamma,\tau)=(0,0)$,
$\det (\sum_{j\ne d}i\eta_j\alpha_j (0,0,\eta))\ne 0$, for $\eta\ne 0$.
Though still nontrivial, the formulation \eqref{reducedres} represents
a substantial reduction in order.
}
\end{rem}


\subsection{Some illustrative examples}
We now give some examples indicating the range of possible behaviors.
Before presenting these,
we first
prove two theoretical results that serve to frame the discussion.

\begin{lem}\label{lopcomp}
Assuming that $\mathcal{A}_0$ is invertible for $(\tau,\gamma,\eta)\in\overline{S}^d_+$ and
that
\be\label{rankcond}
\hbox{\rm
$M:=\bp \Gamma_1\\
\tilde \Gamma_2A_2^{-1}(\gamma+i\tau+\sum_{j\ne d}i\eta_jA_j)\ep$
is full rank for $\gamma\ge 0$ and $|\gamma,\tau,\eta|=1$,}\footnote{
For pure Neumann boundary conditions, and non-totally incoming hyperbolic
characteristics, this holds generically for dimension $d=2$ in the
sense that it is true for all choices of boundary matrices $\Gamma_2$
except for a measure zero set, and fails generically for $d\ge 3$
in the sense that it is false away from a measure
zero set of matrix entries $A_1,\dots, A_d$ (precisely, those for
which $A_1, \dots, A_{d-1}$ all
share a common eigenvector).
},
\ee
the uniform Lopatinski condition for the rescaled boundary
condition of method one \eqref{lop}
may be expressed, equivalently, in the convenient form
\be\label{lopver}
\hbox{\rm
 $|(\gamma +i\tau)^{ \rank \tilde \Gamma_2 -N}
\det
 (\gamma+i\tau+ \sum_{j\neq d} i\eta_j \tilde{\mathcal{A}}_j(\gamma,\tau,\eta) )|
\geq \delta_0> 0$
for all $\gamma> 0$, $|\gamma,\tau, \eta|=1$.
}
\ee
\end{lem}

\begin{proof}
Under \eqref{rankcond},
$|M|$ and its pseudo-inverse $M^\dagger:=M^*(MM^*)^{-1}$
are uniformly bounded on $\gamma\ge 0$, $|\gamma,\tau, \eta|=1$, by
continuity, as is  $\Gamma_0$ (which has orthonormal rows) and its pseudo-inverse.
Thus, (see \cite[pp. 135--136]{Met4})
the uniform Lopatinski condition, or transversality of the kernels of $M$ and $\Gamma_0$,
is equivalent to $|\det \bp \Gamma_0\\M\ep|\ge \delta_0>0$,
whence equivalence of \eqref{lopver} follows by uniform boundedness of
$\det \mathcal{A}_0^{-1}$
(a consequence of continuity and evolutionarity) and
$$
\det \bp \Gamma_0 (\gamma+i\tau)\\\Gamma_1(\gamma+i\tau)\\
\tilde \Gamma_2  A_d^{-1}(\gamma + i\tau + \sum_{j\ne d}i\eta_j A_j\ep
=
(\gamma +i\tau)^{ \rank N-\tilde \Gamma_2} \det \bp \Gamma_0\\M\ep.
$$
\end{proof}

\bc
Under the assumptions of Lemma \ref{lopcomp},
the uniform Lopatinski condition
is sufficient (but not necessary; see the examples below)
for resolvent estimate \eqref{resestD}.
\ec

\begin{proof}
Conjugating by the uniformly invertible $\mathcal{A}_0$,
we have that \eqref{resestD} is equivalent to
\be\label{resestB}
\hbox{\rm
    $| ((\gamma+i\tau)+ \sum_{j\neq d} i\eta_j
{\mathcal{A}}_j(\gamma,\tau,\eta)\mathcal{
A}_0^{-1} )^{-1}|\le C/\gamma$
for $\gamma>0$,}
\ee
or, computing explicitly in block-diagonal form,
\be\label{ulx}
C/\gamma\geq
\Big|     \bp \gamma+i\tau & 0\\
|\gamma,\tau,\eta|\beta &|\gamma,\tau,\eta| \alpha_+\ep^{-1} \Big|=
\Big|
\bp \frac{1}{\gamma+i\tau}&0\\
\frac{-\beta \alpha_+^{-1}}{\gamma+i\tau } &
\frac{\alpha_+^{-1}}{|\gamma,\tau,\eta|} \ep\Big|,
\ee
where $\alpha_+:=
\frac{\gamma+i\tau + \sum_{j\ne d}i \alpha_j \eta_j}{|\gamma,\tau,\eta|}$ and
$\beta:= \frac{\sum_{j\ne d}i\eta_j \beta_j}{|\gamma,\tau,\eta|}$,
$\alpha_j$ and $\beta_j$ homogeneous degree zero in $(\gamma,\tau,\eta)$,
defined as in \eqref{alb}.
By \eqref{lopver}, Uniform Lopatinski is equivalent to
$|\det \alpha_+|\geq \delta_0> 0$, whence, by boundedness of $\alpha_+(\gamma,\tau,\eta)$, $|\alpha_+^{-1}|$ is uniformly bounded.
This, along with boundedness of $\beta$, verifies \eqref{ulx}.
\end{proof}

\begin{example}\label{2eg}
Consider the system
$u_t+A_1 u_{x_1}+A_2 u_{x_2}-\eps\Delta_x u=f$, $u\in \RR^3$,
with two incoming hyperbolic modes,
two Neumann conditions $\Gamma_2 u|_{x_2}=g_2$, and one Dirichlet
condition $\Gamma_1 u|_{x_2=0}=g_2$, given by
\be\label{Acoeffs}
A_1:=\bp 0&1&0\\ 1&0&0\\ 0&0&1\ep,\quad
A_2:=\bp -1&0&0\\ 0&1&0\\ 0&0&1\ep,
\ee
and
$\Gamma_1:=\bp 0 & 1 & 0\ep$,
$\Gamma_2:=\bp 1 & * & *\\ 0 & \alpha & \beta \ep$.
We have evidently $\hat \Gamma_2=\bp 0 & \alpha & \beta \ep$.
More, by the decoupled form of $A_1$, $A_2$, and the fact
that $\Pi_+(A_2)u=0$ is a maximally dissipative
Dirichlet condition, we find without calculation
that $\EE_+(A_2^{-1}(\gamma +i\tau + i\eta A_2))$
is spanned by $(0,0,1)^T$ and $(*, 1,0)^T$, so that
\be\label{Gamma0coeff}
\mathcal{A}_0:=
\bp \Gamma_0\\\Gamma_1\\ \hat \Gamma_2 A_2^{-1}\ep=
\bp 1& \delta(\gamma,\tau, \eta)&0\\
0 & 1 & 0 \\
0 & \alpha & \beta \ep,
\quad
\mathcal{A}_1:=
\bp 0\\0\\ \hat \Gamma_2 A_2^{-1}A_1\ep=
\bp  0 & 0 & 0\\
0 & 0 & 0\\
\alpha & 0 & \beta \ep,
\ee

Computing, we find that evolutionarity, $\mathcal{A}_0$ invertible
is satisfied when $\beta\ne 0$, with
\be
\mathcal{A}_0^{-1}=
\bp 1& -\delta(\gamma,\tau, \eta)&0\\
0 & 1 & 0 \\
0 & -\alpha/\beta & 1/\beta \ep,
\quad
\mathcal{A}_0^{-1}\mathcal{A}_1=
\bp  0 & 0 & 0\\
0 & 0 & 0\\
\alpha/\beta & 0 & 1 \ep,
\ee
in which case $\mathcal{A}_0^{-1} \mathcal{A}_1$ is independent of
$(\gamma, \tau, \eta)$, with real,
semisimple constant-multiplicity characteristics
$0$, $0$, $1$.
Thus, by Lemma \ref{wellcond2} method
two works in this case.

On the other hand, by \eqref{lopver} and the easily-checked \eqref{rankcond}
for $\alpha\neq 0$, the uniform Lopatinski condition for the rescaled boundary
condition of method one is equivalent to
$$
0<\delta_0\leq |(\gamma+i\tau)^{-2}
\det \bp \gamma+i\tau& 0 &0\\
0 & \gamma+i\tau & 0\\
i\eta \alpha/\beta & 0 & \gamma+i\tau +i\eta  \ep|
=|\gamma + i(\tau +\eta)|
$$
for $\gamma > 0$, $|\gamma,\tau,\eta|=1$.   This clearly fails for
$\gamma=0$, $\eta=-\tau$.
Thus, method two can apply even when uniform Lopatinski fails.
\end{example}

\begin{example}\label{neueg2}
Consider again Example \ref{neueg}, of the first-order wave equation
with drift $\alpha$,
$$
A_1=\bp 0 & 1\\ 1 & 0 \ep,
\quad
A_2 =\bp 1+\alpha & 0 \\ 0 & -1+\alpha \ep,
\quad
|\alpha|<1,
$$
with full Neumann boundary conditions, so that
$\tilde \Gamma_2= \bp 1 & 0\ep$, $\Gamma_0=(\delta(\gamma,\tau,\eta),1)$.
Then,
$\tilde \Gamma_2 A_2^{-1}=\bp \frac{1}{1+\alpha}& 0\ep$ and
$\tilde \Gamma_2 A_2^{-1}A_1=
\bp 0 & \frac{1}{1+\alpha} \ep,$ so that
$ \mathcal{A}_0= \bp \delta & 1\\ \frac{1}{\alpha+1}& 0\ep$,
$\mathcal{A}_1= \bp 0 & 0\\0&\frac{1}{1+\alpha}\ep$,
$\mathcal{A}_0^{-1}= \bp 0 & (1+\alpha)\\ 1& -\delta(1+\alpha) \ep,$
and thus
$
\mathcal{A}_0^{-1}\mathcal{A}_1=
\bp 0 & 1\\ 0& -\delta \ep
$
is $\alpha$-independent, with real, semi-simple eigenvalues of constant
multiplicity whenever $\delta\neq 0$, or, equivalently, $\eta\neq 0$.
Applying Lemma \ref{wellcond2},
we thus find that the associated Cauchy problem of method two
is well-posed with standard hyperbolic estimate \eqref{resestD}.
Thus, again, method two succeeds despite failure of the uniform Lopatinski
condition
observed in Example \ref{neueg} for method one.
\end{example}

\begin{example}\label{noest}
Substituting in Example \ref{neueg2} the value
$ A_1=\bp \theta & 1\\ 1 & \theta \ep$, $\theta\ne 0$,
we find that $\mathcal{A}_1 \mathcal{A}_0^{-1}=
\bp 0 & 1\\ \frac{1}{1+\alpha}& \theta -\delta \ep $,
and so semisimplicity fails, in general, for $\eta\ne 0$
and $|\gamma,\tau|\ne 0$.
Thus, in this case the frozen-coefficient
algorithm proposed for method two fails, even though
by Proposition \ref{sharp} the problem is well-posed.
\end{example}

\subsection{Questions and comparison to first approach}\label{s:2disc}
\emph{\quad} We have described two methods for solving the reduced
hyperbolic outer problem obtained by the derivation of Section
\ref{s:mixed}, which appear to give slightly different
bounds and apply in slightly different scenarios.
It would be interesting to further clarify the relation between the two methods.
It is not clear that our technique of obtaining estimates through
an enlarged family of
frozen-in frequency coefficients $\mathcal{A}_0(\gamma_0,\tau_0,\eta_0)$
will always produce optimum results for the problem of method two
in problems of interest.
On the other hand, for situations with mixed incoming/outgoing modes,
the first method requires that the uniform Lopatinski condition be satisfied
in order to obtain good bounds, a scenario that might not occur even when
the conditions for method two do apply.

%


\begin{thebibliography}{GMWZ6}

\bibitem[BS]{BS}
Benzoni--Gavage, S. and  Serre, D.,
{\it Multidimensional hyperbolic partial differential equations. First-order systems and applications,}
Oxford Mathematical Monographs. The Clarendon Press, Oxford University Press, Oxford, 2007. xxvi+508 pp. ISBN: 978-0-19-921123-4; 0-19-921123-X.

\bibitem[Br]{Br}
Braslow, A.L., {\it A history of suction-type laminar-flow control
with emphasis on flight research}, NSA History Division, Monographs
in aerospace history, number 13 (1999).



\bibitem[CP]{CP}
Chazarain J. and Piriou, A., \emph{Introduction to the Theory of
Linear Partial Differential Equations}, North Holland, Amsterdam,
1982.


\bibitem[Co2]{Co2} Coulombel, J.-F, and Secchi, P.,
\emph{Nonlinear compressible vortex sheets in two space dimensions}, Ann. Scient. de L'E.N.S.,41 (2008), 85-139.


\bibitem[F1]{F1} Fornet, B. {\it  Viscous approach for Linear Hyperbolic Systems with Discontinuous Coefficients,}
Annales faculté des Sciences de Toulouse, Vol. 18/2 (2009), pp 397-443.

\bibitem[F2]{F2} Fornet, B. {\it  The Cauchy Problem for 1D Linear
Nonconservative Hyperbolic Systems with possibly expansive Discontinuity of the coefficient:
a Viscous Approach,} J. Differential Equations, Vol. 245 (2008), pp 2440-2476.


\bibitem[GS]{GS} Gisclon, M. and Serre, D.,
\textit{Conditions aux limites pour un syst\`eme strictement
hyperbolique fournies par le sch\'ema de Godunov}. RAIRO Mod\'el.
Math. Anal. Num\'er. 31 (1997), 359--380.

\bibitem[GMWZ2]{GMWZ2}
Gues, O., Metivier, G., Williams, M., and Zumbrun, K.,
\emph{Multidimensional viscous shocks II: the small viscosity limit},   Comm. Pure Appl. Math., 57, 2004, 141-218.


\bibitem[GMWZ5]{GMWZ5}
Gues, O., Metivier, G., Williams, M., and Zumbrun, K.,
 \emph{Existence and stability of noncharacteristic boundary-layers for the compressible Navier-Stokes and viscous MHD equations},
 Archive for Rational Mechanics and Analysis, 197, (2010), 1-87.



\bibitem[GMWZ6]{GMWZ6}Gues, O., Metivier, G., Williams, M., and Zumbrun, K.,
\emph{Viscous boundary value problems for symmetric systems with
variable multiplicities}, with  G. Metivier, O. Gues, K. Zumbrun,
  J. Differential Equations, 244 (2008), 309-387.



\bibitem[GMWZ8]{GMWZ8}Gues, O., Metivier, G., Williams, M., and Zumbrun, K.,
{\it Uniform stability estimates for constant-coefficient symmetric hyperbolic
boundary value problems,}
 Comm. Partial Differential Equations  32  (2007),  no. 4-6, 579--590.

\bibitem[K]{K}
Kreiss, H.-O., \emph{Initial boundary value problems for hyperbolic systems},
Comm. Pure Appl. Math. 23 (1970), 277-298.






\bibitem[Met4]{Met4} Metivier, G.,
\emph{Small viscosity and boundary layer methods, theory,
stability analysis, and applications,}
Modeling and simulation in Science, Engineering, and Technology,
Birkh\"auser, Boston, 2003.

\bibitem[MZ1]{MZ1}
Metivier, G. and Zumbrun, K., \emph{Large viscous boundary layers
for noncharacteristic nonlinear hyperbolic problems},
Mem. Amer. Math. Soc.  175  (2005),  no. 826, vi+107 pp.

\bibitem[MZ2]{MZ2} Metivier, G. and Zumbrun, K.,
\textit{Hyperbolic Boundary Value Problems for Symmetric Systems with
Variable Multiplicities},  J. Diff. Eq. 211  (2005),  no. 1, 61--134.

\bibitem[MZ3]{MZ3} M\'etivier, G. and Zumbrun, K.,
\textit{Symmetrizers and continuity of stable subspaces for
parabolic--hyperbolic boundary value problems.}
to appear, J. Discrete. Cont. Dyn. Systems (2004).


 \bibitem[NZ1]{NZ1}
Nguyen, T. and Zumbrun, K.,
{\it Long-time stability of large-amplitude noncharacteristic
boundary layers for hyperbolic--parabolic systems},
 J. Math. Pures Appl. (9)  92  (2009),  no. 6, 547--598.

 \bibitem[NZ2]{NZ2}
Nguyen, T. and Zumbrun, K.,
{Long-time stability of multi-dimensional noncharacteristic viscous
boundary layers},
Comm. Math. Phys.  299  (2010),  no. 1, 1--44.

\bibitem[R]{R}  Rao, I.N.,
{\it Stability of noncharacteristic
boundary-layers for the
compressible nonisentropic
Navier-Stokes equations,}
Doctoral Thesis, UNC Chapel Hill (2010).

\bibitem[R2]{R2} Rousset, F.,
{\it Inviscid boundary conditions and stability of viscous boundary
layers,}  (English summary) Asymptot. Anal. 26 (2001), no. 3-4,
285--306.

\bibitem[R3]{R3} Rousset, F.,
{\it  Stability of small amplitude boundary layers for mixed
hyperbolic-parabolic systems,} Trans. Amer. Math. Soc.  355  (2003),
no. 7, 2991--3008.

\bibitem[S]{S}  Schlichting, H.,
{\em Boundary layer theory}, Translated by J. Kestin. 4th ed.
McGraw-Hill Series in Mechanical
 Engineering. McGraw-Hill Book Co., Inc., New York, 1960.

\bibitem [Se1]{Se1}  Serre, D.,
{\it Second order initial boundary-value problems of variational type,}
J. Funct. Anal. 236 (2006), no. 2, 409--446.

\bibitem[Se2]{Se2} Serre, D.,
{\it Sur la stabilit\'e des couches limites de viscosit\'e,}
(French. English, French summary)  [Stability of viscosity boundary layers]
Ann. Inst. Fourier (Grenoble) 51 (2001), no. 1, 109--130.

\bibitem[SZ]{SZ}
 Serre, D. and  Zumbrun, K.,  {\it Boundary layer stability in real
vanishing-viscosity limit}, Comm. Math. Phys. 221 (2001), no. 2,
267--292.

\bibitem[Z]{Z}  Zumbrun, K.
{\it Stability of noncharacteristic boundary layers in
the standing-shock limit},
 Trans. Amer. Math. Soc.  362  (2010),  no. 12, 6397--6424.

\end{thebibliography}
\end{document}